\documentclass[12pt,letterpaper]{amsart}
\usepackage[margin=1.1in]{geometry}
\usepackage[english]{babel}
\usepackage[alphabetic]{amsrefs}

\usepackage{amsmath}
\usepackage{amssymb}
\usepackage{amsthm}
\usepackage{amscd}
\usepackage{bbm}
\usepackage{esvect}
\usepackage{mathrsfs}
\usepackage{hyperref}
\usepackage{cleveref}
\usepackage{pgf,tikz}
\usetikzlibrary{arrows}
\usepackage[T1]{fontenc}

\newcommand{\norm}[1]{\left\lVert#1\right\rVert}
\newcommand{\pref}[1]{(\ref{#1})}
\newcommand{\real}{\text{Re}}

\theoremstyle{definition}
\newtheorem{theorem}{Theorem}[section]
\newtheorem{claim}{Claim}[theorem]
\newtheorem{lemma}[theorem]{Lemma}

\newtheorem{proposition}[theorem]{Proposition}
\newtheorem{corollary}[theorem]{Corollary}
\newtheorem{definition}[theorem]{Definition}
\newtheorem{remark}[theorem]{Remark}

\allowdisplaybreaks

\title{Limits of Probability Measures with General Coefficients}
\author{Andrew Yao}

\address{Massachusetts Institute of Technology, Cambridge, Massachusetts}
\email{andrew.j.yao@gmail.com}

\begin{document}

\begin{abstract} 
We study the convergence of probability measures in terms of moments by applying operators to their Bessel generating functions. We consider a general setting of applying operators such as the Dunkl operator to formal power series that are symmetric or symmetric in all but one variable. Afterwards, we apply the results from this setting by considering Bessel generating functions as the formal power series to obtain a Law of Large Numbers as $N$, the number of variables, increases to infinity and $N^c\beta$ converges to a constant, where $c\in (-\infty, 1)$. In contrast with previous results, we consider when the scaled partial derivatives of the logarithms of the Bessel generating functions evaluated at the origin can have nonzero $N\rightarrow\infty$ limit when any number of variables is involved. Then, the free cumulant of order $k$ is a linear combination of the limits of the order $k$ partial derivatives. 
\end{abstract}

\maketitle

\section{Introduction}
\label{sec:intro}

The goal of this paper is to characterize the convergence of a sequence of probability measures in terms of moments using the coefficients of their Bessel generating functions. In order to reach this goal, we study the applications of sequences of operators to formal power series. In particular, we consider the Dunkl operator introduced in \cite{dunkloperators}, which we define in \Cref{def:operators}. Furthermore, we mainly focus on formal power series over $x_i$, $1\leq i\leq N$ that are symmetric in $N-1$ or $N$ of the $x_i$ for positive integers $N$. Partitions are useful for characterizing such formal power series, see \Cref{subsec:powerseries}.

\subsection{Setup}
\label{subsec:setup}

We first state the setting of the main result. The setting we study is based on the setting of the paper \cite{matrix}, and we state the definitions of LLN-satisfaction and exponentially decaying measures from the paper in \Cref{def:LLN,def:decay}, respectively, with some modifications; in particular, we prove results for a larger class of measures. The notion of LLN-satisfaction that we study is also discussed in \cite{llnclt} in the context of Jack generating functions.

Suppose $\theta\in\mathbb{C}$ and $\real(\theta)\geq 0$. If $\beta=2\theta$, $\beta=1$, $2$, and $4$ correspond to the GUE, GOE, and GSE, respectively. For positive integers $N$, let $\mathcal{M}_N$ denote the set of Borel probability measures over $\mathbb{C}^N$. Given $\theta$ and $\mu\in\mathcal{M}_N$, the Bessel generating function $G_\theta(x_1,\ldots,x_N;\mu)$ is defined in \Cref{def:BGF}.

Suppose $c$ is a real number with $c<1$. For a sequence $\{\mu_N\}_{N\geq 1}$ of probability measures such that $\mu_N\in\mathcal{M}_N$ for $N\geq 1$, we let the random variable $p_k^{N,c}$ be
\[
    p_k^{N,c} \triangleq \frac{1}{N}\sum_{i=1}^N \left(\frac{a_i}{N^{1-c}}\right)^k,
\]
where $(a_1,\ldots,a_N)\sim\mu_N$ for $N, k\geq 1$. The moments $\{m_k\}_{k\geq 1}$ of $\{\mu_N\}_{N\geq 1}$ are given by $m_k=\lim_{N\rightarrow\infty}\mathbb{E}[p_k^{N,c}]$ for $k\geq 1$.

\begin{definition}
\label{def:LLN}
A sequence $\{\mu_N\}_{N\geq 1}$ of probability measures such that $\mu_N\in\mathcal{M}_N$ for $N\geq 1$ \textit{satisfies a Law of Large Numbers} with constant $c$ and moments $\{m_k\}_{k\geq 1}$ if
\[
\lim_{N\rightarrow\infty} \mathbb{E}_{(a_1,\ldots,a_N)\sim\mu_N}\left(\prod_{i=1}^s p_{k_i}^{N,c}\right)  = \prod_{i=1}^s m_{k_i}
\]
for all positive integers $s$ and $k_i, 1\leq i\leq s$. For simplicity, we also equivalently state that the sequence $\{\mu_N\}_{N\geq 1}$ satisfies a $c$-LLN.
\end{definition}

For $v\in\mathbb{C}^N$, we let $|v|$ denotes its magnitude $\sqrt{\sum_{i=1}^N |v_i|^2}$. This notation is used in the following definition of the class of probability measures that we study.

\begin{definition}
\label{def:decay}
A probability measure $\mu$ in $\mathcal{M}_N$ is \textit{exponentially decaying} at rate $R>0$ if
\[\int_{\mathbb{C}^N} e^{R|(a_1,\ldots,a_N)|} \mu(da_1, \ldots, da_N)\]
is finite.
\end{definition}

\begin{remark}
\cite{dunklbound}*{Lemma 4.4} discusses a condition similar to that of the previous definition. The results of this paper are also true for compactly supported generalized functions, see \cite{rectangularmatrix}*{Theorem 2.27}.
\end{remark}

The following lemma showcases a key property of exponentially decaying probability measures. It can be proved using the method of \cite{matrix}*{Lemma 2.9} and \pref{eq:bessel}. Recall that we are assuming that $\real(\theta)\geq 0$. 

\begin{lemma}
\label{lemma:convergence}
Suppose $\mu\in\mathcal{M}_N$ is exponentially decaying at rate $R>0$. Then, $G_\theta(x_1, \ldots,\\ x_N; \mu)$ converges and is holomorphic in the closed ball of radius $R$ centered at the origin.
\end{lemma}

\subsection{Partitions}
\label{subsec:partitions}
Before stating the main result, we introduce some terminology and notation related to partitions. The set of all partitions is denoted by $P$. For $\lambda\in P$, $|\lambda|$ denotes the size of $\lambda$, and $P^+$ is defined as the set of $\lambda\in P$ such that $|\lambda|\geq 1$.

\begin{definition}
\label{def:orderedtuple} Let $N$ be a positive integer. For an ordered $N$-tuple $a=(a_1, a_2, \ldots, a_N)$ of nonnegative integers, let the $\textit{equivalent partition}$ of $a$, denoted by $\pi(a)$, be the partition $\lambda=(\lambda_1\geq\cdots\geq\lambda_m)$ such that the multisets $\{\lambda_1,\ldots,\lambda_m\}$ and $\{a_i|1\leq i\leq N, a_i>0\}$ are equal.
\end{definition}

Suppose $\nu_1, \nu_2, \ldots, \nu_k$ are partitions. Suppose $\nu_i=(a_{i,1}\geq\cdots\geq a_{i,m_i})$ for $1\leq i\leq k$ and let $p=(a_{1,1},\ldots,a_{1,m_1}, a_{2,1}, \ldots, a_{k, m_k})$.
Then, $\nu_1+\cdots+\nu_k$ denotes the partition $\pi(p)$. Moreover, for a partition $\nu=(a_1\geq a_2\geq\cdots\geq a_m)$, define $P(\nu)$ to be the number of distinct permutations of $(a_1, a_2, \ldots, a_m)$. Also, for a finite list $S$ of positive integers with maximum element $M$, suppose that $n_i$ of the elements of $S$ are $i$ for $1\leq i\leq M$. Then, let $\sigma(S)$ be $\pi((n_1, n_2, \ldots, n_M))$.

\subsection{Main result}

First, we state \Cref{lln}, which we use to prove \Cref{thm:equivalence}, the main result of this paper. Note that in \Cref{lln}, $NC(k)$ denotes the set of noncrossing partitions of $[k]$, see \Cref{subsec:importantcoeff} for the definition.

\begin{theorem}
\label{lln}
Suppose $\theta\in\mathbb{C}$ has nonnegative real part and $c$ is a real number such that $c<1$. Also, assume $\{\theta_N\}_{N\geq 1}$ is a sequence of complex numbers with nonnegative real part such that $\lim_{N\rightarrow\infty}N^c\theta_N=\theta$. Let $\{\mu_N\}_{N\geq 1}$ be a sequence of probability measures such that for all $N\geq 1$, $\mu_N$ is in $\mathcal{M}_N$ and is exponentially decaying. Assume that for all $\nu\in P^+$, a complex number $c_\nu$ exists such that
\begin{equation}
\label{eq:llnconditions}
\displaystyle\lim_{N\rightarrow\infty}\frac{1}{N^{1-c}}\cdot\frac{\partial}{\partial x_{i_1}}\cdots\frac{\partial}{\partial x_{i_r}} \ln(G_{\theta_N}(x_1,\ldots,x_N;\mu_N))\bigg|_{x_i=0,1\leq i\leq N} = \frac{|\nu|!c_\nu}{P(\nu)}
\end{equation}
for all positive integers $i_1, \ldots, i_r$ such that $\sigma((i_1,\ldots,i_r))=\nu$. Then, $\{\mu_N\}_{N\geq 1}$ satisfies a $c$-LLN and
\[
m_k = \sum_{\pi\in NC(k)} \prod_{B\in \pi}\theta^{|B|-1}\left(\sum_{\nu\in P, |\nu|=|B|}(-1)^{\ell(\nu)-1} \frac{|\nu|P(\nu)}{\ell(\nu)}c_\nu\right)
\]
for all positive integers $k$.
\end{theorem}

\begin{remark}
In \pref{eq:llnconditions}, $c_\nu$ is the $N\rightarrow\infty$ limit of the coefficient of $\prod_{j=1}^r x_{i_j}$ in $\ln(G_{\theta_N})$ scaled by $\frac{1}{N^{1-c}}$.
\end{remark}

In \Cref{lln} the free cumulant of order $k\geq 1$ is
\[
c_k = \theta^{k-1}\sum_{\nu\in P, |\nu|=k}(-1)^{\ell(\nu)-1} \frac{|\nu|P(\nu)}{\ell(\nu)} c_\nu,
\]
using the definition of free cumulants presented in \cite{freeprobability}. Observe that $c_k$ is a linear combination of the limits of the order $k$ partial derivatives. For the proof of \Cref{lln}, see \Cref{subsec:mainproof}.

In \Cref{sec:measures}, we see that we can view $\ln(G_\theta(x_1,\ldots,x_N;\mu))$ as a symmetric polynomial and therefore as a symmetric formal power series. Also, we often evaluate functions at $x_i=0$, $1\leq i\leq N$; for example, see \pref{eq:llnconditions}. This corresponds to the constant terms of formal power series, and \Cref{lln:generalization} is essential for this approach.

Using \Cref{lln} allows for the proof of the following generalization, see \Cref{sec:coeff}. The result requires an uniformity condition on the coefficients of the logarithms of the Bessel generating functions, similarly to what \cite{llnclt}*{Assumption 2.1} requires for Jack generating functions. 

\begin{theorem}
\label{thm:equivalence}
Suppose $\theta\in\mathbb{C}$ has nonnegative real part and $c$ is a real number such that $c<1$. Let $\{\theta_N\}_{N\geq 1}$ be a sequence of complex numbers with nonnegative real part such that $\lim_{N\rightarrow\infty}N^c\theta_N=\theta$. Let $\{\mu_N\}_{N\geq 1}$ be a sequence of probability measures such that for all $N\geq 1$, $\mu_N$ is in $\mathcal{M}_N$ and is exponentially decaying. For $N\geq 1$ and $\nu\in P^+$, define
\[
c_\nu^N\triangleq\frac{P(\nu)}{|\nu|!N^{1-c}}\cdot\frac{\partial}{\partial x_{i_1}}\cdots\frac{\partial}{\partial x_{i_r}} \ln(G_{\theta_N}(x_1,\ldots,x_N;\mu_N))\bigg|_{x_i=0,1\leq i\leq N}
\]
for any positive integers $i_1, \ldots, i_r \leq N$ such that $\sigma((i_1,\ldots,i_r))=\nu$. By symmetry, any choice of $i_1,\ldots,i_r$ results in the same derivative.

Assume that for all $\nu\in P^+$, $|c_\nu^N|= N^{o_N(1)}$. Then, $\{\mu_N\}_{N\geq 1}$ satisfies a $c$-LLN with free cumulants $\{c_k\}_{k\geq 1}$ if and only if 
\begin{equation}
\label{eq:cumulantlimit}
\lim_{N\rightarrow\infty} \theta^{k-1}\sum_{\nu\in P, |\nu|=k}(-1)^{\ell(\nu)-1} \frac{|\nu|P(\nu)}{\ell(\nu)} c_\nu^N = c_k
\end{equation}
for all $k\geq 1$. Recall that if the free cumulants are $\{c_k\}_{k\geq 1}$, then the moments are $m_k=\sum_{\pi\in NC(k)} \prod_{B\in \pi}c_{|B|}$ for $k\geq 1$.
\end{theorem}

In \Cref{sec:coeff}, we also discuss a generalization of the previous theorem to the regime $|\theta_NN|\rightarrow\infty$, see \Cref{thm:equivalence2}.

\subsection{Related works}

\Cref{lln} generalizes Claim 9.1 of the paper \cite{matrix}. Furthermore, similarly to the proof of Theorem 3.8 of the paper, which corresponds to the case $c=1$, we prove \Cref{lln} by evaluating Bessel generating functions at $x_i=0$, $1\leq i\leq N$ after applying Dunkl operators as well as other operators. Additionally, we consider when the limit of any partial derivative can be nonzero. On the other hand, in Theorem 3.8 and Claim 9.1 of the paper, the limits of partial derivatives with two or more distinct indices are $0$. The arbitrary limits of partial derivatives are a reason why formal power series and in particular \Cref{lln:generalization} are needed to prove \Cref{lln}. 

Additionally, \cites{representations, fluctuations, unitary,llnclt,qmk,gaussianfluctuation,jackgenfunc,perelomov-popov,rectangularmatrix} consider similar results characterized by the limits of partial derivatives of Bessel, Jack, and Schur generating functions. If the partial derivatives involve only one index, the limit can be nonzero, and in the papers \cites{fluctuations, unitary,llnclt,gaussianfluctuation} studying the Central Limit Theorem, if two distinct indices are involved, the limit can be nonzero. However, if three or more distinct indices are involved, the limit must be zero. In this paper, we consider when the limits of partial derivatives involving any number of distinct indices can be nonzero. 

The paper \cite{gaussianfluctuation} shows LLN and CLT results for the case $c=0$ and $\theta=1$ and does so after evaluating the Bessel generating function at a limiting distribution as $N\rightarrow\infty$ rather than at the origin. In particular, the paper evaluates the logarithm of the scaled Bessel generating function
\[
\mathbb{E}_{(a_1,\ldots,a_N)\sim\mu}\left[\frac{B_{(a_1,\ldots,a_N)}(x_1,\ldots,x_N; 1)}{B_{(a_1,\ldots,a_N)}(\chi; 1)}\right]
\]
at $(x_1,\ldots,x_N)=\chi$ for $\chi\in\mathbb{R}^N$; we consider the case when $\chi=(0,\ldots,0)$. However, we believe that the results of this paper are true for general values of $\chi$, since the results of \Cref{sec:formal,sec:main} are true in this setting. 

We introduce the operator $\mathcal{Q}_i^N$ in \Cref{sec:formal} which computes the leading order terms after applying Dunkl operators to formal power series symmetric in all but one variable. The main contribution of this paper is the computation in the section of the leading order terms of the constant term of formal power series after the application of products of the $\mathcal{Q}_i^N$ operators. 

\subsection{Formal Power Series}
\label{subsec:powerseries}

Suppose $x_1,\ldots,x_N$ are variables and $\vec{x}=(x_1,\ldots,x_N)$. A formal power series $F(x_1,\ldots,x_N)$ can be expressed as
\[
F(x_1,\ldots,x_N)=\sum_{\alpha=(\alpha_1,\ldots,\alpha_N)\in\mathbb{Z}_{\geq 0}^N} c_F^\alpha \prod_{i=1}^N x_i^{\alpha_i},
\]
where the $c_F^\alpha\in\mathbb{C}$ are constants. For brevity, we often use the term formal series to refer to a formal power series. 

The most important coefficient of a formal series $F(x_1,\ldots,x_N)$ in this paper is the constant term $[1]F(x_1,\ldots,x_N)=c_F^{(0,\ldots,0)}$. Particularly, we study the asymptotics of the constant terms of formal series after applying sequences of operators in the $N\rightarrow\infty$ limit.

We also consider formal series that are symmetric or symmetric in $N-1$ of the $x_i$. Suppose $N$ is a positive integer. For $\vec{x}=(x_1,\ldots, x_N)$ and $\nu\in P$ such that $\ell(\nu)\leq N$, let
\[
M_{\nu}(\vec{x})=\sum_{\substack{a=(a_1, \ldots, a_N), \\ a_i\in \mathbb{Z}_{\geq 0}, 1\leq i\leq N, \pi(a)=\nu}} \prod_{i=1}^N x_i^{a_i}.
\]

Note that $N$ can be $\infty$, in which case $\vec{x}=(x_i)_{i\geq 1}$. We define the equivalent partitions of infinite sequences of positive integers. 
\begin{definition}
\label{def:equivpart} Let $a=(a_i)_{i\geq 1}$ be a sequence of nonnegative integers such that $a_i=0$ for all $i>M$ for some positive integer $M$. Then, the \textit{equivalent partition} of $a$, denoted by $\pi(a)$, is the equivalent partition of $(a_i)_{1\leq i\leq M}$. 
\end{definition}
Let $Z$ be the set of sequences $(a_i)_{i\geq 1}$ of nonnegative integers such that $a_i=0$ for $i>M$ for some positive integer $M$. For $\vec{x}=(x_i)_{i\geq 1}$ and $\nu\in P$ we have that
\[
M_\nu(\vec{x})=\sum_{\substack{(a_i)_{i\geq 1}\in Z, \\ \pi((a_i)_{i\geq 1})=\nu}}\prod_{i\geq 1, a_i>0}x_i^{a_i}.
\]

For a symmetric formal series $F(x_1,\ldots, x_N)$, let $c_F^\nu\in\mathbb{C}$ be the coefficient of $M_\nu$ in $F$ for all $\nu\in P$ such that $\ell(\nu)\leq N$ so that
\[
F(x_1,\ldots,x_N) = \sum_{\nu\in P, \ell(\nu)\leq N} c_F^\nu M_\nu(x_1, \ldots, x_N),
\]

Suppose $N$ and $i$ are positive integers such that $N\geq i$. Suppose $\vec{x}_i=(x_j)_{\substack{1\leq j\leq N, j\not= i}}$ and let $\mathcal{F}_i^N$ denote the set of formal series $F(x_1,\ldots,x_N)$ which are symmetric in $x_j$ for $1\leq j\leq N, j\not= i$. Suppose $F(x_1,\ldots,x_N)\in \mathcal{F}_i^N$. Let $c_F^{d, \nu}\in\mathbb{C}$ be the coefficient of $x_i^dM_\nu(\vec{x}_i)$ in $F(x_1, \ldots, x_N)$ for $\nu\in P$ such that $\ell(\nu)\leq N-1$ so that
\[
F(x_1,\ldots,x_N) = \sum_{d=0}^\infty\left(\sum_{\nu\in P, \ell(\nu)\leq N-1} c_F^{d,\nu} M_\nu(\vec{x}_{i})\right)x_i^d,
\]
Furthermore, we can define formal series given a sequence of coefficients.

\begin{definition}
\label{def:inf}
Suppose $s=\{c^{d,\nu}\}_{d\geq 0, \nu\in P}$ is a sequence of elements of $\mathbb{C}$ and $i$ is a positive integer. For $N\geq i$ let the formal series $F_i(s)(x_1, \ldots, x_N)$ in $\mathcal{F}_i^N$ be
\[
F_i(s)(x_1,\ldots,x_N) = \sum_{d=0}^\infty\sum_{\nu\in P, \ell(\nu)\leq N-1} c^{d,\nu} M_\nu(x_j, 1\leq j\leq N, j\not=i)x_i^d.
\]
Also, let the formal series $F_i(s)$ over $(x_j)_{j\geq 1}$ be
\[
F_i(s) = \sum_{d=0}^\infty\sum_{\nu\in P} c^{d,\nu} M_\nu(x_j, j\geq 1, j\not=i)x_i^d.
\]
\end{definition}

\begin{remark}
We can view $F_i(s)(x_1,\ldots,x_n)$ as $F_i(s)$ with $x_j=0$ for $j>n$.
\end{remark}

We focus on $F(x_1,\ldots,x_N)$ symmetric in $N-1$ of the $x_i$ in this paper. Moreover, with \Cref{def:inf}, it is possible to consider the limits of sequences of formal series as $N$ increases to infinity. Particularly, the limit of a sequence $\{F_N(x_1,\ldots,x_N)\}_{N\geq i}$ with $F_N(x_1,\ldots,x_N)\in\mathcal{F}_i^N$ for $N\geq i$ can be viewed as $F_i(s)$ for some $s$ if the limit exists. This concept is discussed more in \Cref{sec:formal}.

\subsection{Paper Organization}

The organization of this paper is as follows. In \Cref{sec:formal}, we discuss results on applying sequences of operators to formal series with $N$ variables which are symmetric in $N-1$ variables as $N$ increases to $\infty$. Afterwards, in \Cref{sec:main}, we prove results involved in the proof of \Cref{lln}, and in \Cref{sec:measures}, we prove the theorem. Following this, in \Cref{sec:eigenvalue}, we apply \Cref{lln} to the $\beta$-Hermite ensemble and discuss a generalization of the theorem, see \Cref{cor:scale}. Finally, in \Cref{sec:coeff}, we consider the coefficients resulting from applications of Dunkl operators as polynomials in $\theta$ and $N$ and prove \Cref{thm:equivalence} and discuss a generalization to the regime $|N\theta_N|\rightarrow\infty$.

\textbf{Acknowledgements.} This paper is a continuation of research done in MIT UROP+. I would like to thank my UROP+ mentor Matthew Nicoletti for providing useful guidance and suggesting the problem. Also, I would like to thank Alexei Borodin and Vadim Gorin for giving comments on this paper.

\section{Limits of Operators}
\label{sec:formal}
In this section, we look at operators on formal series in $x_j$, $1\leq j\leq N$ such that there exists $i$, $1\leq i\leq N$ for which the formal series is symmetric for $j\not=i$. Recall that the set of such formal series is $\mathcal{F}_i^N$. The main results in this section are \Cref{value} and \Cref{finalvalue1}, with \Cref{value} being a special case of \Cref{finalvalue1}. 

Furthermore, in this section, $\theta\in\mathbb{C}$ and the sequence $\{\theta_N\}_{N\geq 1}$ of complex numbers satisfies $\lim_{N\rightarrow\infty}\theta_N=\theta$. Note that $\theta$ does not need to have nonnegative real part, which is in contrast to the statements of \Cref{lln} and \Cref{thm:equivalence}.

\subsection{Basic Results} For a partition $\nu=(a_1\geq a_2\geq\cdots\geq a_m)$, let
\[
S(\nu)=\{(p_1, p_2)|p_i=(p_{i,1}, \ldots, p_{i,m}), 1\leq i\leq 2; p_{1,j}+p_{2,j}=a_j, 1\leq j\leq m\}.
\]

For $N$ variables $x_1, \ldots, x_N$, let $s_{i,j}$ be the operator which switches $x_i$ and $x_j$, for distinct $i,j\in[N]$. Note that $s_{i,j}$ is essential for the definition of the Dunkl operator, see \Cref{def:operators}. Furthermore, for $1\leq i\leq N$, let $d_i$ be the operator such that for nonnegative integers $a_k$, $1\leq k\leq N$, 
\begin{equation}
\label{eq:decdeg}
d_i\left(\prod_{k=1}^N x_k^{a_k}\right) = \begin{cases}
x_i^{a_i-1}\displaystyle\prod_{1\leq k\leq N, k\not= i} x_k^{a_k} & \text{if $a_i\geq 1$,} \\
0 & \text{ if $a_i=0$}.
\end{cases}
\end{equation}

Also, for $1\leq i,j\leq N$, $i\not=j$, let $C_{i,j}$ be the operator such that for nonnegative integers $a_k$, $1\leq k\leq N$, 
\begin{equation}
\label{eq:change}
C_{i,j}\left(\prod_{k=1}^N x_k^{a_k}\right) = \begin{cases}
x_i^{a_j-1}\displaystyle\prod_{1\leq k\leq N, k\not= i,j} x_k^{a_k} & \text{if $a_i=0$ and $a_j\geq 1$}, \\
0 & \text{if $a_i\geq 1$ or $a_i=a_j=0$}.
\end{cases}
\end{equation}
With this operator, for positive integers $N\geq i$ and a formal series $f(x_1,\ldots,x_N)\in \mathcal{F}_i^N$, let the operator $\mathcal{Q}_i^N(f(x_1,\ldots,x_N))$ be
\begin{equation}
\label{eq:operator2}
    \mathcal{Q}_i^N(f(x_1,\ldots,x_N)) =\theta_N\sum_{\substack{1\leq j\leq N, \\ j\not= i}} \frac{d_i-C_{i,j}}{N}+f(x_1, \ldots, x_N).
\end{equation}
In \Cref{operatorpoly}, we show that $\mathcal{Q}_i^N(f(x_1,\ldots,x_N))$ is an operator from $\mathcal{F}_i^N$ to $\mathcal{F}_i^N$. 

For $f(x_1,\ldots,x_N)\in\mathcal{F}_i^N$ we sometimes denote $\mathcal{Q}_i^N(f(x_1,\ldots,x_N))$ by $\mathcal{Q}_i^N(f)$, and for $s=\{c^{d,\nu}\}_{d\geq 0,\nu\in P}$ we sometimes denote $\mathcal{Q}_i^N(F_i(s)(x_1,\ldots,x_N))$ by $\mathcal{Q}_i^N(F_i(s))$. Furthermore, for operators $\mathcal{T}_i$, $1\leq i\leq m$ the product
\[\prod_{i=1}^m \mathcal{T}_i \]
denotes the operator $\mathcal{T}_m\circ\mathcal{T}_{m-1}\circ\cdots\circ\mathcal{T}_1$.

\begin{remark}
Under certain conditions, the $\mathcal{Q}_i^N$ operators are asymptotically equivalent to Dunkl operators, see the proof of \Cref{remainder3}. We analyze them rather than Dunkl operators directly for simplicity. It would be interesting if results similar to this section's results are true for Dunkl operators.
\end{remark}

\begin{definition}
\label{def:limitingsequence}
For a positive integer $i$, a sequence of formal series $\{f_N(x_1,\ldots,x_N)\}_{N\geq i}$ is \textit{symmetric outside of $i$} if $f_N(x_1,\ldots,x_N)\in \mathcal{F}_i^N$ for all $N\geq i$ and $\lim_{N\rightarrow\infty} c_{f_N}^{d,\nu}$ exists for all $d\geq 0$ and $\nu\in P$. If $f=\{f_N(x_1,\ldots,x_N)\}_{N\geq i}$ is symmetric outside of $i$, the $\textit{limiting sequence outside of $i$}$ of $f$ is $\{\lim_{N\rightarrow\infty} c_{f_N}^{d,\nu}\}_{d\geq 0, \nu\in P}$, and the \textit{limit outside of $i$} of $f$ is
\[\lim_{N\rightarrow\infty} f_N(x_1,\ldots,x_N) \triangleq F_i\left(\{\lim_{N\rightarrow\infty} c_{f_N}^{d,\nu}\}_{d\geq 0, \nu\in P}\right).
\]
\end{definition}

\begin{proposition}
\label{operatorpoly}
For $f(x_1,\ldots, x_N)$ and $g(x_1,\ldots, x_N)$ in $\mathcal{F}_i^N$, $\mathcal{Q}_i^N(f)g(x_1,\ldots,x_N)$ is in $\mathcal{F}_i^N$, and is
\begin{align*}
& \sum_{d=0}^\infty \sum_{\nu\in P, \ell(\nu)\leq N-1}\left(\theta_N c_g^{d+1, \nu}-\theta_N c_g^{0, \nu + (d+1)} + \sum_{\substack{a+b=d, \\ (p_1, p_2)\in S(\nu)}} c_f^{a,\pi(p_1)}c_g^{b,\pi(p_2)}\right)M_{\nu}(\vec{x}_i)x_i^d \\ 
& +\frac{\theta_N}{N}\cdot\sum_{d=0}^\infty \sum_{\nu\in P, \ell(\nu)\leq N-1}(-c_g^{d+1,\nu}+(\ell(\nu)+1)c_g^{0,\nu+(d+1)})M_\nu(\vec{x}_i)x_i^d. 
\end{align*}
\end{proposition}
\begin{proof}
Let
\begin{align*}
g_1(x_1, \ldots, x_N)=d_ig(x_1, \ldots, x_N) &= \sum_{d=1}^\infty \sum_{\nu\in P,\ell(\nu)\leq N-1} c_{g}^{d,\nu}M_{\nu}(\vec{x}_{i})x_i^{d-1} \\ & =\sum_{d=0}^\infty \sum_{\nu\in P,\ell(\nu)\leq N-1} c_{g}^{d+1,\nu}M_{\nu}(\vec{x}_{i})x_i^d.
\end{align*}

Also, let
\begin{align*}
g_2(x_1, \ldots, x_N)& =\sum_{1\leq j\leq N, j\not= i} C_{i,j} g(x_1, \ldots, x_N) \\ & = \left(\sum_{1\leq j\leq N, j\not= i} C_{i,j}\right)\left(\sum_{\nu\in P, \ell(\nu)\leq N-1} c_{g}^{0, \nu} M_\nu(\vec{x}_i)\right).
\end{align*}
Observe that $g_2(x_1, \ldots, x_N)$ is a formal power series which is symmetric in $\vec{x}_i$. We find the coefficient of $M_\nu(\vec{x}_i)x_i^d$. If $\ell(\nu)\geq N$, the coefficient is $0$ in $g_2(x_1,\ldots,x_N)$. Then, suppose $\ell(\nu)\leq N-1$. For a monomial $p$ in $M_\nu(\vec{x}_i)$ with variables $(\vec{x}_i)_j$, $1\leq j\leq \ell(\nu)$, there are $N-1-\ell(\nu)$ $j$ such that $x_j$ is not in $p$, and these are the $j$ such that $C_{i,j}$ applied to a monomial will give $px_i^d$. For a monomial $q$, if $C_{i,j}q=px_i^d$, then $q=px_j^{d+1}$ and has coefficient $c_{g}^{0, \nu+(d+1)}$ in $g$. Therefore, the coefficient of $px_i^d$, and thus $M_\nu(\vec{x}_i)x_i^d$, in $g_2(x_1, \ldots, x_N)$ is $(N-\ell(\nu)-1)c_g^{0, \nu+(d+1)}$. From this,
\[g_2(x_1, \ldots, x_N)=\sum_{d=0}^\infty\sum_{\nu\in P,\ell(\nu)\leq N-1} (N-\ell(\nu)-1) c_g^{0, \nu+(d+1)}M_\nu(\vec{x}_i)x_i^d.\]

Next, let
\begin{align*}
    g_3(x_1, \ldots, x_N) & = f(x_1, \ldots, x_N)g(x_1, \ldots, x_N) \\
    & = \sum_{d_1, d_2=0}^\infty\sum_{\nu_1, \nu_2\in P} c_f^{d_1, \nu_1}c_g^{d_2, \nu_2}M_{\nu_1}(\vec{x}_i)M_{\nu_2}(\vec{x}_i)x_i^{d_1+d_2}.
\end{align*}
For $d\geq 0$ and $\nu\in P$, we find the coefficient of $M_\nu(\vec{x}_i)x_i^d$ in $g_3(x_1,\ldots,x_N)$. If $\ell(\nu)\geq N$, the coefficient will be $0$. Suppose $\ell(\nu)\leq N-1$, and let
\[
q = \prod_{j=1}^{\ell(\nu)} x_{(\vec{x}_i)_j}^{\nu_j}.
\]
Note that if $q_1$ and $q_2$ are monic monomials such that $q_1q_2=q$, then for $1\leq s\leq 2$, if $b_{s,j}$ is the degree of $x_{(\vec{x}_i)_j}$ in $q_s$ for $1\leq j\leq \ell(\nu)$ and $p(q_s)=(b_{s,1}, \ldots, b_{s, \ell(\nu)})$, $(p(q_1), p(q_2))\in S(\nu)$. We see that if $S$ is the set of $(q_1,q_2)$ such that $q_1$ and $q_2$ are monic monomials with $q_1q_2=q$,
\[p:S\rightarrow S(\nu), (q_1,q_2)\mapsto (p(q_1),p(q_2))\]
is injective as well as surjective, and is therefore a bijection. Also, the coefficient of $q_1x_1^a$ and $q_2x_1^b$ is $c_f^{a, \pi(p(q_1))}$ in $f$ and $c_g^{b, \pi(p(q_2))}$ in $g$, respectively. Then, the coefficient of $qx_i^d$ is
\[
\sum_{\substack{a+b = d, \\ q_1q_2=q}} c_f^{a, \pi(p(q_1))}c_g^{b, \pi(p(q_2))}=\sum_{\substack{a+b=d, \\ (p_1, p_2)\in S(\nu)}} c_f^{a, \pi(p_1)}c_g^{b, \pi(p_2)},
\]
which is also the coefficient of $M_\nu(\vec{x}_i)x_i^d$ in $g_3(x_1,\ldots,x_N)$. From this,
\[
g_3(x_1,\ldots,x_N) = \sum_{d=0}^\infty\sum_{\nu\in P, \ell(\nu)\leq N-1}\left(\sum_{\substack{a+b=d, \\ (p_1,p_2)\in S(\nu)}} c_f^{a,\pi(p_1)}c_g^{b,\pi(p_2)}\right)M_\nu(\vec{x}_i)x_i^d.
\]

Since $g_1(x_1,\ldots, x_N)$, $g_2(x_1,\ldots,x_N)$, and $g_3(x_1,\ldots,x_N)$ are in $\mathcal{F}_i^N$, $\mathcal{Q}_i(f)g(x_1,\ldots,x_N)$ is in $\mathcal{F}_i^N$, with
\[
\mathcal{Q}_i(f)g(x_1,\ldots,x_N) = \frac{(N-1)\theta_N g_1(x_1,\ldots x_N)}{N}-\frac{\theta_N g_2(x_1,\ldots,x_N)}{N}+g_3(x_1,\ldots,x_N).
\]
This completes the proof.
\end{proof}

\begin{corollary}
\label{limitexists}
Let $m\geq 0$ be an integer. Suppose $\{f_{j,N}(x_1,\ldots,x_N)\}_{N\geq i}$ for $1\leq j\leq m$ and $\{g_N(x_1,\ldots,x_N)\}_{N\geq i}$ are symmetric outside of $i$. Then, $\\ \left\{ \left(\prod_{j=1}^m\mathcal{Q}_i^N(f_{j,N})\right) g_N(x_1, \ldots, x_N)\right\}_{N\geq i}$ is symmetric outside of $i$. Also, the limit outside of $i$ of $\mathcal{Q}_i^N(f_{1,N})g_N(x_1,\ldots,x_N)$ as $N\rightarrow\infty$ is
\begin{equation}
\label{eq:limitoperator}
\sum_{d=0}^\infty \sum_{\nu\in P} \lim_{N\rightarrow\infty}\left(\theta c_{g_N}^{d+1, \nu}-\theta c_{g_N}^{0, \nu + (d+1)} + \sum_{\substack{a+b=d, \\ (p_1, p_2)\in S(\nu)}} c_{f_{1,N}}^{a,\pi(p_1)}c_{g_N}^{b,\pi(p_2)}\right)M_\nu(\vec{x}_{i})x_i^d.
\end{equation}
\end{corollary}
\begin{proof}
The $m=0$ case is clear. For $m=1$, from \Cref{operatorpoly}, it is clear that the expression is symmetric outside of $i$. Furthermore, the result implies that the limit of $\mathcal{Q}_i^N(f_{1,N})g_N(x_1,\ldots,x_N)$ as $N\rightarrow\infty$ with respect to $i$ exists and equals \pref{eq:limitoperator}. The rest of \Cref{limitexists} can be proved using induction.
\end{proof}
\begin{proposition}
\label{limit}
Let $i\geq 1$ and $m\geq 0$ be integers. Suppose $\{f_{j,N}(x_1,\ldots,x_N)\}_{N\geq i}$ for $1\leq j\leq m$ and $\{g_N(x_1,\ldots,x_N)\}_{N\geq i}$ are symmetric outside of $i$. Assume that for $1\leq j\leq m$, $f_{j,N}(x_1, \ldots, x_N)$ has limiting sequence $f_j$ and $g_N(x_1, \ldots, x_N)$ has limiting sequence $g$ outside of $i$. Then, outside of $i$,
\[
\lim_{N\rightarrow\infty} \left(\prod_{j=1}^m\mathcal{Q}_i^N(f_{j,N})\right) g_N(x_1, \ldots, x_N) = \lim_{N\rightarrow\infty} \left(\prod_{j=1}^m\mathcal{Q}_i^N(F_i(f_j))\right) F_i(g)(x_1, \ldots, x_N).
\]
\end{proposition}
\begin{proof}
From \Cref{limitexists}, both limits exist. Suppose that for $1\leq j\leq m$, $f_j=\{c_{f_j}^{d,\nu}\}_{d\geq 0, \nu\in P}$, and $g=\{c_g^{d,\nu}\}_{d\geq 0, \nu\in P}$. We use induction on $m$, where the base case $m=0$ is clear. For $m=1$, we can use \pref{eq:limitoperator} to show that $\lim_{N\rightarrow\infty} \mathcal{Q}_i^N(f_{1,N})g_N(x_1,\ldots,x_N)=\lim_{N\rightarrow\infty} \mathcal{Q}_i^N(F(f_1))F(g)(x_1,\ldots,x_N)$. 

Assume the statement is true for $m\geq 1$. We want to prove the statement is true for $m+1$. Let
\[r_N(x_1, \ldots, x_N) = \mathcal{Q}_{i}^N(f_{1,N}) g_N(x_1, \ldots, x_N),\]
and $r$ be the limiting sequence of $\mathcal{Q}_{i}^N(F_i(f_1)) F_i(g)(x_1, \ldots, x_N)$ as $N\rightarrow\infty$ outside of $i$. From $m=1$, $\lim_{N\rightarrow\infty} r_N(x_1, \ldots, x_N) = \lim_{N\rightarrow\infty} \mathcal{Q}_{i}^N(F_i(f_1)) F_i(g)(x_1, \ldots, x_N) = F_i(r)$, or $r$ is the limiting sequence of $r_N(x_1,\ldots,x_N)$ as $N\rightarrow\infty$ outside of $i$. Then, by the inductive hypothesis,
\begin{align*}
& \lim_{N\rightarrow\infty} \left(\prod_{j=1}^{m+1}\mathcal{Q}_i^N(f_{j,N})\right) g_N(x_1, \ldots, x_N) = \lim_{N\rightarrow\infty} \left(\prod_{j=2}^{m+1}\mathcal{Q}_i^N(f_{j,N})\right) r_N(x_1, \ldots, x_N) \\
& = \lim_{N\rightarrow\infty} \left(\prod_{j=2}^{m+1}\mathcal{Q}_i^N(F_i(f_j))\right) F_i(r)(x_1, \ldots, x_N)
\\
& = \lim_{N\rightarrow\infty} \left(\prod_{j=2}^{m+1}\mathcal{Q}_i^N(F_i(f_j))\right) \mathcal{Q}_i^N(F_i(f_1))F_i(g)(x_1, \ldots, x_N) \\
& = \lim_{N\rightarrow\infty} \left(\prod_{j=1}^{m+1}\mathcal{Q}_i^N(F_i(f_j))\right) F_i(g)(x_1, \ldots, x_N). 
\end{align*}
This completes the proof.
\end{proof}
\begin{proposition}
\label{prop:commutativeswitch}
Suppose that $i, j$ are integers such that $1\leq i,j\leq N$ and $i\not= j$. Then, for integers $k$ with $1\leq k \leq N$ and $k\not= i,j$, if $f\in\mathcal{F}_k^N$, $s_{i,j}\mathcal{Q}_k^N(f) = \mathcal{Q}_k^N(f) s_{i,j}$
as operators from $\mathcal{F}_k^N$ to $\mathcal{F}_k^N$. Also, if $f\in\mathcal{F}_i^N$, $s_{i,j}\mathcal{Q}_i^N(f) = \mathcal{Q}_j^N(s_{i,j}f) s_{i,j}$ as operators from $\mathcal{F}_i^N$ to $\mathcal{F}_j^N$.
\end{proposition}
\begin{proof} This follows from expanding the operators and applying the following identities:
\begin{itemize}
\item $s_{i,j}d_k=d_k s_{i,j}$, $s_{i,j}C_{k,\ell}=C_{k,\ell}s_{i,j}$, and $s_{i,j}C_{k,j}=C_{k,i}s_{i,j}$, where $k\not=i,j$ and $\ell\not=i,j,k$.
\item $s_{i,j} d_i = d_j s_{i,j}$, $s_{i,j}C_{i,k}=C_{j,k}s_{i,j}$, and $s_{i,j}C_{i,j}=C_{j,i}s_{i,j}$, where $k\not=i,j$.
\end{itemize}
\end{proof}

\subsection{Constant term}
\label{subsec:importantcoeff}
The most important part of the formal series after applying the $\mathcal{Q}_i^N$ operators is the term with degree $0$, see \Cref{sec:main}. For a sequence of coefficients $\{c^{d,\nu}\}_{d\geq 0, \nu\in P}$, this corresponds to $c^{0,0}$. In this subsection, we show \Cref{value} to compute the degree $0$ term following $\mathcal{Q}_i^N$ operators when there is only one value of $i$. Later on, in \Cref{finalvalue1}, we look at the constant term following $\mathcal{Q}_i^N$ operators when there are any number of distinct values of $i$. These values are computed with free cumulants, which are introduce in \Cref{def:cumulants}. However, we must first introduce noncrossing partitions.

Suppose that $\pi$ of is a partition of a finite, nonempty set $S$ of real numbers. Suppose
\[
\pi = B_1\sqcup B_2\sqcup\cdots\sqcup B_m,
\]
with $B_i$, $1\leq i\leq m$ being the blocks of $\pi$, such that the smallest element of $B_{i+1}$ is greater than the smallest element of $B_i$ for $1\leq i\leq m-1$. Also, the length of $\pi$ is $\ell(\pi)=m$. 

A partition $\pi$ is $\textit{noncrossing}$ if for any distinct blocks $B_1$ and $B_2$ of $\pi$, there do not exist $a,b\in B_1$ and $c,d\in B_2$ such that $a<c<b<d$. The paper \cite{KREWERAS} discusses noncrossing partitions extensively. Let the set of noncrossing partitions of a finite, nonempty set $S$ of real numbers be $NC(S)$, and for $k\geq 1$, let $NC(k)=NC([k])$. A way to represent a partition $\pi$ is the $\textit{circular representation}$, where the elements are spaced around a circle in order and the convex hulls of the elements of each block of $\pi$ are added. If $\pi$ is noncrossing, the convex hulls are disjoint.

\begin{figure}[!h]
\centering
\begin{tikzpicture}[line cap=round,line join=round,>=triangle 45,x=1.1cm,y=1.1cm]
\clip(-2,-1.1) rectangle (10,3.3);
\fill[line width=1.5pt,fill=black,fill opacity=0.3] (-0.43969262078590776,2.493620766483186) -- (0.5,2.8356409098088537) -- (1.,0.) -- (0.,0.) -- cycle;
\fill[line width=1.5pt,fill=black,fill opacity=0.3] (1.4396926207859084,2.4936207664831853) -- (1.939692620785908,1.6275953626987465) -- (1.7660444431189777,0.642787609686539) -- cycle;
\node at (0.5,-0.9) {\normalsize Noncrossing partition};
\draw [line width=1.5pt] (0.5,1.3737387097273115) circle (1.4619022000815423*1.1cm);
\draw [line width=1.5pt] (-0.939692620785908,1.6275953626987478)-- (-0.7660444431189781,0.6427876096865397);
\draw [line width=1.5pt] (-0.43969262078590776,2.493620766483186)-- (0.5,2.8356409098088537);
\draw [line width=1.5pt] (0.5,2.8356409098088537)-- (1.,0.);
\draw [line width=1.5pt] (1.,0.)-- (0.,0.);
\draw [line width=1.5pt] (0.,0.)-- (-0.43969262078590776,2.493620766483186);
\draw [line width=1.5pt] (1.4396926207859084,2.4936207664831853)-- (1.939692620785908,1.6275953626987465);
\draw [line width=1.5pt] (1.939692620785908,1.6275953626987465)-- (1.7660444431189777,0.642787609686539);
\draw [line width=1.5pt] (1.7660444431189777,0.642787609686539)-- (1.4396926207859084,2.4936207664831853);
\begin{scriptsize}
\draw [fill=black] (0.,0.) circle (2.pt);
\draw[color=black] (-0.11563625798620362,-0.3177080076873236) node {1};
\draw [fill=black] (1.,0.) circle (2.pt);
\draw[color=black] (1.1156362579862036,-0.3177080076873239) node {2};
\draw [fill=black] (1.7660444431189777,0.642787609686539) circle (2.pt);
\draw[color=black] (2.058845726811989,0.47373870972731047) node {3};
\draw [fill=black] (1.939692620785908,1.6275953626987465) circle (2.pt);
\draw[color=black] (2.272653955421975,1.6863054295277848) node {4};
\draw [fill=black] (1.4396926207859084,2.4936207664831853) circle (2.pt);
\draw[color=black] (1.657017697435771,2.752618707341471) node {5};
\draw [fill=black] (0.5,2.8356409098088537) circle (2.pt);
\draw[color=black] (0.5,3.173738709727311) node {6};
\draw [fill=black] (-0.43969262078590776,2.493620766483186) circle (2.pt);
\draw[color=black] (-0.6570176974357707,2.7526187073414725) node {7};
\draw [fill=black] (-0.939692620785908,1.6275953626987478) circle (2.pt);
\draw[color=black] (-1.2726539554219745,1.6863054295277862) node {8};
\draw [fill=black] (-0.7660444431189781,0.6427876096865397) circle (2.pt);
\draw[color=black] (-1.0588457268119897,0.4737387097273116) node {9};

\fill[line width=1.5pt,fill=black,fill opacity=0.3] (8.439692620785909,2.4936207664831853) -- (8.766044443118978,0.642787609686539) -- (8.,0.) -- cycle;
\fill[line width=1.5pt,fill=black,fill opacity=0.3] (6.560307379214092,2.493620766483186)-- (6.233955556881022,0.6427876096865397) -- (6.060307379214092,1.6275953626987478) -- cycle;
\node at (7.5,-0.9) {\normalsize Crossing partition};
\draw [line width=1.5pt] (7.5,1.373738709727307) circle (1.4619022000815467*1.1cm);
\draw [line width=1.5pt] (7.,0.)-- (8.939692620785908,1.6275953626987465);
\draw [line width=1.5pt] (6.560307379214092,2.493620766483186)-- (6.233955556881022,0.6427876096865397);
\draw [line width=1.5pt] (6.233955556881022,0.6427876096865397)--(6.060307379214092,1.6275953626987478);
\draw [line width=1.5pt] (6.560307379214092,2.493620766483186)--(6.060307379214092,1.6275953626987478);
\draw [line width=1.5pt] (8.439692620785909,2.4936207664831853)-- (8.766044443118978,0.642787609686539);
\draw [line width=1.5pt] (8.766044443118978,0.642787609686539)-- (8.,0.);
\draw [line width=1.5pt] (8.,0.)-- (8.439692620785909,2.4936207664831853);
\draw [fill=black] (7.,0.) circle (2.pt);
\draw[color=black] (7-0.11563625798620362,-0.3177080076873236) node {1};
\draw [fill=black] (8.,0.) circle (2.pt);
\draw[color=black] (8.1156362579862036,-0.3177080076873239) node {2};
\draw [fill=black] (8.7660444431189777,0.642787609686539) circle (2.pt);
\draw[color=black] (9.058845726811989,0.47373870972731047) node {3};
\draw [fill=black] (8.939692620785908,1.6275953626987465) circle (2.pt);
\draw[color=black] (9.272653955421975,1.6863054295277848) node {4};
\draw [fill=black] (8.4396926207859084,2.4936207664831853) circle (2.pt);
\draw[color=black] (8.657017697435771,2.752618707341471) node {5};
\draw [fill=black] (7.5,2.8356409098088537) circle (2.pt);
\draw[color=black] (7.5,3.173738709727311) node {6};
\draw [fill=black] (7-0.43969262078590776,2.493620766483186) circle (2.pt);
\draw[color=black] (7-0.6570176974357707,2.7526187073414725) node {7};
\draw [fill=black] (7-0.939692620785908,1.6275953626987478) circle (2.pt);
\draw[color=black] (7-1.2726539554219745,1.6863054295277862) node {8};
\draw [fill=black] (7-0.7660444431189781,0.6427876096865397) circle (2.pt);
\draw[color=black] (7-1.0588457268119897,0.4737387097273116) node {9};
\end{scriptsize}
\end{tikzpicture}
\caption{Circular representations of a noncrossing and crossing partition of $\{1,2, \ldots, 9\}.$}
\end{figure}
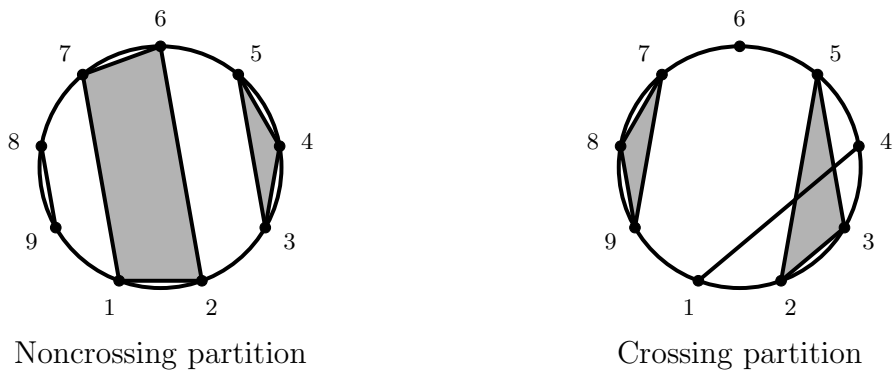

The proof of the following result is similar to \cite{matrix}*{Step 3 of Proof of Theorem 6.2}. The main difference arises from the fact that we only consider noncrossing partitions.

\begin{lemma}
\label{noncrossingsum}
For a positive integer $k$,
\[
    \sum_{\substack{\pi\in NC(k+1), \\ \pi = B_1\sqcup\cdots\sqcup B_m}} b_{|B_1|}\prod_{i=2}^m a_{|B_i|} = \sum_{\substack{\pi\in NC(k), \\ \pi = B_1\sqcup\cdots\sqcup B_m}} \left( b_{|B_1|+1}+\sum_{j=1}^{|B_1|} a_jb_{|B_1|+1-j}\right)\prod_{i=2}^m a_{|B_i|}.
\]
\end{lemma}
\begin{proof}
Suppose that for $k\geq 1$, $NC'(k)=NC(\{1, 3, 4, \ldots, k+1\})$. For all $\pi\in NC'(k)$, define $S(\pi)$ to be the set of partitions $\lambda$ in $NC(k+1)$ such that $\lambda$ is $\pi$ with $2$ added to $B_1$ or $\lambda$ is $\pi$ with $B_1$ replaced by $B_1'$ and $B_1''$, where $B_1'$ is the $j-1$ largest elements of $B_1$ with $1$ added, $B_1''$ is the remaining elements of $B_1$ with $2$ added, and $1\leq j\leq |B_1|$.

\begin{claim}
\label{claim:disjointunion}
The sets $S(\pi)$ are disjoint for $\pi\in NC'(k)$ and the union of these sets is $NC(k+1)$.
\end{claim}

\begin{proof}
First, we must prove that if $\pi\in NC'(k)$, then $S(\pi)\subset NC(k+1)$. It is clear that $\lambda\in S(\pi)$ that is $\pi$ with $2$ added to $B_1$ is in $NC(k+1)$. Suppose $\lambda\in S(\pi)$ is $\pi$ with $B_1$ replaced by $B_1'$ and $B_1''$, where $B_1'$ and $B_1''$ are defined as mentioned previously. The noncrossing condition for $\lambda$ is clearly satisfied when the two distinct blocks are $B_1'$ and $B_1''$. When the two distinct blocks are both not $B_1'$ and $B_1''$, the noncrossing condition is also satisfied since $\pi$ is noncrossing. When one block is $B_1'$ and the other is not $B_1''$, the noncrossing condition is satisfied because $B_1'\subset B_1$ and $\pi$ is noncrossing. On the other hand, when one block is $B_1''$ and the other is not $B_1'$, the noncrossing condition is satisfied because $B_1''$ is a subset of $B_1\cup\{2\}\backslash\{1\}$ and $\pi$ is noncrossing.

For $\lambda\in NC(k+1)$, suppose the decomposition into blocks is $\lambda = C_1\sqcup C_2\sqcup\cdots\sqcup C_m$, such that the smallest element of $C_{i+1}$ is greater than the smallest element of $C_i$ for $1\leq i\leq m-1$. Note that $2\in C_1$ or $2\in C_2$.

If $2\in C_1$, let $\pi=B_1\sqcup\cdots\sqcup B_m\in NC'(k)$ have $B_i=C_i$ for $2\leq i\leq m$ and $B_1=C_1\backslash\{2\}$. We see that $\pi$ is noncrossing and adding $2$ to $B_1$ gives $\lambda$, so $\lambda\in S(\pi)$. Suppose $\lambda\in S(\pi)$, $\pi=B_1\sqcup\cdots\sqcup B_{m'}\in NC'(k)$. Since the first block of $\lambda$ contains a $2$ and the only element of $S(\pi)$ with first block containing $2$ is $\pi$ with $2$ added to $B_1$, we know that $\pi$ must be $\lambda$ with $2$ removed from $C_1$. So, $\pi$ is unique, and $\lambda$ is not in $S(\pi)$ for more than one $\pi\in NC'(k)$.

Suppose that $2\in C_2$. Then, let $\pi=B_1\sqcup\cdots\sqcup B_{m-1}\in NC'(k)$ have $B_1=\left(C_1\cup C_2\right)\backslash\{2\}$ and $B_{i-1}=C_i$ for $3\leq i\leq m$.
Since $\lambda$ is noncrossing, $\pi$ is also noncrossing. The argument for this is as follows. It is clear that the noncrossing condition between $B_i$ and $B_j$ is satisfied when $2\leq i<j\leq m-1$. It suffices to prove that the noncrossing condition between $B_1=(C_1\cup C_2)\backslash\{2\}$ and $B_i=C_{i+1}$ is satisfied for $2\leq i\leq m$. If there exists $b\in B_1$ and $c,d\in B_i$ such that $c<b<d$, then $1<c<b<d$ and $2<c<b<d$, which contradictions the noncrossing condition between either $C_1$ and $C_{i+1}$ or $C_2$ and $C_{i+1}$. Hence, $\pi$ is noncrossing.

Note that $1\in C_1$ and $2\in C_2$. Suppose $j_1\in C_1$ and $j_2\in C_2$, where $j_1, j_2>2$. Because $\lambda$ is noncrossing, we must have that $j_1>j_2.$ Therefore, the elements of $C_1$ which are not $1$ are greater than all of the elements of $C_2$. We see that, where $1\leq j=|C_1|\leq |B_1|$, $C_1$ is the largest $j-1$ elements of $B_1$ with $1$ added, and $C_2$ is the remaining elements of $B_1$ with $2$ added. Therefore, $\lambda\in S(\pi)$.

Assume $\lambda\in S(\pi)$, where $\pi=B_1\sqcup\cdots\sqcup B_{m'}\in NC'(k)$. Since $\lambda$ does not have $2$ in the first block, $\lambda$ is $\pi$ with $B_1'$ and $B_1''$ instead of $B_1$, where $B_1'$ is the last $j-1$ elements of $B_1$ with $1$ added, and $B_1''$ is the remaining elements of $B_1$ with $2$ added, for some $1\leq j\leq |B_1|$. Then, the first and second block of $\lambda$ would be $C_1=B_1'$ and $C_2=B_1''$, respectively. But, in $\pi$, $B_1=(B_1'\cup B_1'')/\{2\} = (C_1\cup C_2)/\{2\}$ and $B_i=C_{i+1}$ for $2\leq i\leq m-1$, where $m'=m-1$. Thus, $\pi$ is unique.
\end{proof}

With \Cref{claim:disjointunion},
\begin{align*}
    \sum_{\substack{\pi\in NC(k+1), \\ \pi=B_1\sqcup\cdots\sqcup B_m}} b_{|B_1|}\cdot\prod_{i=2}^m a_{|B_i|} & =\sum_{\pi\in NC'(k)} \sum_{\substack{\lambda\in S(\pi), \\ \lambda=B_1\sqcup\cdots\sqcup B_m}} b_{|B_1|}\cdot\prod_{i=2}^m a_{|B_i|} \\
    & = \sum_{\substack{\pi\in NC'(k), \\ \pi=B_1\sqcup\cdots\sqcup B_m}} \left(b_{|B_1|+1}+\sum_{j=1}^{|B_1|} a_jb_{|B_1|+1-j}\right)\cdot \prod_{i=2}^m a_{|B_i|},
\end{align*}
giving \Cref{noncrossingsum} since $|\{1, 3, 4, \ldots, k+1\}|=|\{1,2,3,\ldots, k\}|=k$. We are done.
\end{proof}

\begin{definition}
\label{def:cumulants}
For a sequence of coefficients $s=\{c^{d,\nu}\}_{d\geq 0, \nu\in P}$ and a positive integer $k$, let the $\textit{free cumulant}$ of order $k$ of $s$ be
\[
    c_k(s) = \theta^{k-1}\sum_{\substack{\nu\in P, d\geq 0,\\ |\nu|+d=k-1}} (-1)^{\ell(\nu)}P(\nu)c^{d, \nu}.
\]
\end{definition}

The free cumulants appear in \pref{eq:limitcumulants} and these free cumulants will be used in \Cref{sec:main} and \Cref{sec:measures} to obtain the free cumulants corresponding to the moments $m_k$ in \Cref{lln} that we mentioned previously. Furthermore, \Cref{thm:equivalence} is based on the asymptotics of the free cumulants.

\begin{theorem}
\label{value}
Let $i$ and $k$ be positive integers. Suppose $\{f_{j,N}(x_1, \ldots, x_N)\}_{N\geq i}$ for $1\leq j\leq k-1$ and $\{g_N(x_1,\ldots, x_N)\}_{N\geq i}$ are sequences of formal series which are symmetric outside of $i$. For $1\leq j\leq k-1$, assume that $f_{j,N}(x_1,\ldots,x_N)$ has limiting sequence $f$ outside of $i$. Also, suppose $g_N(x_1,\ldots, x_N)$ has limiting sequence $g$ outside of $i$. Then, 
\begin{equation}
\label{eq:limitcumulants}
   \lim_{N\rightarrow\infty}\left([1]\prod_{j=1}^{k-1}\mathcal{Q}_i^N(f_{j,N})g_N(x_1, \ldots, x_N)\right) = \sum_{\substack{\pi\in NC(k), \\ \pi = B_1\sqcup\cdots\sqcup B_m}} c_{|B_1|}(g)\prod_{l=2}^m c_{|B_l|}(f).
\end{equation}
\end{theorem}
\begin{proof}
From \Cref{limitexists}, the limit of $\prod_{j=1}^{k-1}\mathcal{Q}_i^N(f_{j,N})g_N(x_1, \ldots, x_N)$ as $N\rightarrow\infty$ outside of $i$ exists, see the definition of being symmetric outside of $i$. Thus, the left hand side of \pref{eq:limitcumulants} is
\[
[1]\lim_{N\rightarrow\infty}\prod_{j=1}^{k-1}\mathcal{Q}_i^N(f_{j,N})g_N(x_1, \ldots, x_N),
\]
To evaluate this, we use induction. For the base case $k=1$, $\lim_{N\rightarrow\infty} [1] g_N(x_1, \ldots, x_N)= c_g^{0,0} = c_1(g)$.

Next, assume that the statement is true for a positve integer $k$. We want to show the statement is true for $k+1$. From \Cref{limit}, outside of $i$,
\[\lim_{N\rightarrow\infty}\prod_{j=1}^k\mathcal{Q}_i^N(f_{j,N})g_N(x_1, \ldots, x_N)=\lim_{N\rightarrow\infty}\mathcal{Q}_i^N(F_i(f))^kF_i(g)(x_1, \ldots, x_N).\]
Using \Cref{limitexists}, let $g'$ be the limiting sequence of $\mathcal{Q}_i^N(F_i(f)) F_i(g)(x_1, \ldots, x_N)$ outside of $i$, with
\[
    c_{g'}^{d, \nu} = \theta c_g^{d+1, \nu}-\theta c_g^{0, \nu + (d+1)}+\sum_{\substack{a+b=d, \\ (p_1, p_2)\in S(\nu)}} c_f^{a, \pi(p_1)}c_g^{b, \pi(p_2)}
\]
for $d\geq 0$ and $\nu\in P$. By \Cref{limit} with 
\[
g_N(x_1,\ldots,x_N)=\mathcal{Q}_i^N(F_i(f)) F_i(g)(x_1, \ldots, x_N),
\]
we have that outside of $i$,
\[
\lim_{N\rightarrow\infty}\mathcal{Q}_i^N(F_i(f))^kF_i(g)(x_1, \ldots, x_N) = \lim_{N\rightarrow\infty}\mathcal{Q}_i^N(F_i(f))^{k-1}F_i(g')(x_1, \ldots, x_N).
\]

By the inductive hypothesis,
\begin{align*}
    & [1]\lim_{N\rightarrow\infty}\mathcal{Q}_i^N(F_i(f))^{k-1}F_i(g')(x_1, \ldots, x_N) = \sum_{\substack{\pi\in NC(k), \\ \pi = B_1\sqcup\cdots\sqcup B_m}} c_{|B_1|}(g')\prod_{l=2}^m c_{|B_l|}(f) \\
    & = \sum_{\substack{\pi\in NC(k), \\ \pi = B_1\sqcup\cdots\sqcup B_m}} \theta^{|B_1|-1}\left(\sum_{\substack{\nu\in P, d\geq 0 \\ |\nu|+d=|B_1|-1}} (-1)^{\ell(\nu)}P(\nu)c_{g'}^{d, \nu}\right)\prod_{l=2}^m c_{|B_l|}(f).
\end{align*}
Also, from \Cref{noncrossingsum},
\begin{align*}
&\sum_{\substack{\pi\in NC(k+1), \\ \pi = B_1\sqcup\cdots\sqcup B_m}} c_{|B_1|}(g)\prod_{l=2}^m c_{|B_l|}(f) \\
& = \sum_{\substack{\pi\in NC(k), \\ \pi = B_1\sqcup\cdots\sqcup B_m}} \left(c_{|B_1|+1}(g)+\sum_{j=1}^{|B_1|} c_j(f)c_{|B_1|+1-j}(g)\right) \prod_{l=2}^m c_{|B_l|}(f).
\end{align*}
Then, to prove that the statement is true for $k+1$, it suffices to show that
\[
 \theta^{|B_1|-1}\sum_{\substack{\nu\in P, d\geq 0 \\ |\nu|+d=|B_1|-1}} (-1)^{\ell(\nu)}P(\nu)c_{g'}^{d, \nu} = c_{|B_1|+1}(g)+\sum_{j=1}^{|B_1|} c_j(f)c_{|B_1|+1-j}(g).
\]
The left hand side is
\begin{align*}
    & \theta^{|B_1|}\sum_{\substack{\nu\in P, d\geq 0 \\ |\nu|+d=|B_1|-1}}(-1)^{\ell(\nu)}P(\nu)c_g^{d+1, \nu}-\theta^{|B_1|}\sum_{\substack{\nu\in P, d\geq 0 \\ |\nu|+d=|B_1|-1}}(-1)^{\ell(\nu)}P(\nu)c_g^{0, \nu + (d+1)} \\
    & + \theta^{|B_1|-1}\sum_{\substack{\nu\in P, d\geq 0 \\ |\nu|+d=|B_1|-1}}\sum_{\substack{a+b=d, \\ (p_1, p_2)\in S(\nu)}} (-1)^{\ell(\nu)}P(\nu)c_f^{a, \pi(p_1)}c_g^{b, \pi(p_2)}.
\end{align*}
However, the right hand side is
\begin{align*}
& \theta^{|B_1|}\sum_{\substack{\nu\in P, d\geq 0,\\ |\nu|+d=|B_1|}} (-1)^{\ell(\nu)}P(\nu)c_g^{d, \nu}\\
& +\theta^{|B_1|-1}\sum_{j=1}^{|B_1|}\left(\sum_{\substack{\nu_1\in P, d_1\geq 0,\\ |\nu_1|+d_1=j-1}} (-1)^{\ell(\nu_1)}P(\nu_1)c_f^{d_1, \nu_1}\right)\left(\sum_{\substack{\nu_2\in P, d_2\geq 0,\\ |\nu_2|+d_2=|B_1|-j}} (-1)^{\ell(\nu_2)}P(\nu_2)c_g^{d_2, \nu_2}\right).
\end{align*}

Observe that
\[
    \sum_{\substack{\nu\in P, d\geq 0,\\ |\nu|+d=|B_1|}} (-1)^{\ell(\nu)}P(\nu)c_g^{d, \nu} 
    = \sum_{\substack{\nu\in P, d\geq 0, \\ |\nu|+d=|B_1|-1}}(-1)^{\ell(\nu)}P(\nu)c_g^{d+1, \nu} + \sum_{\nu\in P, |\nu|=|B_1|} (-1)^{\ell(\nu)}P(\nu)c_g^{0, \nu}.
\]
Here, we want to show that
\[
\sum_{\nu\in P,|\nu|=|B_1|} (-1)^{\ell(\nu)}P(\nu)c_g^{0, \nu} = \sum_{\substack{\nu\in P, d\geq 0, \\ |\nu|+d=|B_1|-1}}(-1)^{\ell(\nu)+1}P(\nu)c_g^{0, \nu + (d+1)}.
\]
We look at the coefficient of $c_g^{0, \nu}$ with $|\nu|=|B_1|\geq 1$. If $\nu=(a_1\geq\cdots\geq a_m)$ and $R(\nu)$ is the set $\{a_i| 1\leq i\leq m\}$, the coefficient on the right hand side is
\[(-1)^{\ell(\nu)}\sum_{i\in R(\nu)} P(\nu_i),\]
with $\nu_i+(i)=\nu$ for all $i\in R(\nu)$. To compute $P(\nu)$, in a permutation of $\nu$, the first integer must be an element $i$ of $R(\nu)$. Then, there are $P(\nu_i)$ permutations of the remaining components of $\nu$, and therefore, for $\nu$ with $|\nu|\geq 1$,
\begin{equation}
\label{eq:partitionsplit}
    P(\nu)=\sum_{i\in R(\nu)} P(\nu_i).
\end{equation}
With this, the coefficients on the left and right hand sides are equal.

In order to complete the proof, it suffices to show that
\begin{equation}
\label{eq:prodseries}
    \begin{split}
        & \sum_{\substack{\nu\in P, d\geq 0, \\ |\nu|+d=|B_1|-1}}\sum_{\substack{a+b=d, \\ (p_1, p_2)\in S(\nu)}} (-1)^{\ell(\nu)}P(\nu)c_f^{a, \pi(p_1)}c_g^{b, \pi(p_2)} \\
        & = \sum_{j=1}^{|B_1|}\left(\sum_{\substack{\nu_1\in P, d_1\geq 0,\\ |\nu_1|+d_1=j-1}} (-1)^{\ell(\nu_1)}P(\nu_1)c_f^{d_1, \nu_1}\right)\left(\sum_{\substack{\nu_2\in P, d_2\geq 0,\\ |\nu_2|+d_2=|B_1|-j}} (-1)^{\ell(\nu_2)}P(\nu_2)c_g^{d_2, \nu_2}\right).
    \end{split}
\end{equation}
We find the coefficient of $c_f^{d_1,\nu_1}c_g^{d_2,\nu_2}$ on both sides, where $d_1+d_2+|\nu_1|+|\nu_2|=|B_1|-1$. On the right hand side, the coefficient is
$(-1)^{\ell(\nu_1)+\ell(\nu_2)}P(\nu_1)P(\nu_2)$. For $\nu$ with $|\nu|=|B_1|-d_1-d_2-1$, let $T(\nu)=\{(p_1, p_2)\in S(\nu) | \pi(p_1)=\nu_1, \pi(p_2)=\nu_2\}$. Then, on the left hand side, the coefficient is, where $d=d_1+d_2$,
\[
    \sum_{\substack{\nu\in P,\\ |\nu|=|B_1|-d-1}} \sum_{\substack{(p_1, p_2)\in S(\nu), \\ \pi(p_1)=\nu_1, \pi(p_2)=\nu_2}} (-1)^{\ell(\nu)}P(\nu) = \sum_{\substack{\nu\in P,\\ |\nu|=|B_1|-d-1}} (-1)^{\ell(\nu)}|T(\nu)| P(\nu)
\]
Suppose $|T(\nu)|\not=0$. Then, if $\ell(\nu)=\ell(\nu_1)+\ell(\nu_2)-k$, $0\leq k\leq \min(\ell(\nu_1), \ell(\nu_2))$. 

\begin{claim}
\label{claim:coeffsum}
For $0\leq k\leq \min(\ell(\nu_1), \ell(\nu_2))$,
\[
\sum_{\substack{\nu\in P, |\nu|=|B_1|-d-1, \\ \ell(\nu)=\ell(\nu_1)+\ell(\nu_2)-k}} |T(\nu)| P(\nu)=\frac{\binom{\ell(\nu_1)}{k}\binom{\ell(\nu_2)}{k}}{\binom{\ell(\nu_1)+\ell(\nu_2)}{k}}\cdot\frac{(\ell(\nu_1)+\ell(\nu_2))!}{\ell(\nu_1)!\ell(\nu_2)!}P(\nu_1)P(\nu_2).
\]
\end{claim}
\begin{proof}
Let $l=\ell(\nu_1)+\ell(\nu_2)-k$ and $\vec{x}=(x_1,\ldots,x_l)$. Suppose partition $\nu$ has $|\nu|=|B_1|-d-1$ and $\ell(\nu)=l$. We see that the coefficient of any term of $M_\nu(\vec{x})$ in $M_{\nu_1}(\vec{x})M_{\nu_2}(\vec{x})$ is $|T(\nu)|$. Therefore, the sum of the coefficients in $M_{\nu_1}(\vec{x})M_{\nu_2}(\vec{x})$ of terms of $M_\nu(\vec{x})$ is $|T(\nu)| P(\nu)$, where there are $P(\nu)$ terms of $M_\nu(\vec{x})$. 

Suppose that $S$ is the sum of the terms of $M_{\nu_1}(\vec{x})M_{\nu_2}(\vec{x})$ which contain all of $x_1,\ldots,x_l$. If a term $p$ of $S$ contains all of the $x_i$, $p$ without its coefficient must be a term of $M_\nu(\vec{x})$ for some $\nu\in P$ with $|\nu|=|B_1|-d-1$ and $\ell(\nu)=l$. But, the sum of the coefficients in $M_{\nu_1}(\vec{x})M_{\nu_2}(\vec{x})$ of terms of $M_\nu(\vec{x})$ is $|T(\nu)| P(\nu)$. Therefore, the sum of the coefficients of $S$ is
\[
C=\sum_{\substack{\nu\in P,|\nu|=|B_1|-d-1, \\ \ell(\nu)=\ell(\nu_1)+\ell(\nu_2)-k}} |T(\nu)|P(\nu).
\]
However, for terms $p_i$ in $M_{\nu_i}(\vec{x})$, $1\leq i\leq 2$, if $p_1p_2$ contains all of the $x_i$, then exactly $k$ of the $x_i$ must be in both $p_1$ and $p_2$. Note that $S$ is the sum of $p_1p_2$ for such $(p_1,p_2)$, and $C$ is the number of $(p_1,p_2)$; we count $C$. There are
\[
\frac{(\ell(\nu_1)+\ell(\nu_2)-k)!}{k!(\ell(\nu_1)-k)!(\ell(\nu_2)-k)!}=\frac{\binom{\ell(\nu_1)}{k}\binom{\ell(\nu_2)}{k}}{\binom{\ell(\nu_1)+\ell(\nu_2)}{k}}\cdot\frac{(\ell(\nu_1)+\ell(\nu_2))!}{\ell(\nu_1)!\ell(\nu_2)!}
\]
ways to choose the $x_i$ for $p_1$ and $p_2$. For each choice of the $x_i$, there are $P(\nu_1)$ and $P(\nu_2)$ choices for $p_1$ and $p_2$, respectively. Thus,
\[
\sum_{\substack{\nu\in P, |\nu|=|B_1|-d-1, \\ \ell(\nu)=\ell(\nu_1)+\ell(\nu_2)-k}} |T(\nu)| P(\nu)=C=\frac{\binom{\ell(\nu_1)}{k}\binom{\ell(\nu_2)}{k}}{\binom{\ell(\nu_1)+\ell(\nu_2)}{k}}\cdot\frac{(\ell(\nu_1)+\ell(\nu_2))!}{\ell(\nu_1)!\ell(\nu_2)!}\cdot P(\nu_1)P(\nu_2).
\]
\end{proof}

\begin{claim}
\label{binomsum}
For all nonnegative integers $a$, $b$, and $m$ such that $m\leq a$,
\[
\sum_{k=m}^a\frac{\binom{a}{k}\binom{b}{k-m}}{\binom{a+b}{k}}(-1)^k=\frac{(-1)^ma!b!}{(a+b)!}.
\]
\end{claim}
\begin{proof}
Fix the value of $a$. We use induction on $m$, from $m=a$ to $0$. The $m=a$ base case is clear. Suppose the statement holds for $m+1$, $0\leq m< a$. We show the statement holds for $m$ with induction on $b$, where the base case $b=0$ is clear. For the inductive step, assume the statement for $m$, $a$, and $b\geq 0$ holds. We want to show the statement for $m$, $a$, and $b+1$ holds. Note that
\begin{align*}
& \sum_{k=m}^a\frac{\binom{a}{k}\binom{b+1}{k-m}}{\binom{a+b+1}{k}}(-1)^k = \sum_{k=m}^a\frac{\binom{a}{k}\binom{b}{k-m}}{\binom{a+b}{k}}\cdot\frac{(b+1)(a+b+1-k)}{(a+b+1)(b+1-k+m)}\cdot (-1)^k \\
& = \frac{b+1}{a+b+1}\left(\sum_{k=m}^a\frac{\binom{a}{k}\binom{b}{k-m}}{\binom{a+b}{k}}\left(1+\frac{a-m}{b+1-k+m}\right)(-1)^k\right) \\
& = \frac{(-1)^ma!(b+1)!}{(a+b+1)!}+\sum_{k=m}^a\frac{\binom{a}{k}\binom{b}{k-m}}{\binom{a+b}{k}}\cdot\frac{(a-m)(-1)^k}{b+1-k+m},
\\
& = \frac{(-1)^ma!(b+1)!}{(a+b+1)!}+\frac{a-m}{b+1}\cdot \sum_{k=m}^a\frac{\binom{a}{k}\binom{b+1}{k-m}}{\binom{a+b}{k}}(-1)^k \\
& = \frac{(-1)^ma!(b+1)!}{(a+b+1)!} +\frac{a-m}{b+1}\cdot\left(\sum_{k=m}^a\frac{\binom{a}{k}\binom{b}{k-m}}{\binom{a+b}{k}}(-1)^k + \sum_{k=m+1}^a\frac{\binom{a}{k}\binom{b}{k-m-1}}{\binom{a+b}{k}}(-1)^k\right) \\
& = \frac{(-1)^ma!(b+1)!}{(a+b+1)!}
\end{align*}
after applying the inductive hypotheses, completing the proof.
\end{proof}

Afterwards, the coefficient of $c_f^{d_1,\nu_1}c_g^{d_2,\nu_2}$ on the left hand side of \pref{eq:prodseries} is, with \Cref{claim:coeffsum} and \Cref{binomsum},
\begin{align*}
    & \sum_{\substack{\nu\in P, \\ |\nu|=|B_1|-d-1}}(-1)^{\ell(\nu)}|T(\nu)|P(\nu) \\ 
    & = \sum_{k=0}^{\min(\ell(\nu_1), \ell(\nu_2))}(-1)^{\ell(\nu_1)+\ell(\nu_2)-k}\sum_{\substack{\nu\in P, |\nu|=|B_1|-d-1, \\ \ell(\nu)=\ell(\nu_1)+\ell(\nu_2)-k}} |T(\nu)| P(\nu)
    = (-1)^{\ell(\nu_1)+\ell(\nu_2)}P(\nu_1)P(\nu_2).
\end{align*}
We are done.
\end{proof}

\subsection{Distinct Indices}

Also, we look at the constant term after using combinations of $\mathcal{Q}_i^j$ with distinct indices $i$ in \Cref{finalvalue1}. This is important for the proof of \Cref{lln}, as will be seen in \Cref{sec:main} and \Cref{sec:measures}, and for this, we use the operator in \pref{eq:multoperators}. But, it is necessary to consider the free cumulants of $s=\{c^{d,\nu}\}_{d\geq 0, \nu\in P}$ which are symmetric in the $x_i$ beforehand. 

\begin{proposition}
\label{symmetriccumulants}
Suppose $i$ is a positive integer and $s = \{c^{d,\nu}\}_{d\geq 0, \nu\in P}$ is a sequence such that $F_i(s)$ is symmetric with respect to $(x_j)_{j\geq 1}$. Then, $c_k(s)=0$ for all integers $k\geq 2$.
\begin{proof} Since $F_i(s)$ is symmetric, for all $\lambda\in P$, there exists a constant $c^{\lambda}$ such that $c^{d,\nu}=c^{\lambda}$ for all $d, \nu$ such that $\nu+(d)=\lambda$. The free cumulant of $s$ of order $k$ is
\[
c_k(s) = \theta^{k-1}\sum_{\substack{\nu\in P, d\geq 0,\\ |\nu|+d=k-1}} (-1)^{\ell(\nu)}P(\nu)c^{d, \nu}
 = \theta^{k-1}\sum_{\lambda\in P, |\lambda| = k-1}c^\lambda\sum_{\substack{\nu\in P, d\geq 0,\\ \nu+(d)=\lambda}} (-1)^{\ell(\nu)}P(\nu),
\]
and $c_k(s)=0$ for $k\geq 2$ follows from \pref{eq:partitionsplit}.
\end{proof}
\end{proposition}

Suppose we have positive integers $i\leq j$, and $k$. For $F(x_1,\ldots,x_j), G(x_1,\ldots,x_j)\in\mathcal{F}_i^j$, let
\begin{equation}
\label{eq:multoperators}
\mathcal{R}_{i,j}^k(F(x_1,\ldots,x_j)) G(x_1,\ldots,x_j) \triangleq \left(\mathcal{Q}_i^j\left(F(x_1,\ldots,x_j)\right)^k G(x_1,\ldots,x_j)\right)\bigg|_{x_i=0}.
\end{equation}
By \Cref{limitexists}, the resulting formal series has variables $x_1,\ldots,x_{i-1},x_{i+1},\ldots,x_j$ and is symmetric.

\begin{lemma}
\label{symmetric}
Suppose $\lambda$ is a partition of length $m\geq 1$, $N\geq m$ is a positive integer, $f_i\in\mathcal{F}_{N-i+1}^{N-i+1}$ for $1\leq i\leq m$, and $g\in\mathcal{F}_N^N$. Then,
\[
\left(\prod_{i=1}^m \mathcal{R}_{N-i+1,N-i+1}^{\lambda_i}(f_i)\right)(g)
\]
is a symmetric formal series in $x_1, \ldots, x_{N-m}$.
\end{lemma}

\begin{proof}
This follows from induction on $m$ from $m=1$ to $N$.
\end{proof}

\begin{lemma}
\label{limit2}
Suppose $\lambda$ is a partition of length $m\geq 1$. Let $\{f_{i,N}\}_{N\geq m}$ for $1\leq i\leq m$ and $\{g_N(x_1,\ldots,x_N)\}_{N\geq m}$ be sequences of formal series such that $f_{i,N}\in\mathcal{F}_{N-i+1}^{N-i+1}$ and $g_N\in\mathcal{F}_N^N$ for $N\geq m$. Also, suppose $\{s_{1,N-i+1}f_{i,N}\}_{N\geq m}$ for $1\leq i\leq m$ and $\{s_{1,N}g_N\}_{N\geq m}$, where $s_{1,N-i+1}f_{i,N}, s_{1,N}g_N\in\mathcal{F}_1^N$, are symmetric outside of $1$. For $N\geq m+1$, consider $\left(\prod_{i=1}^m \mathcal{R}_{N-i+1,N-i+1}^{\lambda_i}(f_{i,N})\right)(g_N)$ as a formal series in $x_1,\ldots,x_{N-m}$. For $k\geq 1$,
$\{\left(\prod_{i=1}^m \mathcal{R}_{N-i+1,N-i+1}^{\lambda_i}(f_{i,N})\right)(g_N)\}_{N\geq m+k}$ is symmetric outside of $k$.
\end{lemma}

\begin{proof}
Induction on $m$ is used. For the base case $m=1$, $\lim_{N\rightarrow\infty}\mathcal{Q}_1^N(s_{1,N}f_{1,N})^{\lambda_1}(s_{1,N}g_N)$
exists outside of $1$ from \Cref{limitexists} because $s_{1,N}f_{1,N}$ and $s_{1,N}g_N$ have limits outside of $1$. For some $i\geq 1$, consider $N\geq i+1$. From \Cref{symmetric}, $\mathcal{R}_{N,N}^{\lambda_1}(f_{1,N})(g_N)\in\mathcal{F}_i^{N-1}$. 

From \Cref{prop:commutativeswitch}, $s_{1,N}\mathcal{Q}_N^N(f_{1,N})^{\lambda_1}(g_N)= \mathcal{Q}_1^N(s_{1,N}f_{1,N})^{\lambda_1}(s_{1,N}g_N)$. With this, the coefficient $c^{d, \nu}$ of $\mathcal{Q}_N^N(f_{1,N})^{\lambda_1}(g_N)|_{x_N=0}$ as a formal series in $x_1,\ldots,x_{N-1}$ outside of $i$ will be the coefficient $c^{0,\nu+(d)}$ of $\mathcal{Q}_N^N(f_{1,N})^{\lambda_1}(g_N)$ outside of $N$, and thus the coefficient $c^{0,\nu+(d)}$ of $\mathcal{Q}_1^N(s_{1,N}f_{1,N})^{\lambda_1}(s_{1,N}g_N)$ outside of $1$. Since the limit of $\mathcal{Q}_1^N(s_{1,N}f_{1,N})^{\lambda_1}(s_{1,N}g_N)$ outside of $1$ exists, the limit of this coefficient from $N=i+1$ to $\infty$ exists. As this holds for all $d, \nu$, the base case is complete. 

Next, assume that the result holds for $m\geq 1$. We want to show that the result holds for $m+1$. For $N\geq m+1$, let
\[
h_{N-m}(x_1,\ldots,x_{N-m})=\left(\prod_{i=1}^m \mathcal{R}_{N-i+1,N-i+1}^{\lambda_i}(f_{i,N})\right)(g_N).
\]
By the inductive hypothesis and because $h_{N-m}$ is symmetric by \Cref{symmetric}, the limit of $s_{1,N-m}h_{N-m}=h_{N-m}$ from $N-m=1$ to $\infty$ exists outside of $1$. Then,
\[
\left(\prod_{i=1}^{m+1} \mathcal{R}_{N-i+1,N-i+1}^{\lambda_i}(f_{i,N})\right)(g_N)=\mathcal{R}_{N-m, N-m}^{\lambda_{m+1}}(f_{m+1,N})(h_{N-m}).
\]
is the base case $m=1$ with $f_{m+1, N}$ as $f_{1,N-m}$, $h_{N-m}$ as $g_{N-m}$, and $N-m$ as $N$, where the conditions are satisfied. Therefore, from $m=1$, the limit from $N-m=i+1$ to $\infty$ of the above expression outside of $i$ exists for $i\geq 1$. This limit is from $N=i+m+1$ to $\infty$, which completes the proof.
\end{proof}

\begin{theorem}
\label{finalvalue1}
Suppose $\lambda$ is a partition of length $m\geq 1$. Let $\{f_{i,N}(x_1,\ldots,x_{N-i+1})\}_{N\geq m}$ for $1\leq i\leq m$ and $\{g_N(x_1,\ldots,x_N)\}_{N\geq m}$ be sequences of formal series such that $f_{i,N}\in\mathcal{F}_{N-i+1}^{N-i+1}$ and $g_N\in\mathcal{F}_N^N$ for $N\geq m$. Moreover, for $1\leq i\leq m$, assume that $\{s_{1,N-i+1}f_{i,N}\}_{N\geq m}$, with $s_{1, N-i+1}f_{i,N}\in\mathcal{F}_1^{N-i+1}$, is symmetric outside of $1$ with limiting sequence $f_i$. Also, assume that $\{s_{1,N}g_N\}_{N\geq m}$, with $s_{1,N}g_N\in\mathcal{F}_1^N$, is symmetric outside of $1$ with limiting sequence $g$. Then,
\begin{align*}
& \lim_{N\rightarrow\infty}[1]\left(\prod_{i=1}^m \mathcal{R}_{N-i+1,N-i+1}^{\lambda_i}(f_{i,N})\right)(g_N)\\
& = \left(\sum_{\substack{\pi\in NC(\lambda_1+1), \\ \pi = B_1\sqcup\cdots\sqcup B_{\ell(\pi)}}}c_{|B_1|}(g)\prod_{j=2}^{\ell(\pi)} c_{|B_j|}(f_1)\right) \prod_{i=2}^m\left(\sum_{\substack{\pi\in NC(\lambda_i), \\ \pi = B_1\sqcup\cdots\sqcup B_{\ell(\pi)}}}\prod_{j=1}^{\ell(\pi)} c_{|B_j|}(f_i)\right).
\end{align*}
\end{theorem}
\begin{proof} We prove this with induction on $m$. For the base case $m=1$, use \Cref{value}. Suppose the result holds for $m\geq 1$. We want to show it holds for $m+1$.

For $N\geq m+1$, let the symmetric formal series $h_{N-m}(x_1,\ldots,x_{N-m})$ be
\[
h_{N-m}(x_1,\ldots,x_{N-m})=\left(\prod_{i=1}^m\mathcal{R}_{N-i+1,N-i+1}^{\lambda_i}(f_{i,N})\right)(g_N).
\]
Then, $s_{1,N-m}\mathcal{Q}_{N-m}^{N-m}\left(f_{m+1,N}\right)^{\lambda_{m+1}}h_{N-m}$ is $\mathcal{Q}_1^{N-m}\left(s_{1,N-m}f_{m+1,N}\right)^{\lambda_{m+1}}h_{N-m}$ from \Cref{prop:commutativeswitch}. Since the switch does not change the constant,
\begin{align*}
& [1]\mathcal{R}_{N-m,N-m}^{N, \lambda_{m+1}}(f_{m+1,N})\left(\prod_{i=1}^m \mathcal{R}_{N-i+1,N-i+1}^{\lambda_i}(f_{i,N})\right)(g_N) = [1]\mathcal{Q}_{N-m}^{N-m}\left(f_{m+1,N}\right)^{\lambda_{m+1}}h_{N-m} \\& = [1]\mathcal{Q}_1^{N-m}\left(s_{1,N-m}f_{m+1,N}\right)^{\lambda_{m+1}}h_{N-m}.
\end{align*}
The limit of this as $N\rightarrow\infty$ is computed. With $f_{N-m}=s_{1,N-m}f_{m+1,N}$ for $N\geq m+1$, $f_N$ has $N\rightarrow\infty$ limit $f_{m+1}$ outside of $1$, and from \Cref{limit2}, $h_N$ has a limit outside of $1$; let $h$ be the limiting sequence of $\{h_N\}_{N\geq 1}$. Then, from \Cref{value} with $N-m$ for $N$,
\[
\lim_{N\rightarrow\infty} [1]\mathcal{Q}_1^{N-m}(f_{N-m})^{\lambda_{m+1}} h_{N-m}= \sum_{\substack{\pi\in NC(\lambda_{m+1}+1), \\ \pi = B_1\sqcup\cdots\sqcup B_{\ell(\pi)}}} c_{|B_1|}(h)\prod_{j=2}^{\ell(\pi)} c_{|B_j|}(f_{m+1}).
\]
But, as $h_N$ is symmetric in $x_i$, $1\leq i\leq N$ for $N\geq 1$, $h$ is symmetric. Then, by \Cref{symmetriccumulants}, $c_k(h)=0$, $k\geq 2$, and $c_1(h)=c_h^{0,0}$. From this,
\begin{align*}
& \lim_{N\rightarrow\infty} [1]\mathcal{Q}_1^{N-m}(f_{N-m})^{\lambda_{m+1}} h_{N-m} = c_h^{0,0}\left(\sum_{\substack{\pi\in NC(\lambda_{m+1}),\\ \pi = B_1\sqcup\cdots\sqcup B_{\ell(\pi)}}} \prod_{j=1}^{\ell(\pi)} c_{|B_j|}(f_{m+1})\right) \\
& = \left(\sum_{\substack{\pi\in NC(\lambda_1+1), \\ \pi = B_1\sqcup\cdots\sqcup B_{\ell(\pi)}}}c_{|B_1|}(g)\prod_{j=2}^{\ell(\pi)} c_{|B_j|}(f_1)\right) \prod_{i=2}^{m+1}\left(\sum_{\substack{\pi\in NC(\lambda_i), \\ \pi = B_1\sqcup\cdots\sqcup B_{\ell(\pi)}}} \prod_{j=1}^{\ell(\pi)} c_{|B_j|}(f_i)\right)
\end{align*}
using the inductive hypothesis, as needed.
\end{proof}

\section{Sequences of Operators}
\label{sec:main}

\subsection{Setup}
\begin{definition}
\label{def:operators}
For $\theta\in\mathbb{C}$, the \textit{Dunkl operators} are, for $1\leq i\leq N$,
\[
\mathcal{D}^\theta_i:=\frac{\partial}{\partial x_i}+\theta\sum_{\substack{1\leq j\leq N, \\ j\not=i}}\frac{1}{x_i-x_j}(1-s_{i,j}).
\]
\end{definition}

For positive integers $k$, let
\[
\mathcal{P}^\theta_k=\sum_{i=1}^N(\mathcal{D}_i^\theta)^k,
\]
and for $1\leq i\leq N$, denote $\frac{\partial}{\partial x_i}$ by $\partial_i$. From \cite{dunkloperators}*{Theorem 1.9}, the Dunkl operators are commutative, and because of this the $\mathcal{P}^\theta_k$ are commutative as well.

Suppose indices $\mathfrak{r}=\{i_j\}_{1\leq j\leq k}$, a positive integer $N\geq\max(\mathfrak{r})$, and a constant $\theta$ are given. For a symmetric formal series $F(x_1,\ldots,x_N)$, let
\begin{equation}
\label{eq:sequenceoperators}
\mathcal{D}^\theta_\mathfrak{r}(F(x_1,\ldots,x_N)) = \left(\prod_{j=1}^k\left(\mathcal{D}^\theta_{i_j}+\frac{\partial}{\partial x_{i_j}} F(x_1,\ldots,x_N)\right)\right)(1).
\end{equation}
For $\lambda\in P$, suppose the set $I_N(\lambda)$ consists of indices $l$ of length $|\lambda|$ such that there exists $i_j\in [N]$ for $1\leq j\leq \ell(\lambda)$ such that in $l$, the first $\lambda_1$ indices are $i_1$, the next $\lambda_2$ indices are $i_2$, and so forth until the last $\lambda_{\ell(\lambda)}$ indices are $i_{\ell(\lambda)}$. 

The main result of \Cref{sec:main} is \Cref{lln:generalization} below. This is used to prove \Cref{lln}, see \Cref{subsec:mainproof}. Note that \Cref{lln} is a generalization of Claim 9.1 of \cite{matrix} and some techniques from Section 5 of that paper are used in the proof of \Cref{lln:generalization} in this section.

\begin{theorem}
\label{lln:generalization} 
Suppose $\theta\in\mathbb{C}$ and $c$ is a real number with $c<1$. Let $\{\theta_N\}_{N\geq 1}$ be a sequence of complex numbers such that $\lim_{N\rightarrow\infty}N^c\theta_N=\theta$. Suppose $\{F_N(x_1,\ldots,x_N)\}_{N\geq 1}$ is a sequence of symmetric formal series. Assume that for all $\nu\in P^+$, a complex number $c_\nu$ exists such that
\[
    \displaystyle\lim_{N\rightarrow\infty}\frac{1}{N^{1-c}}[1]\frac{\partial}{\partial x_{i_1}}\cdots\frac{\partial}{\partial x_{i_r}} F_N(x_1,\ldots,x_N) = \frac{|\nu|!c_\nu}{P(\nu)}
\]
for all positive integers $i_1, \ldots, i_r$ such that $\sigma((i_1,\ldots,i_r))=\nu$. Then, for all $\lambda\in P^+$,

\begin{align*}
    &\lim_{N\rightarrow\infty} \left(\frac{1}{N^{\ell(\lambda)+|\lambda|(1-c)}}\sum_{l\in I_N(\lambda)}[1]\mathcal{D}^{\theta_N}_l(F_N(x_1,\ldots,x_N))\right) \\
    & =  \prod_{i=1}^{\ell(\lambda)}\left(\sum_{\pi\in NC(\lambda_i)} \prod_{B\in\pi}\theta^{|B|-1}\left(\sum_{\nu\in P, |\nu|=|B|}(-1)^{\ell(\nu)-1} \frac{|\nu|P(\nu)}{\ell(\nu)}c_{\nu}\right)\right).
\end{align*}
\end{theorem}

Note that for $\nu\in P^+$, if $i_1, \ldots, i_r$ are positive integers such that $\sigma((i_1,\ldots,i_r))=\nu$, then 
\[
[1]\frac{\partial}{\partial x_{i_1}}\cdots\frac{\partial}{\partial x_{i_r}} F_N(x_1,\ldots,x_N) = \frac{|\nu|! c_{F_N}^\nu}{P(\nu)}.
\]
Hence, the condition in the theorem is that $\lim_{N\rightarrow\infty} \frac{c_{F_N}^\nu}{N^{1-c}} = c_\nu$ for all $\nu\in P^+$.

\begin{remark}
In \Cref{lln:generalization}, the constant term of $F_N$ can be any value. This is because in \pref{eq:sequenceoperators}, the constant term does not impact the output of the operator. Furthermore, in contrast to the results for Bessel generating functions, $\theta$ can have negative real part.
\end{remark}

\subsection{Sequences of operators} Suppose $N$ is a positive integer. Given the sequence of variables $\{c_F^\nu\}_{\nu\in P^+}$, we define the symmetric polynomial $F$ as
\[
F(x_1,\ldots,x_N)=\sum_{\nu\in P^+, \ell(\nu)\leq N} c_F^\nu M_\nu(x_1,\ldots,x_N),
\]
so that the $c_F^\nu$ are the coefficients of $F(x_1,\ldots,x_N)$. It is clear that $\mathcal{D}_{\mathfrak{r}}(F(x_1,\ldots,x_N))$ is a polynomial in the $x_i$ with coefficients that are polynomials in the $c_F^\nu$, $\nu\in P^+$. 

When considering a polynomial in the $c_F^\nu$, $\nu\in P^+$, the degree of $c_F^\nu$ is $|\nu|$. Then, the degree of the product $\prod_{i=1}^m c_F^{\nu_i}$ is $\sum_{i=1}^m |\nu_i|$.

Furthermore, observe that $c_F^\nu$ has order $N^{1-c}$ in the context of \Cref{lln:generalization}. For a rigorous definition of orders which is derived from this idea, see \Cref{def:order}. 

\begin{definition}
\label{def:seq}
Suppose $s=\{s_j\}_{1\leq j\leq k}$ is a \textit{sequence} of $k$ operators that act on formal series over $x_1,\ldots,x_N$. For $1\leq j\leq k$, $s_j$ has an associated \textit{index} $i_j$, $1\leq i_j\leq N$ and $\mathfrak{r}=\{i_j\}_{1\leq j\leq k}$ are the $\textit{indices}$ of $s$. Also, $\theta\in\mathbb{C}$ is the $\textit{factor}$ of $s$. For $1\leq j\leq k$, $s_j$ is one of the following:
\begin{enumerate}
    \item (\textit{Derivative}) $\frac{\partial}{\partial x_{i_j}},$ denoted by $\partial_{i_j}$.
    \item (\textit{Switch}) $\frac{\theta}{x_{i_j}-x_i}(1-s_{i_j, i})$, where $1\leq i\leq N$, $i\not= i_j$. This is the $\textit{switch}$ from $i_j$ to $i$.
    \item (\textit{Term multiplication}) Multiplication by $c_F^\nu\partial_{i_j}(x_1^{a_1}\cdots x_N^{a_N})$ for $a_i$, $1\leq i\leq N$ such that $\nu\in P^+$ and $\pi((a_1,\ldots,a_N))=\nu$. We say that $c_F^\nu$ is the $\textit{constant}$ of the term multiplication.
    \item (\textit{Change}) $\theta (d_{i_j}-C_{{i_j},i})$, where $1\leq i\leq N$, $i\not= i_j$. This is the $\textit{change}$ from $i_j$ to $i$. Recall \pref{eq:decdeg} and \pref{eq:change} for the definitions of $d_{i_j}$ and $C_{i_j,i}$.
\end{enumerate}
\end{definition}

Let $s$ be a sequence with indices $\{i_j\}_{1\leq j\leq k}$. Suppose $1\leq j\leq k$. Given $i_j$, observe that sum of the term multiplications choices for $s_j$ is equivalent to multiplying by $\partial_{i_j} F(x_1,\ldots,x_N)$. Also, if $s_j$ is a term multiplication by $p=c_F^\nu\frac{\partial}{\partial x_{i_j}}x_1^{a_1}\cdots x_N^{a_N}$, we say that $x_i$ is \textit{degree-altered} by $s_j$ if $p\not=0$ and the degree of $x_i$ in $p$ is at least $1$. Next, for $0\leq j\leq k$, let
\begin{equation}
\label{eq:seqval}
r(s)_j = s_j\circ s_{j-1}\circ\cdots\circ s_1(1), 
\end{equation}
where $r(s)_0=1$.

\begin{lemma}
\label{sequences}
Suppose $s$ is a sequence of length $k$ and $0\leq j\leq k$. Suppose the positive integers $q$ such that $1\leq q\leq j$ and $s_q$ is a term multiplication are $j_i$, $1\leq i\leq m$, where $j_i<j_{i+1}$ for $1\leq i\leq m-1$. Also, assume $s_{j_i}$ has constant $c_F^{\nu_i}$ for $1\leq i\leq m$. Then,
\begin{equation}
\label{eq:polyform}
r(s)_j = \prod_{i=1}^m c_F^{\nu_i} \cdot P(x_1,\ldots,x_N)
\end{equation}
for a homogeneous integer polynomial $P(x_1,\ldots,x_N)$ which is $0$ or has degree $\sum_{i=1}^m |\nu_i|-j$.
\end{lemma}
\begin{proof}
It is clear that \pref{eq:polyform} is true for some integer polynomial $P(x_1, \ldots, x_N)$ from the definition of sequences. We prove that $P$ is $0$ or has degree $\sum_{i=1}^m |\nu_i|-j$ with induction on $j$. The base case $j=0$ is clear, since $r(s)_0=1$. Assume the result holds for $j$, $0\leq j\leq k-1$. We want to show the result holds for $j+1$. Suppose $P(x_1,\ldots,x_N)$ and $P'(x_1,\ldots,x_N)$ are the polynomials for $r(s)_j$ and $r(s)_{j+1}$, respectively.

If $r(s)_j=0$, then $r(s)_{j+1}=0$ and the statement holds. Assume $r(s)_j\not=0$, so $P$ is a homogeneous integer polynomial with degree $\sum_{i=1}^m |\nu_i|-j$. If $s_{j+1}$ is a derivative, switch, or change, we see that $P'$ is either $0$ or a homogeneous integer polynomial with degree
$\sum_{i=1}^m |\nu_i|-j-1$. However, if $s_{j+1}$ is term multiplication by $c_F^{\nu_{m+1}}\partial_{i_j}(x_1^{a_1}\cdots x_N^{a_N})$, we have that $P'(x_1,\ldots,x_N)=(\partial_{i_j}x_1^{a_1}\cdots x_N^{a_N})P(x_1,\ldots,x_N)$
is either $0$ or a homogeneous integer polynomial with degree $\sum_{i=1}^{m+1} |\nu_i|-j-1$. The induction is complete.
\end{proof}

\begin{corollary}[Lemma 5.3 of \cite{matrix}]
\label{polynomial}
For a sequence $\mathfrak{r}=\{i_j\}_{1\leq j\leq k}$ of positive integers and $N\geq\max(\mathfrak{r})$, $[1]\mathcal{D}^\theta_\mathfrak{r}(F(x_1,\ldots,x_N))$ is a polynomial in the $c_F^\nu$ for $\nu\in P^+$ which is homogeneous of degree $k$, where $c_F^\nu$ has degree $|\nu|$ for $\nu\in P^+$.
\end{corollary}
\begin{proof}
This follows from summing the applications of \Cref{sequences} to sequences $s$ with indices $\mathfrak{r}$ that only contain derivatives, term multiplications, and switches.
\end{proof}

Note that when computing $[1]\mathcal{D}_\mathfrak{r}^\theta(F)$, we only need to consider sequences $s$ such that $r(s)_k$ is nonzero and has degree $0$ in the $x_i$. We discuss this idea further and characterize such sequences in \Cref{def:setseq}.

Suppose $s$ is a sequence. Let $k$ be the length of $s$ and assume that $r(s)_k$ is nonzero and has degree $0$ in the $x_i$. From \Cref{sequences}, $r(s)_k=d\prod_{i=1}^m c_F^{\nu_i}$ for some nonzero integer $d$ and $\nu_i\in P^+$, $1\leq i\leq m$ such that the sum of the $|\nu_i|$ is $k$. Let $C(s)=d$.

\begin{definition}
\label{def:setseq}
Suppose $k$ is a positive integer, $\mathfrak{r}=\{i_j\}_{1\leq j\leq k}$ is a sequence of positive integers, and $\theta\in\mathbb{C}$. For $N\geq\max(\mathfrak{r})$, let $T_{N,\theta}(\mathfrak{r})$, $T_{N,\theta}^1(\mathfrak{r})$, and $T_{N,\theta}^2(\mathfrak{r})$ denote the sets of sequences $s$ over $x_1,\ldots,x_N$ with indices $\mathfrak{r}$ and factor $\theta$ such that $r(s)_k$ is nonzero and has degree $0$ in the $x_i$ such that:
\begin{itemize}
    \item If $s\in T_{N, \theta}(\mathfrak{r})$, the operators of $s$ can be any operator (derivatives, term multiplications, switches, or changes).
    \item If $s\in T_{N,\theta}^1(\mathfrak{r})$, the operators of $s$ are derivatives, term multiplications, or switches.
    \item If $s\in T_{N,\theta}^2(\mathfrak{r})$, the operators of $s$ are term multiplications or changes.
\end{itemize}
\end{definition}

Throughout \Cref{sec:main}, $T_{N,\theta}(\mathfrak{r})$, $T_{N,\theta}^1(\mathfrak{r})$, and $T_{N,\theta}^2(\mathfrak{r})$ are referred to. Importantly, we have that
\[
[1]\mathcal{D}^\theta_{\mathfrak{r}}(F(x_1,\ldots,x_N))= [1]\left(\prod_{j=1}^k\left(\mathcal{D}^\theta_{i_j} + \frac{\partial}{\partial x_{i_j}}F(x_1,\ldots,x_N)\right)\right)(1)= \sum_{s\in T_{N,\theta}^1(\mathfrak{r})} r(s)_k.
\]

We provide some notation we use to work in the context of \Cref{lln:generalization}. A complex number $\theta$ and a real number $c$ are given such that $c<1$. Furthermore, a sequence $\{\theta_N\}_{N\geq 1}$ of complex numbers that satisfies $\lim_{N\rightarrow\infty}N^c\theta_N=\theta$ is given.

\subsection{Orders of polynomials}

\begin{definition}
\label{def:order}
Suppose $Q(c_F^\nu, \nu\in P^+; N)$ is a polynomial in the $c_F^\nu$, $\nu\in P^+$ with finite degree and coefficients that are functions of $N$. Then, $Q(c_F^\nu, \nu\in P^+; N)$ is of $\textit{order}$ $N^k$ if for a term  
\[
\prod_{i=1}^m c_F^{\nu_i},
\]
the absolute value of its coefficient in $Q(c_F^\nu; N)$ is $O(N^{k-m(1-c)})$ for sufficiently large $N$.
\end{definition} 
If $Q(c_F^\nu; N)$ has an order of $N^k$, we also say that $Q(c_F^\nu; N)$ is $O(N^k)$. Later, for a sequence of $k$ positive integers $\mathfrak{r}$ and $S_N\subset T_{N,\theta_N}(\mathfrak{r})$ for $N\geq \max(\mathfrak{r})$, we consider when
\begin{equation}
\label{eq:order}
Q(c_F^\nu; N) = \sum_{s\in S_N} r(s)_k.
\end{equation}
Here, note that in $\sum_{s\in S_N} r(s)_k$, $\theta_N$ can be considered as a function of $N$. For example, we get $Q(c_F^\nu;N)=\mathcal{D}^{\theta_N}_{\mathfrak{r}}(F(x_1,\ldots,x_N))|_{x_i=0, 1\leq i\leq N}$ when $S_N=T_{N,\theta_N}^1(\mathfrak{r})$ for $N\geq \max(\mathfrak{r})$. Later on, \Cref{def:order} and \pref{eq:order} are used with $\mathcal{D}^{\theta_N}_{\mathfrak{r}}$ as well as other operators. 

\begin{proposition}
\label{prop:boundedconst}
Suppose $\mathfrak{r}=\{i_j\}_{1\leq j\leq k}$ is a sequence of positive integers and $l$ is a nonnegative integer. A constant $M>0$ exists such that for any $s\in T_{N,\theta_N}(\mathfrak{r})$ with $l$ switches or changes, $|C(s)|\leq M\cdot N^{-lc}$ for sufficiently large $N$.
\end{proposition}
\begin{proof}
For such a sequence $s$, for $0\leq j\leq k$, suppose $P_j(x_1,\ldots,x_N)$ is the polynomial for $r(s)_j$ in \pref{eq:seqval}. Suppose $s$ has $m$ term multiplications. Since each other operator decreases the degree in the $x_i$ by $1$, the degree in the $x_i$ of $P_j(x_1,\ldots,x_N)$ for $0\leq j\leq k$ is at most $k-m$, which is less than $k$. Otherwise, $r(s)_k$ will not be nonzero and have degree $0$ in the $x_i$.

Using this, the coefficient $M$ in \Cref{prop:boundedconst} can be obtained by, for $0\leq j\leq k-1$, bounding the factor $s_{j+1}$ changes the sum of the absolute values of the coefficients from $P_j(x_1,\ldots,x_N)$ to $P_{j+1}(x_1,\ldots,x_N)$. Furthermore, the $N^{-lc}$ is from the $l$ switches or changes, since $\lim_{N\rightarrow\infty}N^c\theta_N=\theta$. To finish the proof, note that $P_k(x_1,\ldots,x_N)=C(s)$.
\end{proof}

\begin{proposition}
\label{order}
Suppose $k$ is a positive integer and $\mathfrak{r}=\{i_j\}_{1\leq j\leq k}$ is a sequence of positive integers. For $N\geq\max(\mathfrak{r})$, $[1]\mathcal{D}^{\theta_N}_\mathfrak{r}(F(x_1,\ldots,x_N))$ is of order $N^{k(1-c)}$.
\end{proposition}

Note that \Cref{order} for $c=1$ is proved in Lemma 5.1 of the paper \cite{matrix}. The idea of the proof from the paper is also true for $c<1$, since the contribution of $\mathcal{D}_{i_j}^\theta + \frac{\partial}{\partial x_{i_j}} F(x_1,\ldots,x_N)$ to the order is $N^{1-c}$, which arises from $N$ switches with order $N^{-c}$ and a term multiplication with order $N^{1-c}$. Because this is true for $1\leq j\leq k$, the total contribution is $N^{k(1-c)}$. This method does not account for the $N^{\Omega(1)}$ terms of $F$, although it is clear that only a finite number of terms will contribute to the final value. For completeness, we include a detailed proof of \Cref{order}. The casework involved is also relevant for proving \Cref{remainder2}, \Cref{order2}, and \Cref{thm:equivalence} in \Cref{sec:coeff}.

\begin{proof}[Proof of \Cref{order}]
We know from \Cref{polynomial} that the expression will be a polynomial in the $c_{F}^\nu$ which is homogeneous of order $k$. Suppose that for a positive integer $m$, $\nu_i\in P^+$ for $1\leq i\leq m$ and the sum of the $|\nu_i|$ is $k$. Note that the number of possible $\nu_1, \ldots, \nu_m$, $1\leq m\leq k$, is finite. Then, if we show that for $N\geq\max(\mathfrak{r})$, the coefficient of
\[
p=\prod_{i=1}^m c_{F}^{\nu_i}
\]
is of order $N^{(k-m)(1-c)}$, we will be done by \Cref{def:order}. For $N\geq\max(\mathfrak{r})$, suppose the coefficient of $p$ is $r$; we want to show that $|r|=O(N^{(k-m)(1-c)})$. Let $T$ be the set of sequences $s \in T_{N,\theta_N}^1(\mathfrak{r})$ such that $r(s)_k=dp$, with $d$ a nonzero integer. Also, for $0\leq l\leq k$, let $T_l$ be the set of sequences $s\in T$ with $l$ switches. Moreover, for $s\in T$, let $D(s)$ be the set of $x_i$ such that $i\notin\mathfrak{r}$ which are degree-altered by a term multiplication of $s$.

The coefficient of $p$ is
\[
r=\sum_{s\in T} C(s) = \sum_{l = 0}^k \sum_{s\in T_l} C(s).
\]
From \Cref{prop:boundedconst}, there exists $P_l>0$ such that for all $N$, for all $s\in T_l$, $|C(s)|\leq P_lN^{-lc}$. Then, 
\[
\sum_{s\in T_l} |C(s)| \leq P_l N^{-lc}|T_l|
\]
from the triangle inequality.

Suppose $s\in T$. We know that exactly $m$ of the $s_i$ must be term multiplications, and as seen in the proof of \Cref{prop:boundedconst}, the degree in the $x_i$ of $r(s)_j$ for $0\leq j\leq k$ is at most $k-m$. Also, the number of switches is at most $k-m$, meaning that $|T_l|=0$ if $l>k-m$.

\begin{claim}
\label{claim:switches}
Suppose $s\in T$. If $x_i\in D(s)$, suppose the first term multiplication that degree-alters $x_i$ is $s_{j'}$. Then, $s_j$ must be a switch from $i_j$ to $i$ for some $j$, $j'<j\leq k$.
\end{claim}
\begin{proof}
Note that in $r_{j'}(s)$, $x_i$ is in all of the nonzero terms, and must be removed because $r_k(s)$ cannot contain $x_i$. For the sake of contradiction, assume the statement does not hold. If $j'=k$, then $r_k(s)=r_{j'}(s)$ contains $x_i$, a contradiction since $r_k(s)=dp$ for a nonzero integer $d$. Therefore, $j'<k$.

For $j$ such that $j'<j\leq k$, we have that if $s_j$ is a derivative, term multiplication, or switch, all terms of $r_j(s)$ will have $x_i$. Since $x_i\in D(s)$, it cannot be removed from a term by a derivative without converting the term to zero, and it similarly cannot be removed by a term multiplication. Based on the assumption, it also cannot be removed by a switch. Then, in $r_k(s)\not= 0$, the terms will contain $x_i$, a contradiction.
\end{proof}

Suppose $s\in T_l$ and let $S=\{x_{i_j}|1\leq j\leq k\}$. If $|D(s)|=d$, from \Cref{claim:switches}, we see that we can find one switch from $i_j$ to $i$ for each $x_i\in D(s)$, giving $d$ switches in total, where $0\leq d\leq l$. Also, the number possibilities for $D(s)$, the number subsets of $\{x_1, \ldots, x_N\}\backslash S$ with size $d$, is at most $N^d$. Next, suppose $X\subset\{x_1, \ldots, x_N\}\backslash S$, $|X|=d$. Because the variables not in $\mathfrak{r}$ are symmetric, the number of $s\in T$ such that $D(s)=X$ is the same for all such $X$; let this number be $Q_d$.

In $s$, there are $m$ term multiplications, known as $\alpha(s)$ operators. Also, there are $d\leq k-m$ operators which are switches from $i_j$ to $i$, where $x_i\in D(s)$ and the $i$ are distinct, known as $\beta(s)$ operators. The other $k-m-d$ operators, known as $\gamma(s)$ operators, can be derivatives or any switch. In this proof, there can be overcounting of $Q_d$, with $s\in T$ counted multiple times. Also, for $0\leq d\leq l$, there are $\frac{k!}{m!d!(k-m-d)!}$ possible groupings of the $s_j$ into $\alpha(s), \beta(s)$, and $\gamma(s)$ operators. 

For a term multiplication, or an $\alpha(s)$ operator, suppose the term before the derivative is $q$. We know that all $x_i$ in $q$ must be in $S\cup D(s)$. Also, the total degree of $q$ after the derivative is at most $k-m$. We look at degrees of the $x_i$ in $q$. Note that for $x_{i_j}$, where $\partial_{i_j}q\not=0$, the degree must be at least $1$ and at most $k-m+1$, giving $k-m+1$ possibilities. For the other $x_i$, the degree must be at least $0$ and at most $k-m$, also giving $k-m+1$ possibilities. With this, the number of possibilities for the $\alpha(s)$ operators is at most
\[
(k-m+1)^{(k+|D(s)|)m}=(k-m+1)^{(k+d)m}.
\]
On the other hand, for the $\beta(s)$ operators, there being $d$ in total, there are $d!$ possible orderings. Finally, consider the $k-m-d$ $\gamma(s)$ operators. We know that $l-d$ of the $\gamma(s)$ operators are switches, and the others are derivatives. We see that since each switch has $N-1$ possibilities, the number of possible $\gamma(s)$ operators is at most
\[
(N-1)^{l-d}\cdot\frac{(k-m-d)!}{(l-d)!(k-m-l)!}
\]
Then, for a constant $C_d'$ not depending on $N$,
\begin{align*}
Q_d&\leq \frac{k!}{m!d!(k-m-d)!}\cdot (k-m+1)^{(k+d)m}\cdot d!\cdot (N-1)^{l-d}\cdot\frac{(k-m-d)!}{(l-d)!(k-m-l)!}\\
& \leq C_d'N^{l-d}.
\end{align*}

From this, for $0\leq d\leq l$, the number of $s\in T_l$ such that $|D(s)|=d$ is at most $N^d Q_d\leq C_d'N^l$, with $D(s)$ having at most $N^d$ possibilities. Therefore,
\[
|T_l| \leq \left(\sum_{d=0}^l C_d'\right)N^l,
\]
and $|T_l|=O(N^l)$. From this, for a constant $K_l$, $|T_l|\leq K_lN^l$ for sufficiently large $N$. Note that
\[
r = \sum_{s\in T} C(s) = \sum_{l=0}^{k-m} \sum_{s\in T_l} C(s),
\]
and with the triangle inequality,
\[
|r|\leq\sum_{l=0}^{k-m}\sum_{s\in T_l} |C(s)|\leq\sum_{l=0}^{k-m} P_l N^{-lc}|T_l|\leq\sum_{l=0}^{k-m} P_lK_l N^{l(1-c)}
\]
for sufficiently large $N$. Then, since $c<1$, $|r|=O(N^{(k-m)(1-c)})$ for sufficiently large $N$, as needed. 
\end{proof}

\subsection{Remainders}
\label{subsec:remainders}
\begin{lemma}
\label{prop:limitzero}
Suppose $\{a_\nu\}_{\nu\in P^+}$ is a sequence of functions from $\mathbb{N}$ to $\mathbb{C}$ such that for all $\nu\in P^+$, there exists a constant $C>0$ such that $|a_\nu(N)|\leq C N^{1-c}$ for sufficiently large $N$. Then, if $Q(c_F^\nu, \nu\in P^+; N)$ is of order $N^k$ and $\epsilon>0$,
\[
\lim_{N\rightarrow\infty}\frac{Q(c_F^\nu = a_\nu(N), \nu\in P^+; N)}{N^{k+\epsilon}}=0.
\]
\end{lemma}
\begin{proof}
This is clear by showing that the $N\rightarrow\infty$ limit of each of the finitely many terms of $\frac{1}{N^{k+\epsilon}}Q(c_F^\nu=a_\nu(N), \nu\in P^+;N)$ is $0$.
\end{proof}

Suppose we have a polynomial $R(c_F^\nu, \nu\in P^+; N)$ which is of order $N^k$. From the conditions of \Cref{lln:generalization}, we see that for each $\nu\in P^+$, there exists a constant $C>0$ such that $|c_{F_N}^\nu|\leq CN^{1-c}$ for sufficiently large $N\geq 1$. For all $\nu\in P^+$, let $a_\nu$ be the function such that $a_\nu(N)=c_{F_N}^{\nu}$ for $N\geq 1$. Afterwards, from \Cref{prop:limitzero}, for $\epsilon>0$,
\begin{equation}
\label{eq:remainderto0}
\lim_{N\rightarrow\infty} \frac{R(c_F^\nu = c_{F_N}^\nu, \nu\in P^+; N)}{N^{k+\epsilon}}=\lim_{N\rightarrow\infty} \frac{R(c_F^\nu = a_\nu(N), \nu\in P^+; N)}{N^{k+\epsilon}} = 0.
\end{equation}
We use \pref{eq:remainderto0} in the proof of \Cref{lln:generalization} in Subsection \ref{subsec:llngeneralization} to show that such remainders $R$ have a $N\rightarrow\infty$ limit of $0$. Particularly, \pref{eq:remainderto0} can be used with various results from Subsections \ref{subsec:remainders} and \ref{subsec:differentop}.

The following result has a similar statement and proof as \cite{matrix}*{Corollary 5.4}; the only difference is the order of the remainder.

\begin{proposition}
\label{remainder1}
Suppose that $\lambda$ is a partition with $\ell(\lambda)=m$ and $|\lambda|=k$. Also, suppose $\mathfrak{r}=\{i_j\}_{1\leq j\leq k}$ is a sequence such that the first $\lambda_1$ values are $1$, the next $\lambda_2$ values are $2$, and so forth, until the last $\lambda_m$ values are $m$. Then, for sufficiently large $N\geq m$,
\[
\frac{1}{N^{m}} \sum_{l\in I_N(\lambda)} [1]\mathcal{D}^{\theta_N}_{l}(F(x_1,\ldots,x_N))  = [1]\mathcal{D}^{\theta_N}_\mathfrak{r}(F(x_1,\ldots,x_N)) + R
\]
for a homogeneous polynomial $R$ in the $c_F^\nu$ with degree $k$ and order $N^{k(1-c)-1}$.
\end{proposition}
\begin{proof}
Suppose that in $l\in I_N(\lambda)$, the first $\lambda_1$ indices are $i_1$, the next $\lambda_2$ indices are $i_2$, and so forth, until the last $\lambda_m$ indices are $i_m$, where $1\leq i_j\leq N$ for $1\leq j\leq m$. From \Cref{order}, for all $l\in I_N(\lambda)$, $\mathcal{D}^{\theta_N}_{l}(F(x_1,\ldots,x_N))$ has order $N^{k(1-c)}$. Note that if $l$ has all $i_j$ distinct, $[1]\mathcal{D}^{\theta_N}_{l}(F(x_1,\ldots,x_N))$ is equal, by symmetry, to $[1]\mathcal{D}^{\theta_N}_{\mathfrak{r}}(F(x_1,\ldots,x_N))$. Consider the set $S$ of other $l$, where some of the $i_j$ are equal. Let
\[
R'=\frac{1}{N^m}\sum_{l\in S} [1]\mathcal{D}^{\theta_N}_l(F(x_1,\ldots,x_N)).
\]
From \Cref{polynomial}, $R'$ is a polynomial in the $c_F^\nu$ which is homogeneous of degree $k$. 

Note that there are $N^{m}-O(N^{m-1})$ possible $l$ with all of the $i_j$ are distinct, and dividing by $N^{m}$ will give that the sum of $\frac{1}{N^m}[1]\mathcal{D}_l^{\theta_N}(F(x_1,\ldots,x_N))$ for such $l$ is
\[
\left(1-O\left(\frac{1}{N}\right)\right)[1]\mathcal{D}^{\theta_N}_{\mathfrak{r}}(F(x_1,\ldots,x_N)).
\]
Also, $|S|=O(N^{m-1})$ and each $\mathcal{D}^{\theta_N}_l(F(x_1,\ldots,x_N))$ is order $N^{k(1-c)}$, so $R'$ is $\frac{1}{N^m}\cdot O(N^{m-1})\cdot O(N^{k(1-c)})=O(N^{k(1-c)-1})$.
Then, we get
\[
R=R'-O\left(\frac{1}{N}\right)[1]\mathcal{D}^{\theta_N}_{\mathfrak{r}}(F(x_1,\ldots,x_N))
\]
is a homogeneous polynomial in the $c_F^\nu$ with degree $k$ and order $N^{k(1-c)-1}$, as desired.
\end{proof}

The following result has a similar statement as Claim A in the proof of Proposition 5.5 of \cite{matrix} and the approach we use follows the same general idea. However, we require more details, because we must account for the coefficients $c_F^\nu$ having order $N^{1-c}$, which is polynomial in $N$ since $c<1$.

\begin{lemma}
\label{remainder2}
Suppose that $\mathfrak{r}=\{i_j\}_{1\leq j\leq k}$ is a sequence of positive integers. For $N\geq \max(\mathfrak{r})$, let $H$ be the set of sequences $s$ in $T_{N,\theta_N}^1(\mathfrak{r})$ which satisfy the following conditions:
\begin{itemize}
    \item For $1\leq j\leq k$, if $s_j$ is the switch from $i_j$ to $i$, $i\notin \mathfrak{r}$.
    \item There do not exist integers $i$, $j_1$, and $j_2$, $1\leq i\leq N$, $1\leq j_1, j_2\leq k$, $j_1\not=j_2$ such that $s_{j_1}$ is the switch from $i_{j_1}$ to $i$ and $s_{j_2}$ is the switch from $i_{j_2}$ to $i$.
\end{itemize} 
Then, 
\[[1]\mathcal{D}^{\theta_N}_{\mathfrak{r}}(F(x_1,\ldots,x_N)) = \sum_{s\in H} r(s)_k + R\]
for a homogeneous polynomial $R$ in the $c_F^\nu$ with degree $k$ and order $N^{k(1-c)-1}$.
\end{lemma}
\begin{proof}
Suppose that $Q$ is the set $s\in T_{N,\theta_N}^1(\mathfrak{r})$ such that at least one of the conditions is not followed. Then,
\[
R = \sum_{s\in Q} r(s)_k.
\]
Since $Q\subset T_{N,\theta_N}^1(\mathfrak{r})$, $R$ is a polynomial in the $c_F^{\nu}$ which is homogeneous of degree $k$. It suffices to show that the coefficient in $R$ of each term
\[
p=\prod_{i=1}^m c_{F}^{\nu_i}
\]
is order $N^{(k-m)(1-c)-1}$ from \Cref{def:order}. We consider the set $T$ of $s\in Q$ such that $r(s)_k$ equals $p$ multiplied by a nonzero integer. Following the proof of \Cref{order}, let the coefficient of $p$ be $r$. We want to show that $|r|$ is $O(N^{(k-m)(1-c)-1})$. For $0\leq l\leq k-m$, suppose $T_l$ is the set of $s\in T$ with $l$ switches. For $s\in T_l$, we know that $|C(s)|\leq P_l\cdot N^{-lc}$ for a $P_l>0$ independent of $N$, see \Cref{prop:boundedconst}. Also, for $s\in T$, let $D(s)$ be the set of $x_i$ such that $i\notin\mathfrak{r}$ which are degree-altered by a term multiplication of $s$. Let $d=|D(s)|$.

If $s\in T_l$, $s$ has $m$ term multiplications, or $\alpha(s)$ operators. Also, there are $d$, $d\leq l$, $\beta(s)$ operators which are switches from $i_j$ to $i$ where $x_i\in D(s)$, such that each $x_i\in D(s)$ is in exactly one $\beta(s)$ operator. Note that the $\beta(s)$ operators exist by \Cref{claim:switches}. The other $k-m-d$ operators are $\gamma(s)$ operators, and are derivatives or switches. Note that the number of switches in the $\gamma(s)$ operators is $l-d$, and each of these switch has $N-1$ possibilities.

Consider when one of the switches is from $i_j$ to $i\in \mathfrak{r}$. Since $\mathfrak{r}$ and $D(s)$ are disjoint, the switch must be a $\gamma(s)$ operator. Afterwards, since $|\mathfrak{r}|$ does not depend on $N$, the number of possibilities for the $\gamma(s)$ operators is $O(N^{l-d-1})$. If $d=l$, the number of possibilities is $0=O(N^{l-d-1})$, since none of the $\gamma(s)$ operators are switches. With the number of possibilities for $D(s)$ being at most $N^d$ and for the other choices being bounded, the total number of such $s\in T_l$ is then $O(N^{l-1})$. 

For the next case, suppose there are two switches with the same $i$. Then, we could have a $\beta(s)$ and $\gamma(s)$ switch or two $\gamma(s)$ switches with the same $i$. For the first case, there are at most $N$ ways to pick $i$. Afterwards, the number of possibilities for $D(s)$ is at most $N^{d-1}$ and the $\gamma(s)$ switches is $O(N^{l-d-1})$. Where there are $O(1)$ possibilities for the other choices, this gives $O(N^{l-1})$ possibilities for the first case. Next, for the second case, there are at most $N$ ways to choose $i$. Afterwards, the number of possibilities for $D(s)$ is at most $N^d$ and the $\gamma(s)$ switches is $O(N^{l-d-2})$. Since there are $O(1)$ possibilities for the other choices, this gives 
$O(N^{l-1})$ possibilities for the second case. The number of $s\in T_l$ for this case is then $O(N^{l-1})$ as well.

Therefore, $|T_l|=O(N^{l-1})$. Because of this,
\[
\Bigg|\sum_{s\in T_l} C(s)\Bigg| \leq \sum_{s\in T_l} |C(s)| \leq P_lN^{-lc}|T_l|
\]
is $O(N^{l(1-c)-1})$. However,
\[
r=\sum_{s\in T} C(s) = \sum_{l=0}^{k-m}\sum_{s\in T_l} C(s).
\]
After this, since $l\leq k-m$, $|r|$ is $O(N^{(k-m)(1-c)-1})$.
\end{proof}

\subsection{Sequences involving changes}
\label{subsec:differentop}

Next, we consider the operators $\mathcal{Q}_i^N$ from \Cref{sec:formal}, which involve changes. Recall that the operators are defined in \pref{eq:operator2}. Suppose that for $N\geq i$,
\[
\mathcal{Q}_i^N\left(\frac{\partial_i F}{N^{1-c}}\right) = \frac{N^c\theta_N}{N}\sum_{\substack{1\leq j\leq N, \\ j\not= i}} (d_i-C_{i,j})+\frac{\partial_i F(x_1,\ldots,x_N)}{N^{1-c}},
\]
and we recall that $\lim_{N\rightarrow\infty} N^c\theta_N=\theta$. Afterwards,
\[
N^{1-c}\mathcal{Q}_i^N\left(\frac{\partial_i F}{N^{1-c}}\right)  = \theta_N\sum_{\substack{1\leq j\leq N, \\ j\not= i}} (d_i-C_{i,j}) + \frac{\partial}{\partial x_i} F(x_1,\ldots,x_N)
\]
is considered. In \Cref{remainder3}, we show that we can replace each $\mathcal{D}^{\theta_N}_{i_j}$ with $N^{1-c}\mathcal{Q}_{i_j}^N\left(\frac{\partial_{i_j}F}{N^{1-c}}\right)$. Afterwards, in \Cref{remainder4}, we use the $\mathcal{R}_{i,j}^k$ operators to evaluate the resulting expression. Similarly as before, for $\mathfrak{r}=\{i_j\}_{1\leq j\leq k}$ and $N\geq\max(\mathfrak{r})$, we let
\[
\mathcal{Q}_{\mathfrak{r}}^N(F(x_1,\ldots,x_N)) = N^{k(1-c)}\prod_{j=1}^k \mathcal{Q}_{i_j}^N\left(\frac{\partial_{i_j} F}{N^{1-c}}\right)(1), 
\]
with
\[
[1]\mathcal{Q}_{\mathfrak{r}}^N(F(x_1,\ldots,x_N))=\sum_{s\in T_{N,\theta_N}^2(\mathfrak{r})} r(s)_k.
\]

The following analogue of \Cref{polynomial} holds.
\begin{corollary}
\label{polynomialchange}
For a sequence $\mathfrak{r}=\{i_j\}_{1\leq j\leq k}$ of positive integers and $N\geq\max(\mathfrak{r})$, $[1]\mathcal{Q}^N_\mathfrak{r}(F(x_1,\ldots,x_N))$ is a polynomial in the $c_F^\nu$ for $\nu\in P^+$ which is homogeneous of degree $k$, where $c_F^\nu$ has degree $|\nu|$ for $\nu\in P^+$.
\end{corollary}
\begin{proof}
This follows from summing the applications of \Cref{sequences} to sequences $s$ with indices $\mathfrak{r}$ and factor $\theta_N$ that only contain term multiplications and changes.
\end{proof}

\begin{proposition}
\label{order2}
Let $\mathfrak{r}=\{i_j\}_{1\leq j\leq k}$ be a sequence of positive integers. For $N\geq\max(\mathfrak{r})$, $[1]\mathcal{Q}^N_{\mathfrak{r}}(F(x_1,\ldots,x_N))$ has order $N^{k(1-c)}$.
\end{proposition}
\begin{proof}
The same proof as for \Cref{order} can be used, but with no derivatives and changes instead of switches. Particularly, if a switch is from $i$ to $j$, a change from $i$ to $j$ is used instead.
\end{proof}

\begin{lemma}
\label{newremainder2}
Suppose that $\mathfrak{r}=\{i_j\}_{1\leq j\leq k}$ is a sequence of positive integers. For $N\geq\max(\mathfrak{r})$, suppose $H$ is the set of sequences $s$ in $T_{N,\theta_N}^2(\mathfrak{r})$ which satisfy the following conditions:
\begin{itemize}
    \item For $1\leq j\leq k$, if $s_j$ is the change from $i_j$ to $i$,  $i\notin\mathfrak{r}$.
    \item There do not exist integers $i$, $j_1$, and $j_2$, $1\leq i\leq N$, $1\leq j_1, j_2\leq k$, $j_1\not=j_2$, such that $s_{j_1}$ is the change from $i_{j_1}$ to $i$ and $s_{j_2}$ is the change from $i_{j_2}$ to $i$.
\end{itemize} 
Then,
\[
[1]\mathcal{Q}_{\mathfrak{r}}^N(F(x_1,\ldots,x_N)) = \sum_{s\in H} r(s)_k+ R
\]
for a homogeneous polynomial $R$ in the $c_F^\nu$ with degree $k$ and order $N^{k(1-c)-1}$.
\end{lemma}
\begin{proof}
The same proof as for \Cref{remainder2} can be used, but with no derivatives and changes instead of switches. Particularly, if a switch is from $i$ to $j$, a change from $i$ to $j$ is used instead.
\end{proof}

\begin{theorem}
\label{remainder3}
Suppose that $k$ is a positive integer and $\mathfrak{r}=\{i_j\}_{1\leq j\leq k}$ is a sequence of positive integers. For $N\geq \max(\mathfrak{r})$,
\begin{equation}
\label{eq:diffoperators}
    [1]\mathcal{D}^{\theta_N}_\mathfrak{r}(F(x_1,\ldots,x_N))= [1]\mathcal{Q}_\mathfrak{r}^N(F(x_1,\ldots,x_N)) + R
\end{equation}
for a homogeneous polynomial $R$ in the $c_F^\nu$ with degree $k$ and order $N^{k(1-c)-1+\max(c,0)}$.
\end{theorem}
\begin{proof}
In \pref{eq:diffoperators}, $T_{N,\theta_N}^1(\mathfrak{r})$ and $T_{N,\theta_N}^2(\mathfrak{r})$ correspond to the left and right hand side, respectively. However, the sequences in $T_{N,\theta_N}^1(\mathfrak{r})$ can have derivatives, but the sequences in $T_{N,\theta_N}^2(\mathfrak{r})$ cannot. Then, let $T_{N,\theta_N}^3(\mathfrak{r})$ be the set of $s\in T_{N,\theta_N}^1(\mathfrak{r})$ which do not contain derivatives.

Suppose that $H$ is the set of $s$ in $T_{N,\theta_N}^1(\mathfrak{r})$ satisfying the \Cref{remainder2} conditions. From the Lemma,
\[ 
    [1]\mathcal{D}^{\theta_N}_{\mathfrak{r}}(F(x_1,\ldots,x_N))=\sum_{s\in H} r(s)_k + R_1,
\]
where $R_1$ is a homogeneous polynomial in the $c_F^\nu$ with degree $k$ and order $N^{k(1-c)-1}$. Consider $H_1=H\cap T_{N,\theta_N}^3(\mathfrak{r})$. We see that $H_1$ is the set of $s\in T_{N,\theta_N}^1(\mathfrak{r})$ satisfying the conditions of \Cref{remainder2} with only switches and term multiplications, and
\[
    \sum_{s\in H} r(s)_k = \sum_{s\in H_1} r(s)_k + R_2, R_2= \sum_{s\in H\backslash T_{N,\theta_N}^3(\mathfrak{r})} r(s)_k.
\]
With this,
\[
[1]\mathcal{D}^{\theta_N}_{\mathfrak{r}}(F(x_1,\ldots,x_N))=\sum_{s\in H_1} r(s)_k + R_1+R_2.
\]

Also, let $H_2$ be the set of $s\in T_{N,\theta_N}^2(\mathfrak{r})$ satisfying the \Cref{newremainder2} conditions. By the lemma, 
\[
    [1]\mathcal{Q}_{\mathfrak{r}}^N(F(x_1,\ldots,x_N))=\sum_{s\in H_2} r(s)_k+R_3,
\]
for a homogeneous polynomial $R_3$ in the $c_F^\nu$ with degree $k$ and order $N^{k(1-c)-1}$.

\begin{claim}
\label{claim:r2order}
$R_2$ is a homogeneous polynomial in the $c_F^\nu$ with degree $k$ and order $N^{k(1-c)-(1-c)}$.
\end{claim}
\begin{proof}
Suppose 
\[
    p=\prod_{i=1}^m c_F^{\nu_i}.
\] 
Let the set of $s\in H\backslash T_{N,\theta_N}^3(\mathfrak{r})$, the $s\in H$ which have a derivative, such that $r(s)_k$ is $p$ multiplied by a nonzero integer be $T$. Let $T_l$ be the set of $s\in T$ that have $l$ switches, $0\leq l\leq k-m$. Also, for $s\in T$, let $D(s)$ be the set of $x_i$ such that $i\notin\mathfrak{r}$ which are degree-altered by term multiplications of $s$.

Suppose $s\in T_l$. Note that the derivatives of $s$ only appear in $\gamma(s)$ operators. Suppose $|D(s)|=d\leq l$. We know that in the $\gamma(s)$ operators, we have $l-d$ switches and $k-m-l$ derivatives. Since the $\gamma(s)$ operators must have at least $1$ derivative, we see that $0\leq l<k-m$. Particularly, $|T_l|$ is $O(N^l)$ for $0\leq l<k-m$ and $|T_l|=0$ for $l\geq k-m$, after using the method of \Cref{remainder2}. Since $|C(s)|$ is $O(N^{-lc})$ for $s\in T_l$ by \Cref{prop:boundedconst},
\[
\bigg|\sum_{s\in T_l} C(s)\bigg|
\]
is $O(N^{l(1-c)})$. The absolute value of the coefficient of $p$ is at most
\[
\sum_{l=0}^{k-m-1}\bigg|\sum_{s\in T_l} C(s)\bigg|,
\]
and is therefore $O(N^{(k-m)(1-c)-(1-c)})$.
\end{proof}

For a sequence $s$ in $H_1$, let $c(s)$ be the sequence in $H_2$ such that for $1\leq j\leq k$, if $s_{j}$ is the switch from $i_j$ to $i$, then $c(s)_j$ is the change from $i_j$ to $i$, and $c(s)_j=s_j$ if $s_j$ is not a switch. We have that $s\in H_1$ have switches and term multiplications, while $s\in H_2$ have changes and term multiplications. Clearly, $c:H_1\rightarrow H_2$ is a bijection.

\begin{claim}
\label{claim:bijval}
For all $s\in H_1$, $r(s)_k = r(c(s))_k$. 
\end{claim}
\begin{proof}
Consider $s\in T_{N,\theta_N}(\mathfrak{r})$ with term multiplications, switches, and changes such that:
\begin{itemize}
    \item For $1\leq j\leq k$, if $s_j$ is the switch or change from $i_j$ to $i$, $i\notin \mathfrak{r}$.
    \item There do not exist integers $i$, $j_1$, and $j_2$, $1\leq i\leq N$, $1\leq j_1, j_2\leq k$, $j_1\not=j_2$, such that $s_{j_1}$ is the switch or change from $i_{j_1}$ to $i$ and $s_{j_2}$ is the switch or change from $i_{j_2}$ to $i$.
\end{itemize} 
Suppose that $s_j$, $1\leq j\leq k$ is a switch from $i_j$ to $i$ and $s_{j'}$ for $j<j'\leq k$ are not changes. It suffices to show that if we convert $s_j$ to the change from $i_j$ to $i$, $r(s)_k$ is the same. Then, for $s\in H_1$, if the switches of $s$ are $s_{j_1}, s_{j_2}, \ldots, s_{j_r}$, $1\leq j_1<\cdots<j_r\leq k$, we could convert $s_{j_u}$ to a change from $u=1$ to $r$ and get $c(s)$. This would imply $r(c(s))_k=r(s)_k$.

Suppose $q$ is a term of $r(s)_{j-1}$. We look at $s_j(q)$, which is a sum of terms or $0$. Suppose $q'$ is a term of $s_j(q)$ and the degree of $x_i$ in $q'$ is at least $1$. For $j< j'\leq k$, $s_{j'}$ is the switch from $i_{j'}$ to $i'\not=i$ or a term multiplication. In $r(s)_k$, the degree of $x_i$ must be $0$, but to decrease the degree of $x_i$ of a term, some $s_{j'}$ for $j'>j$ will convert the term to $0$. Therefore, $q'$ will not contribute to $r(s)_k$. So, the only term of $s_j(q)$ which will contribute to $r(s)_k$ will be $q'$ with degree of $x_i$ equal to $0$, and since $s_j$ is the switch from $i_j$ to $i$, in $s_j(q)$, there will be at most one such $q'$. With this, we can set $s_j(q)$ as $q'$ if such $q'$ exists and $s_j(q)=0$ if not.

Assume the degree of $x_i$ and $x_{i_j}$ are both at least $1$ in $q$. Then, in $s_j(q)$, the degree in $x_i$ of all terms will be at least $1$, so $q$ does not contribute to $r(s)_k$. Now, replace $s_j$ with the change from $i_j$ to $i$. Note that in $\theta_N(d_{i_j}-C_{i_j,i})(q)=\theta_N d_{i_j}(q)$, the degree of $x_i$ is at least $1$, and because of this, using the logic from above, after replacing $s_j$ there are also no contributions to $r(s)_k$. Therefore, in this case, $r(s)_k$ is unchanged. 

Next, suppose the degree of $x_i$ is $0$ and $x_{i_j}$ is at least $1$ in $q$. In this case, $q'$ exists and will be $\theta_N d_{i_j}(q)$. But, since the degree of $x_{i_j}$ is at least $1$,
\[
    \theta_N(d_{i_j}-C_{i_j,i})(q) = \theta_N d_{i_j}(q) = q'.
\]
Therefore, we can replace $s_j$ with $\theta(d_{i_j}-C_{i_j,i})$ and $r(s)_k$ will be the same.

Consider when the degree of $x_{i_j}$ in $q$ is $0$. If the degree of $x_i$ is at least $1$ in $q$, then $q'$ exists and is equal to $-\theta_N C_{i_j,i}(q)$. Otherwise, if the degree of $x_i$ is $0$, $s_j(q)=0=-\theta_N C_{i_j,i}(q)$. However, since the degree of $x_{i_j}$ in $q$ is $0$,
\[
    \theta_N(d_{i_j}-C_{i_j,i})(q)=-\theta_N C_{{i_j},i}(q)=q'.
\]
Therefore, after replacing $s_j$ with $\theta_N(d_{i_j}-C_{i_j,i})$, $r(s)_k$ is unchanged here as well.
\end{proof}

\begin{remark}
\label{remark:replace}
Note that the proof of the previous claim has similarities with the argument that appears in step 3 of the proof of Proposition 5.5 of \cite{matrix} that justifies replacing a switch from $i_j$ to $i$ with $d_i$. In contrast, we replace a switch from $i_j$ to $i$ with a change from $i_j$ to $i$. The reason for this is that there may be terms that contain $x_i$ prior to the application of the switch from $i_j$ to $i$ due to the nonzero limits of coefficients of terms with multiple indices, which requires the usage of the $C_{i_j,i}$ operators. To address this difference, we analyze the $\mathcal{Q}_\mathfrak{r}^N$ operators in \Cref{sec:formal}.
\end{remark}

Since $c:H_1\rightarrow H_2$ is a bijection, with \Cref{claim:bijval}, $\sum_{s\in H_1} r(s)_k = \sum_{s\in H_2} r(s)_k$. Therefore,
\[
[1]\mathcal{D}^{\theta_N}_{\mathfrak{r}}(F(x_1,\ldots,x_N))=[1]\mathcal{Q}_{\mathfrak{r}}^N(F(x_1,\ldots,x_N))+R_1+R_2-R_3.
\]
However, $R_1$ and $R_3$ have order $N^{k(1-c)-1}$, while $R_2$ has order $N^{k(1-c)-(1-c)}$. Then, the order of $R=R_1+R_2-R_3$ is $N^{k(1-c)-1+\max(c,0)}$. We are done.
\end{proof}

\begin{remark}
The main difference between the regime $c<1$ that we consider and the regime $c=1$ arises from the order of the derivative operator. As implied by \Cref{claim:r2order}, when $c<1$, we can disregard the derivative operators entirely because they are of lower order.
\end{remark}

For the definition of $\mathcal{R}_{N-i+1,N-i+1}^{\lambda_i}$, see \pref{eq:multoperators}.

\begin{theorem}
\label{remainder4}
Suppose that $\lambda$ is a partition with $\ell(\lambda)=m$ and $|\lambda|=k$. Suppose $N\geq m$. For $1\leq i\leq N$, let $\vec{X}_i=(x_1,\ldots,x_i, 0,\ldots,0)$ be $\vec{x}$ with $x_j$ replaced by $0$ for $i+1\leq j\leq N$. Then, where $\mathfrak{r}$ is the sequence of indices from \Cref{remainder1},
\begin{align*}
& [1] \mathcal{Q}_{\mathfrak{r}}^N(F(x_1,\ldots,x_N)) = [1]N^{k(1-c)}\prod_{i=1}^m \mathcal{Q}_{i}^N\left(\frac{\partial_i F}{N^{1-c}}\right)^{\lambda_i}(1) \\
& = [1]\prod_{i=1}^m\left((N-i+1)^{\lambda_i(1-c)}\mathcal{R}_{N-i+1,N-i+1}^{\lambda_i}\left(\frac{\partial_{N-i+1} F(\vec{X}_{N-i+1})}{(N-i+1)^{1-c}}\right)\right)(1)+R
\end{align*}
for a homogeneous polynomial $R$ in the $c_F^\nu$ with degree $k$ and order $N^{k(1-c)-1}$.
\end{theorem}
\begin{proof}
By symmetry,
\[
[1]N^{k(1-c)}\prod_{i=1}^m \mathcal{Q}_{i}^N\left(\frac{\partial_i F}{N^{1-c}}\right)^{\lambda_i}(1)
=[1]N^{k(1-c)}\prod_{i=1}^m \mathcal{Q}_{N-i+1}^N\left(\frac{\partial_{N-i+1} F}{N^{1-c}}\right)^{\lambda_i}(1).
\]
Here, $\mathfrak{r}=\{i_j\}_{1\leq j\leq k}$ consists of $k$ indices, the first $\lambda_1$ being $N$, the next $\lambda_2$ being $N-1$, and so forth, until the last $\lambda_m$ are $N-m+1$. Let $T$ be the set of sequences $s$ with indices $\mathfrak{r}$ over $x_1,\ldots,x_N$ consisting of term multiplications and changes such that:
\begin{itemize}
    \item For $1\leq j\leq m-1$, after $\lambda_1+\cdots+\lambda_j$ operators, the next $\lambda_{j+1}$ operators cannot have changes from $N-j$ to $i$ and term multiplications which degree-alter $x_i$, $N-j+1\leq i\leq N$.
    \item $r(s)_k$ is $0$ or has degree $0$ in the $x_i$.
\end{itemize}

Suppose $0\leq d\leq m$. For a sequence $s$ of $k$ term multiplications and changes, let $r_d(s)_j$ for $0\leq j\leq k$ be defined as follows. Set $r_d(s)_0=1$, and for $1\leq j\leq k$, $j\not=\lambda_1+\cdots+\lambda_i$, $1\leq i\leq d$, let $r_d(s)_j=s_j(r_d(s)_{j-1})$. Also, for $1\leq i\leq d$, let \[r_d(s)_{\lambda_1+\cdots+\lambda_i}=(s_{\lambda_1+\cdots+\lambda_i}(r_d(s)_{\lambda_1+\cdots+\lambda_i-1}))|_{x_{N-i+1}=0}.\]
Note that $r_0(s)_j=r(s)_j$, $0\leq j\leq k$.

\begin{claim}
\label{claim:samedeg}
For $0\leq j\leq k$ and any sequence $s$ of $k$ term multiplications and changes, there exists a positive integer $n$ such that for $0\leq d\leq m$, $r_d(s)_j$ is homogeneous in $x_1, \ldots, x_N$ with degree $n$ or $r_d(s)_j=0$.
\end{claim}

\begin{proof}
We can prove this using induction from $j=0$ to $k$. If we apply a change or a term multiplication, then we remove a term or change the degree of each term by the same constant. Furthermore, if we set $x_{N-i+1}=0$ for some $i\in [d]$, then we remove terms.
\end{proof}

\begin{claim}
\label{claim:equalend}
For all $s\in T$, $r_d(s)_k$ is the same for $0\leq d\leq m$.
\end{claim}

\begin{proof}
To show this, we prove that $r_d(s)_k=r_{d-1}(s)_k$, $1\leq d\leq m$. Observe that $r_d(s)_j=r_{d-1}(s)_j$ for $1\leq j\leq \lambda_1+\cdots+\lambda_d-1$ and
\[
r_{d-1}(s)_k-r_d(s)_k = s_k\circ\cdots\circ s_{\lambda_1+\cdots+\lambda_d+1}(r_{d-1}(s)_{\lambda_1+\cdots+\lambda_d}-r_{d-1}(s)_{\lambda_1+\cdots+\lambda_d}|_{x_{N-d+1}=0}),
\]
where each term of $r_{d-1}(s)_{\lambda_1+\cdots+\lambda_d}-r_{d-1}(s)_{\lambda_1+\cdots+\lambda_d}|_{x_{N-d+1}=0}$ contains $x_{N-d+1}$. But, if $\lambda_1+\cdots+\lambda_d+1\leq j\leq k$, $i_j\not= N-d+1$ since the indices of $s$ are $\mathfrak{r}$ and $s_j$ is not a change from $i_j$ to $N-d+1$ since $s\in T$. Hence, $s_j$ cannot remove $x_{N-d+1}$ from a term that contains $x_{N-d+1}$ without causing the term to become $0$. Then, all nonzero terms of $r_{d-1}(s)_k-r_d(s)_k$ contain $x_{N-d+1}$. Since $s\in T$, $r_0(s)_k=r(s)_k$ has degree $0$ in the $x_i$, so $r_{d-1}(s)_k$ and $r_d(s)_k$, and hence $r_{d-1}(s)_k-r_d(s)_k$, have degree $0$ in the $x_i$ from \Cref{claim:samedeg}. But, all nonzero terms of $r_{d-1}(s)_k-r_d(s)_k$ contain $x_{N-d+1}$, and therefore, $r_{d-1}(s)_k-r_d(s)_k=0$.
\end{proof}
\begin{claim}
\[
[1]\prod_{i=1}^m \left((N-i+1)^{\lambda_i(1-c)}\mathcal{R}_{N-i+1,N-i+1}^{\lambda_i}\left(\frac{\partial_{N-i+1} F(\vec{X}_{N-i+1})}{(N-i+1)^{1-c}}\right)\right)(1) = \sum_{s\in T\cap T_{N,\theta_N}^2(\mathfrak{r})} r(s)_k.
\]
\end{claim}
\begin{proof} 
By the definition of $T$, the left hand side of the above expression is
\[
\sum_{s\in T} r_m(s)_k = \sum_{s\in T} r(s)_k = \sum_{s\in T\cap T_{N,\theta_N}^2(\mathfrak{r})} r(s)_k,
\]
since $r_m(s)_k=r(s)_k$ for $s\in T$ by \Cref{claim:equalend} and $r(s)_k\not=0$ only if $s\in T\cap T_{N,\theta_N}^2(\mathfrak{r})$. 
\end{proof}

Suppose that $T'$ is the set of $s\in T_{N,\theta_N}^2(\mathfrak{r})$ satisfying the \Cref{newremainder2} conditions. From the lemma,
\[
[1]N^{k(1-c)}\prod_{i=1}^m \mathcal{Q}_{N-i+1}^N\left(\frac{\partial_{N-i+1} F}{N^{1-c}}\right)^{\lambda_i}(1) = \sum_{s\in T'} r(s)_k + R_1
\]
for a homogeneous polynomial $R_1$ in the $c_F^\nu$ with degree $k$ and order $N^{k(1-c)-1}$.

\begin{claim}
The set $T'$ is a subset of $T$.
\end{claim}
\begin{proof}
It suffices to show that if $s\in T'$, for $d>\lambda_1+\cdots+\lambda_j$, $s_d$ is not a change from $i_d$ to $x_i$ or a term multiplication that degree-alters $x_i$, $N-j+1\leq i\leq N$. For the sake of contradiction, assume $s\in T'$ and $s_d$ contains $x_i$, $N-j+1\leq i\leq N$, for some $d>\lambda_1+\cdots+\lambda_j$. From the conditions of \Cref{newremainder2}, $x_i$ must be degree-altered by a term multiplication. After, $x_i$ will be in all terms of $r(s)_d$; note that in the term multiplication, we do not take the derivative with respect to $x_i$ because $d>\lambda_1+\cdots+\lambda_j$. But, $x_i$ cannot be removed from a term by an operator unless the operator converts the term to zero, a contradiction to $s\in T_{N,\theta_N}^2(\mathfrak{r})$.
\end{proof}

Hence, $T'\subset T\cap T_{N,\theta_N}^2(\mathfrak{r})$. Again, we consider the coefficient of 
\[
p=\prod_{i=1}^m c_{F}^{\nu_i}.
\]
Let $U$ and $U'$ be the set of $s$ in $T_{N,\theta_N}^2(\mathfrak{r})$ and $T'$, respectively, such that $r(s)_k=dp$, $d\not=0$. Then, $U\cap T$ will be the set of $s$ in $T\cap T_{N,\theta_N}^2(\mathfrak{r})$ such that $r(s)_k=dp$, $d\not=0$. 

Note that because $T'\subset T\cap T_{N,\theta_N}^2(\mathfrak{r})$, $U'\subset U\cap T$. Also, $U\cap T\backslash U'$ is the set of $s\in T\cap T^2_{N,\theta_N}\backslash T'$ such that $r(s)_k=dp$, $d\not=0$. However, the number of $s\in U\backslash U'$ with $l$ changes is $O(N^{l-1})$ from the proof of \Cref{remainder2} without derivatives and using changes instead of switches, as in the proof of \Cref{newremainder2}. Therefore, the number $s\in U\cap T\backslash U'$ with $l$ changes is $O(N^{l-1})$. Then, by following the previous argument, we get that the coefficient of $p$ from sequences $s\in T\cap T_{N,\theta_N}^2\backslash T'$ is $O(N^{(k-m)(1-c)-1})$. Thus,
\[
\sum_{s\in T\cap T_{N,\theta_N}^2(\mathfrak{r})} r(s)_k = \sum_{s\in T'} r(s)_k + R_2
\]
for a homogeneous polynomial $R_2$ in the $c_F^\nu$ with degree $k$ and order $N^{k(1-c)-1}$, completing the proof.
\end{proof}

\subsection{Proof of \texorpdfstring{\Cref{lln:generalization}}{}}
\label{subsec:llngeneralization}

Suppose $\mathfrak{r}$ is the sequence that is described in the statement of \Cref{remainder1}. The result implies that as $N\rightarrow\infty$,
\[ 
\lim_{N\rightarrow\infty} \left(\frac{1}{N^{m+|\lambda|(1-c)}}\sum_{l\in I_N(\lambda)}[1]\mathcal{D}^{\theta_N}_{l}(F_N(x_1,\ldots,x_N))\right) = \lim_{N\rightarrow\infty}\frac{1}{N^{|\lambda|(1-c)}}[1]\mathcal{D}^{\theta_N}_\mathfrak{r}(x_1,\ldots,x_N),
\]
using \pref{eq:remainderto0} with $k=|\lambda|(1-c)-1$ and $\epsilon=1$. Afterwards, from \Cref{remainder3}, \Cref{remainder4}, and \pref{eq:remainderto0},
\begin{align*}
& \lim_{N\rightarrow\infty}\frac{1}{N^{|\lambda|(1-c)}}[1]\mathcal{D}^{\theta_N}_\mathfrak{r}(x_1,\ldots,x_N) \\
& = \lim_{N\rightarrow\infty} [1]\prod_{i=1}^m\left( \mathcal{R}_{N-i+1,N-i+1}^{\lambda_i}\left(\frac{\partial_{N-i+1} F_N(\vec{X}_{N-i+1})}{(N-i+1)^{1-c}}\right)\right)(1).
\end{align*}

Recall the definition of $\vec{X}_i$ from the statement of \Cref{remainder4}. To evaluate this, note that because $F_N$ is symmetric,
\begin{equation}
\label{eq:switchderiv}
s_{1,N-i+1}\frac{\partial_{N-i+1} F_N(\vec{X}_{N-i+1})}{(N-i+1)^{1-c}} = \frac{\partial_1 F_N(\vec{X}_{N-i+1})}{(N-i+1)^{1-c}}
\end{equation}
and the coefficient of $x_1^dM_\nu(x_2,\ldots,x_{N-i+1})$ in \pref{eq:switchderiv} is $\frac{(d+1)c_{F_N}^{\nu+(d+1)}}{(N-i+1)^{1-c}}$ for $\nu\in P$ with $\ell(\nu)\leq N-i$. Here, we can have $\nu=0$. However,
\[
\lim_{N\rightarrow\infty} \frac{(d+1)c_{F_N}^{\nu+(d+1)}}{(N-i+1)^{1-c}} = (d+1)c_{\nu+(d+1)},
\]
and the limiting sequence of \pref{eq:switchderiv} with respect to $1$ for $1\leq i\leq m$ is $f=\{c^{d,\nu}\}_{d\geq 0, \nu\in P}$, where $c^{d,\nu}=(d+1)c_{\nu+(d+1)}$. From \Cref{finalvalue1} with $\lim_{N\rightarrow\infty} N^c\theta_N=\theta$,
\begin{align*}
\lim_{N\rightarrow\infty} [1]\prod_{i=1}^m\left( \mathcal{R}_{N-i+1,N-i+1}^{\lambda_i}\left(\frac{\partial_{N-i+1} F_N(\vec{X}_{N-i+1})}{(N-i+1)^{1-c}}\right)\right)(1) =\prod_{i=1}^m\left(\sum_{\substack{\pi\in NC(\lambda_i)}} \prod_{B\in\pi} c_{|B|}(f)\right).
\end{align*}
Note that we have used the fact that when $g=1$, $c_1(g)=1$ and $c_k(g)=0$ for $k\geq 2$.

Suppose that for $\nu\in P$ and $d\geq 1$, $N_d(\nu)$ of the components of $\nu$ are $d$. Then, if $\nu'+(d)=\nu$, $\ell(\nu)P(\nu')=N_d(\nu)P(\nu)$, and for $l\geq 1$,
\begin{align*}
    c_l(f) & = \theta^{l-1}\sum_{\substack{\nu'\in P, d\geq 0,\\ |\nu'|+d=l-1}} (-1)^{\ell(\nu')}P(\nu')(d+1)c_{\nu'+(d+1)} \\
    & = \theta^{l-1} \sum_{\nu\in P, |\nu|=l} (-1)^{\ell(\nu)-1}c_\nu\sum_{\substack{\nu'\in P, d\geq 1,\\ \nu'+(d)=\nu}}dP(\nu')\\
    & = \theta^{l-1} \sum_{\nu\in P, |\nu|=l} (-1)^{\ell(\nu)-1}\frac{P(\nu)c_\nu}{\ell(\nu)}\sum_{\substack{\nu'\in P, d\geq 1,\\ \nu'+(d)=\nu}}dN_d(\nu) \\
    & = \theta^{l-1} \sum_{\nu\in P, |\nu|=l} (-1)^{\ell(\nu)-1} \frac{|\nu|P(\nu)}{\ell(\nu)}c_\nu,
\end{align*}
as required.

\subsection{Example} For positive integers $k$, let $n_k$ be the number of distinct $\nu\in P$ with $|\nu|=k$. Suppose $\theta\in\mathbb{C}$ and $c$ is a real number with $c<1$. For $N\geq 1$, let $\theta_N=\frac{\theta}{N^c}$. Also, suppose
\[
F_N(x_1,\ldots,x_N)=\sum_{\nu\in P, \ell(\nu)\leq N} c_{F_N}^\nu M_{\nu}(x_1,\ldots,x_N)
\]
such that $c_{F_N}^0=0$ and
\[
c_{F_N}^\nu=\frac{(-1)^{\ell(\nu)-1}\ell(\nu)}{|\nu|P(\nu)\theta^{|\nu|-1}}N^{1-c}
\]
for all $\nu\in P^+$. From \Cref{lln:generalization}, for $\lambda\in P^+$,
\[
    \lim_{N\rightarrow\infty} \left(\frac{1}{N^{\ell(\lambda)+|\lambda|(1-c)}}\sum_{l\in I_N(\lambda)}[1]\mathcal{D}^{\theta_N}_l(F_N(x_1,\ldots,x_N))\right) \\
    =  \prod_{i=1}^{\ell(\lambda)}\left(\sum_{\pi\in NC(\lambda_i)} \prod_{B\in\pi} n_{|B|}\right),
\]
where $n_k$ is the number of permutations of size $k$ for $k\geq 1$.

\section{Probability Measures}
\label{sec:measures}

\subsection{Bessel generating functions}
Next, we define the multivariate Bessel function; the definition is based on \cite{dunklbessel}*{Section 6}, where the uniqueness of the $\text{Exp}$ function is proven.

\begin{definition}
\label{def:multivariate}
For $\theta\in\mathbb{C}$ and $a_1,\ldots,a_N,x_1,\ldots,x_N\in\mathbb{C}$ such that the $a_i$ are not all zero, the function $\text{Exp}_{(a_1,\ldots,a_N)}(x_1,\ldots,x_N;\theta):\mathbb{C}^N\times\mathbb{C}^N\times \mathbb{C}\rightarrow\mathbb{C}$ is the unique meromorphic function such that for all $(a_1,\ldots,a_N)\in\mathbb{C}^N$ and $\theta\in\mathbb{C}$, $\mathcal{D}_i^\theta\text{Exp}_{(a_1,\ldots,a_N)}(\cdot; \theta) = a_i\text{Exp}_{(a_1,\ldots,a_N)}(\cdot; \theta)$ for $1\leq i\leq N$ and $\text{Exp}_{(a_1,\ldots,a_N)}(0; \theta)=1$. The \textit{multivariate Bessel function} $B_{(a_1,\ldots,a_N)}(x_1,\ldots,x_N; \theta):\mathbb{C}^N\times\mathbb{C}^N\times\mathbb{C}\rightarrow\mathbb{C}$ is defined as 
\[
B_{(a_1,\ldots,a_N)}(x_1,\ldots,x_N; \theta)\triangleq \frac{1}{N!}\sum_{\sigma\in \mathcal{S}_N} \text{Exp}_{(a_1,\ldots,a_N)}(x_{\sigma(1)},\ldots,x_{\sigma(N)}; \theta),
\]
where $\mathcal{S}_N$ is the set of permutations of $[N]$.
\end{definition}

The multivariate Bessel function equals one when $(x_1,\ldots,x_N)=0$ and is symmetric in both $(a_1,\ldots,a_N)$ and $(x_1,\ldots,x_N)$, see Proposition 6.8 of the paper \cite{dunklbessel}. Furthermore, although the multivariate Bessel function is meromorphic for all $\theta\in\mathbb{C}$, we only consider when $\real(\theta)\geq 0$ because the multivariate Bessel function is holomorphic in that domain, see Proposition 6.7 of the paper.

The following result is well-known and is true for all $\theta\in\mathbb{C}$; it is straightforward to prove using \Cref{def:multivariate}.

\begin{proposition}
\label{prop:operator}
For any symmetric polynomial $F(x_1, \ldots, x_N)$ and complex numbers $a_1,\ldots,$ $a_N$, $F(\mathcal{D}^{\theta}_1, \ldots, \mathcal{D}^{\theta}_N)B_{(a_1, \ldots, a_N)}(\cdot; \theta) = F(a_1, \ldots, a_N)B_{(a_1, \ldots, a_N)}(\cdot; \theta)$.
\end{proposition}

\begin{remark}
It is not immediate that the multivariate Bessel function is the only solution of the system of partial differential equations in \Cref{prop:operator}. This is established when $\theta$ has certain properties in \cite{dunklbessel}*{Section 5}.
\end{remark}

\begin{definition}
\label{def:BGF}
The \textit{Bessel generating function} of $\mu\in\mathcal{M}_N$ is defined as
\[
    G_\theta(\cdot; \mu) \triangleq \mathbb{E}_{(a_1,\ldots,a_N)\sim\mu}[B_{(a_1,\ldots, a_N)}(\cdot; \theta)].
\]
\end{definition}

The Bessel generating function is also studied in \cites{dunklbound,betaprocess,gaussianfluctuation,airy,matrix,rectangularmatrix}. A particular example that we consider is the Bessel generating function of the $\beta$-Hermite ensemble, see \Cref{sec:eigenvalue,sec:coeff}.

From \Cref{lemma:convergence}, if $\mu\in\mathcal{M}_N$ is exponentially decaying, then the Bessel generating function of $\mu$ converges in a neighborhood of the origin in $\mathbb{C}^N$. Additionally, from \Cref{prop:operator},
\[
\mathcal{P}^{\theta}_kB_{(a_1, \ldots, a_N)}(\cdot; \theta) = \left(\sum_{i=1}^Na_i^k\right)B_{(a_1, \ldots, a_N)}(\cdot; \theta),
\]
which leads to the following proposition.

\begin{proposition}[\cite{matrix}*{Proposition 2.11}]
\label{prop:expectedvalue}
For a positive integer $s$, let $k_1, \ldots, k_s$ be positive integers. Suppose $\mu\in\mathcal{M}_N$ is exponentially decaying. Then,
\[
\left(\prod_{i=1}^s \mathcal{P}^{\theta}_{k_i}\right) G_\theta(x_1,\ldots, x_N; \mu)\bigg|_{x_i=0, 1\leq i\leq N} = \mathbb{E}_{(a_1,\ldots,a_N)\sim \mu}\left(\prod_{i=1}^s\left(\sum_{j=1}^N a_j^{k_i}\right)\right).
\]
\end{proposition}

\begin{remark}
\label{remark:thetazero}
Note that the paper \cite{matrix} proves \Cref{lemma:convergence} and \Cref{prop:expectedvalue} for when $\theta>0$ using the integral formulation of the multivariate Bessel function as the expectation of a function over a probability distribution over the Gelfand-Tsetlin patterns. For the regime $\real(\theta)\geq 0$, we can prove the results using the same method as the paper as well as the inequality
\begin{equation}
\label{eq:bessel}
\begin{split}
|\text{Exp}_{(a_1,\ldots,a_N)}(x_1,\ldots,x_N; \theta)| & \leq \sqrt{N!} \exp\left(\max_{\sigma\in \mathcal{S}_N}\left(\sum_{i=1}^N \real(a_ix_{\sigma(i)})\right)\right) \\
& \leq \sqrt{N!}\exp(|(a_1,\ldots,a_N)|\cdot|(x_1,\ldots,x_N)|)
\end{split}
\end{equation}
for all $a_1,\ldots,a_N,x_1,\ldots,x_N\in\mathbb{C}$ from \cite{dunklbound}*{Corollary 3.2}. The required bounds on the derivatives of the multivariate Bessel function for the proof of \Cref{prop:expectedvalue} follow from \cite{dunklbound}*{Lemma 3.5}.

Additionally, note that the integral formulation of the multivariate Bessel function for $\theta>0$ first appears in \cite{besselformula}; it is also discussed in \cites{crystalrandommat,betaprocess,rectangularmatrix}. Since we do not use the exact formula, we do not state it in this paper. 
\end{remark}

From \Cref{prop:expectedvalue}, for exponentially decaying $\mu\in\mathcal{M}_N$ and $\lambda\in P^+$,
\begin{equation}
\label{eq:llnev}
\begin{split}
\mathbb{E}_{\mu} \left(\prod_{i=1}^{\ell(\lambda)} p_{\lambda_i}^{N,c}\right) & =\mathbb{E}_{\mu}\left(\prod_{i=1}^{\ell(\lambda)}\left(\frac{1}{N}\sum_{j=1}^N \left(\frac{a_j}{N^{1-c}}\right)^{\lambda_i}\right)\right) \\
& = \frac{1}{N^{\ell(\lambda)+|\lambda|(1-c)}}\left(\prod_{i=1}^{\ell(\lambda)} \mathcal{P}^{\theta}_{\lambda_i}\right) G_\theta(x_1,\ldots, x_N; \mu)\bigg|_{x_i=0, 1\leq i\leq N}.
\end{split}
\end{equation}
A step for showing that $\{\mu_N\}_{N\geq 1}$ satisfy a $c$-LLN is to use \pref{eq:llnev} with $\mu=\mu_N$ and $\theta=\theta_N$, see \Cref{subsec:mainproof}.

\begin{proposition}[\cite{matrix}*{Lemma 5.2}]
\label{prop:equalformal}
Suppose that $F$ is a $(k+1)$-times continuously differentiable function in a neighborhood of $(0,\ldots,0)\in\mathbb{C}^N$, with Taylor series expansion 
\[ 
F(x_1,\ldots,x_N)=\sum_{\nu\in P, |\nu|\leq k, \ell(\nu)\leq N} c_F^\nu M_{\nu}(\vec{x}) + O(\norm{x}^{k+1}).
\]
Then, if \[\tilde{F}(x_1,\ldots,x_N)=\sum_{\nu\in P, |\nu|\leq k, \ell(\nu)\leq N} c_F^\nu M_{\nu}(x_1,\ldots,x_N),\]
for $\lambda=(\lambda_1, \ldots, \lambda_m)$ with $|\lambda|=k$,
\[\left(\prod_{i=1}^m\mathcal{P}^\theta_{\lambda_i}\right)\exp(F(x_1,\ldots,x_N))\bigg|_{x_i=0, 1\leq i\leq N} = \left(\prod_{i=1}^m\mathcal{P}^\theta_{\lambda_i}\right)\exp(\tilde{F}(x_1,\ldots,x_N))\bigg|_{x_i=0, 1\leq i\leq N}.\]
\end{proposition}

\subsection{Proof of \texorpdfstring{\Cref{lln}}{}}
\label{subsec:mainproof}

The following result can be proved using equations (5.8) and (5.9) of \cite{matrix}. However, we omit its proof, which is straightforward.

\begin{lemma}
\label{lemma:symmetricpoly}
Suppose $\tilde{F}(x_1,\ldots,x_N)$ is a symmetric polynomial such that $\tilde{F}(0,\ldots,0)=0$. Then, 
\[
\left(\prod_{j=1}^k\mathcal{D}^{\theta_N}_{i_j}\right)\exp(\tilde{F}(x_1,\ldots,x_N))
= \mathcal{D}^{\theta_N}_{\mathfrak{r}}(\tilde{F}(x_1,\ldots,x_N))\exp(\tilde{F}(x_1,\ldots,x_N))
\]
and
\[
\left(\prod_{j=1}^k\mathcal{D}^{\theta_N}_{i_j}\right)\exp(\tilde{F}(x_1,\ldots,x_N))\bigg|_{x_i=0,1\leq i\leq N} = [1]\mathcal{D}^{\theta_N}_{\mathfrak{r}}(\tilde{F}(x_1,\ldots,x_N)).
\]
\end{lemma}

Next, suppose $\lambda\in P^+$ has length $m$. Suppose $\mathfrak{r}=\{i_j\}_{1\leq j\leq |\lambda|}$ is a sequence where the first $\lambda_1$ values are $1$, the next $\lambda_2$ are $2$, and so forth, until the last $\lambda_m$ are $m$. For \Cref{lln}, we look at the limit of \pref{eq:llnev} for $\theta=\theta_N$ and $\mu=\mu_N$ as $N\rightarrow\infty$. For $N\geq 1$, we let $F_N(x_1,\ldots,x_N)$ be a function such that $\exp(F_N(x_1,\ldots,x_N))=G_{\theta_N}(x_1,\ldots,x_N; \mu_N)$; because the $\mu_N$ is exponentially decaying, by \Cref{lemma:convergence}, $G_{\theta_N}(\cdot; \mu_N)$ is holomorphic in a neighborhood of the origin and furthermore evaluates to one at the origin, and is therefore nonzero in a neighborhood of the origin as well. Hence, there exists $F_N$ that is also symmetric and holomorphic in a neighborhood of the origin. Furthermore, we can assume that $F_N(0,\ldots,0)$ equals $\ln(G_{\theta_N}(0,\ldots,0;\mu_N))=0$, rather than a multiple of $2\pi i$.

Suppose that $\tilde{F}_N(x_1,\ldots,x_N)$ is the polynomial for $F_N(x_1,\ldots,x_N)$ from \Cref{prop:equalformal} with $k=|\lambda|$. By \Cref{prop:equalformal}, \pref{eq:llnev}, and \Cref{lemma:symmetricpoly}, 
\begin{align*}
\mathbb{E}_{\mu_N}\left(\prod_{i=1}^m p_{\lambda_i}^{N,c}\right) & = \frac{1}{N^{m+|\lambda|(1-c)}}\left(\prod_{i=1}^m\mathcal{P}^{\theta_N}_{\lambda_i}\right)\exp(F_N(x_1,\ldots,x_N))\bigg|_{x_i=0, 1\leq i\leq N} \\
& = \frac{1}{N^{m+|\lambda|(1-c)}}\left(\prod_{i=1}^m\mathcal{P}^{\theta_N}_{\lambda_i}\right)\exp(\tilde{F}_N(x_1,\ldots,x_N))\bigg|_{x_i=0, 1\leq i\leq N} \\
& = \frac{1}{N^{m+|\lambda|(1-c)}} \sum_{l\in I_N(\lambda)}[1]\mathcal{D}^{\theta_N}_l(\tilde{F}_N(x_1,\ldots,x_N)).
\end{align*}
The symmetric polynomial $\tilde{F}_N(x_1,\ldots,x_N)$ is a symmetric formal series for $N\geq 1$, and with the \Cref{lln} conditions, \Cref{lln:generalization} can be used on $\{\tilde{F}_N(x_1,\ldots,x_N)\}_{N\geq 1}$. This completes the proof.

\section{Eigenvalue Distributions}
\label{sec:eigenvalue}

From \cite{AGZ}*{(2.5.3)}, for all $\beta>0$, the $\beta$-Hermite ensemble is the measure in $\mathcal{M}_N$ with probability density
\begin{equation}
\label{eq:betahermite}
    d_{N, \beta}(x_1, \ldots, x_N) = C_{N, \beta}\prod_{1\leq i<j\leq N} |x_i-x_j|^\beta\prod_{i=1}^N e^{-\frac{\beta x_i^2}{4}}
\end{equation}
over $\mathbb{R}^n$, where
\[
C_{N,\beta} = \frac{(2\pi)^{-\frac{N}{2}}}{N!}\left(\frac{\beta}{2}\right)^{\frac{\beta N(N-1)}{4}+\frac{N}{2}}\Gamma\left(\frac{\beta}{2}\right)^N\left(\prod_{i=1}^N\Gamma\left(\frac{i\beta}{2}\right)\right)^{-1}.
\]
For $\beta=1,2,$ and $4$, $d_{N,\beta}$ is the probability density of the unordered eigenvalues of the GUE, GOE, and GSE, respectively. Also, from \cite{betaensembles}, there exist random tridiagonal matrices with eigenvalue distribution being the $\beta$-Hermitian ensemble for all $\beta>0$. 

Let $d^s_{N,\beta}$ be the density of the pushforward of the $\beta$-Hermite ensemble with respect to the function $f: \mathbb{R}^N\rightarrow \mathbb{R}^N, x\mapsto \sqrt{N}x$. We have that
\begin{equation}
\label{eq:betahermitescaled}
d^s_{N,\beta}(x_1,\ldots,x_N) = \frac{C_{N, \beta}}{N^{\frac{\beta N(N-1)}{4}+\frac{N}{2}}}\prod_{1\leq i<j\leq N} |x_i-x_j|^\beta\prod_{i=1}^N e^{-\frac{\beta x_i^2}{4N}}
\end{equation}
for all $x=(x_1,\ldots,x_N)\in\mathbb{R}^N$.

Suppose $\{\mu_N\}_{N\geq 1}$ is a sequence of probability measures such that $\mu_N$ is in $\mathcal{M}_N$ and has density $d^s_{N,\beta}$. After scaling by $N^{-k}$, the $2k$th moment of the $\beta$-Hermite ensemble converges to the $k$th Catalan number for $k\geq 0$, see \cite{eigenvalstat}*{Theorem 6.2.3}. \Cref{scaledlln} shows that this is the case and proves $0$-LLN satisfaction as an example of an application of \Cref{lln}.

\begin{lemma}[\cite{betaprocess}*{Corollary 3.7}]
\label{prop:scaledbessel}
Let $a_1,\ldots, a_N, y_1, \ldots, y_N, c$ be $2N+1$ arbitrary complex numbers and suppose $\theta>0$. Then, $B_{(ca_1,\ldots,ca_N)}(y_1,\ldots,y_N; \theta)=B_{(a_1,\ldots,a_N)}(cy_1,\\\ldots,cy_N; \theta)$.
\end{lemma}

\begin{lemma}
\label{prop:scaledbgf}
Suppose $\mu\in\mathcal{M}_N$. For $c>0$, let $\mu_c$ be the pushforward of $\mu$ with respect to the function $f:\mathbb{R}^N\rightarrow \mathbb{R}^N, x\mapsto cx$. Then,
$G_\theta(x_1,\ldots,x_N; \mu_c) = G_\theta(cx_1,\ldots,cx_N;\mu)$.
\end{lemma}
\begin{proof}
With \Cref{prop:scaledbessel},
\begin{align*}
    G_\theta(x_1,\ldots,x_N; \mu_c) & = \int_{a\in A} B_a(x_1,\ldots,x_N;\theta) d\mu_c(a_1,\ldots,a_N) \\
    & = \int_{a'\in A} B_{ca'}(x_1,\ldots,x_N;\theta)d\mu(a_1',\ldots, a_N') \\
    & = \int_{a'\in A} B_{a'}(cx_1,\ldots,cx_N;\theta) d\mu(a_1',\ldots,a_N') \\
    & = G_\theta(cx_1,\ldots,cx_N;\mu).
\end{align*}
This finishes the proof.
\end{proof}

\begin{proposition}
\label{scaledlln}
Suppose $\theta>0$. Consider the sequence $\{\mu_N\}_{N\geq 1}$ of probability measures such that for all positive integers $N$, $\mu_N\in\mathcal{M}_N$ and $\mu_N$ has probability density $d_{N,2\theta}^s$ from \pref{eq:betahermitescaled}. Then, $\{\mu_N\}_{N\geq 1}$ satisfy a $0$-LLN, with $m_{2k-1}=0$ and $m_{2k}=\frac{1}{k+1}\binom{2k}{k}$ for $k\geq 1$.
\end{proposition}
\begin{proof}
For $N\geq 1$, consider $\frac{\mu_N}{\sqrt{N}}$, which has density $d_{N,2\theta}$ in \pref{eq:betahermite}. It is easy to see that $\mu_N$ and $\frac{\mu_N}{\sqrt{N}}$ are exponentially decaying. From \cite{betaprocess}*{Proposition 4.2},
\begin{equation}
\label{eq:bgf}
    G_\theta\left(x_1,\ldots,x_N;\frac{\mu_N}{\sqrt{N}}\right)=\exp\left(\frac{1}{2\theta}\sum_{i=1}^N x_i^2\right).
\end{equation}
Then, from \Cref{prop:scaledbgf} and the previous equation,
\[
G_\theta(x_1,\ldots,x_N;\mu_N)=G_\theta\left(\sqrt{N}x_1, \ldots, \sqrt{N}x_N;\frac{\mu_N}{\sqrt{N}}\right)=\exp\left(\frac{N}{2\theta}\sum_{i=1}^N x_i^2\right).
\]
Afterwards, by \Cref{lln} with $c=0$ and $\theta_N=\theta$ for $N\geq 1$, $\{\mu_N\}_{N\geq 1}$ satisfy a $0$-LLN with $c_{(2)}=\frac{1}{2\theta}$ and $c_{\nu}=0$ for $\nu\in P^+$ such that $\nu\not=(2)$. For $k\geq 1$, $m_{2k-1}=0$ and $m_{2k}$ is the number of $\pi\in NC(2k)$ which have all blocks of size $2$, which is the $k$th Catalan number and equals $\frac{1}{k+1}\binom{2k}{k}$.
\end{proof}

From \Cref{scaledlln}, if $(a_1,\ldots,a_N)$ is distributed with density $d_{N,\beta}(x_1,\ldots,x_N)$, for a positive integer $s$ and positive integers $k_i$, $1\leq i\leq s$, 
\[
\lim_{N\rightarrow\infty} \mathbb{E}_{\mu_N}\left(\prod_{i=1}^s \left(\frac{1}{N}\sum_{i=1}^N\left(\frac{a_i}{\sqrt{N}}\right)^{k_i}\right)\right)  = \prod_{i=1}^s m_{k_i},
\]
with the moments $m_k$ given in \Cref{scaledlln}. Thus, the distributions $d_{N,\beta}$ satisfy a $\frac{1}{2}$-LLN, although note that $\theta$ is constant rather than proportional to $N^{-\frac{1}{2}}$. Sometimes, as seen above, for sequences $\{\mu_N\}_{N\geq 1}$ of exponentially decaying probability measures, we can scale $\mu_N$ by an appropriate power of $N$ for $N\geq 1$ and use \Cref{prop:scaledbgf} to satisfy the conditions of \Cref{lln}. We state this idea rigorously in the following result.

\begin{corollary}
\label{cor:scale}
Suppose $\theta\in\mathbb{C}$ has nonnegative real part and $c$ is a real number such that $c<1$. Let $\{\theta_N\}_{N\geq 1}$ be a sequence of complex numbers with nonnegative real part such that $\lim_{N\rightarrow\infty} N^c\theta_N = \theta$. Let $\{\mu_N\}_{N\geq 1}$ be a sequence of probability measures such that for all $N\geq 1$, $\mu_N$ is in $\mathcal{M}_N$ and is exponentially decaying. Assume that for all $\nu\in P^+$, real numbers $\alpha_\nu$ and $c_\nu$ exist such that 
\[
\displaystyle\lim_{N\rightarrow\infty}\frac{1}{N^{1-\alpha_\nu}}\cdot\frac{\partial}{\partial x_{i_1}}\cdots\frac{\partial}{\partial x_{i_r}} \ln(G_{\theta_N}(x_1,\ldots,x_N;\mu_N))\bigg|_{x_i=0,1\leq i\leq N} = \frac{|\nu|!c_\nu}{P(\nu)}
\]
for all positive integers $i_1, \ldots, i_r$ such that $\sigma((i_1,\ldots,i_r))=\nu$. Let $\Delta=\inf_{\nu\in P^+} \left(\frac{\alpha_\nu - c}{|\nu|}\right)$. Then, $\{\mu_N\}_{N\geq 1}$ satisfies a $(c+\Delta)$-LLN and
\[
m_k = \sum_{\pi\in NC(k)} \prod_{B\in \pi}\theta^{|B|-1}\left(\sum_{\nu\in P, |\nu|=|B|}(-1)^{\ell(\nu)-1} \frac{|\nu|P(\nu)}{\ell(\nu)}c_\nu\mathbf{1}\left\{\Delta = \frac{\alpha_\nu-c}{|B|}\right\}\right)
\]
for all positive integers $k$.
\end{corollary}

\begin{proof}
For $D=N^\Delta$, let $\mu_N^D$ be the pushforward of $\mu_N$ with respect to the function $f:A\rightarrow A, x\mapsto Dx$. By \Cref{prop:scaledbgf}, for all $\nu\in P^+$,
\begin{align*}
&\displaystyle\lim_{N\rightarrow\infty}\frac{1}{N^{1-c}}\cdot\frac{\partial}{\partial x_{i_1}}\cdots\frac{\partial}{\partial x_{i_r}} \ln(G_{\theta_N}(x_1,\ldots,x_N;\mu_N^D))\bigg|_{x_i=0,1\leq i\leq N} \\
& = \lim_{N\rightarrow\infty} \frac{|\nu|!N^{\Delta|\nu| + 1 - \alpha_\nu}c_\nu}{P(\nu) N^{1-c}} = \frac{|\nu|!c_\nu\mathbf{1}\left\{\Delta = \frac{\alpha_\nu-c}{|\nu|}\right\}}{P(\nu)}
\end{align*}
for all positive integers $i_1, \ldots, i_r$ such that $\sigma((i_1,\ldots,i_r))=\nu$. Then, from \Cref{lln}, $\{\mu_N^D\}_{N\geq 1}$ satisfy a $c$-LLN with the moments $\{m_k\}_{k\geq 1}$. Then, $\{\mu_N\}_{N\geq 1}$ satisfy a $(c+\Delta)$-LLN with these moments. 
\end{proof}

\section{Polynomial Coefficients}

\label{sec:coeff}

When $[1]\mathcal{D}^\theta_\mathfrak{r}(F(x_1,\ldots,x_N))$ is expressed as a polynomial in the $c_F^\nu$ for $\nu\in P^+$, we expect the coefficients to be polynomials in $\theta$ and $N$. In this section, we show that this is the case and relate the leading terms of the polynomial coefficients with the main results. In fact, we use this characterization of the coefficients to prove \Cref{thm:equivalence}, which generalizes \Cref{lln}.

\subsection{Proof of \texorpdfstring{\Cref{thm:equivalence}}{}}
\label{subsec:proofequiv}
First, we state the following result, which is a straightforward implication of \Cref{lemma:symmetricpoly}.

\begin{corollary}
\label{cor:commutative}
Suppose $\mathfrak{r}_1$ and $\mathfrak{r}_2$ are sequences of positive integers. 
\begin{enumerate}
\item[(a)] If $\mathfrak{r}_2$ is a permutation of $\mathfrak{r}_1$, then $\mathcal{D}_{\mathfrak{r}_2}^\theta(F(x_1,\ldots,x_N)) = \mathcal{D}_{\mathfrak{r}_1}^\theta(F(x_1,\ldots,x_N))$ for $N\geq \max(\mathfrak{r}_1)$.
\item[(b)] If $\sigma(\mathfrak{r}_1)=\sigma(\mathfrak{r}_2)$, then $[1]\mathcal{D}_{\mathfrak{r}_2}^\theta(F(x_1,\ldots,x_N)) = [1]\mathcal{D}_{\mathfrak{r}_1}^\theta(F(x_1,\ldots,x_N))$ for $N\geq \max(\mathfrak{r}_1,\mathfrak{r}_2)$.
\end{enumerate}
\end{corollary}

\begin{proof}Suppose $\mathfrak{r}_1=\{i_j^1\}_{1\leq j\leq k}$ and $\mathfrak{r}_2=\{i_j^2\}_{1\leq j\leq k}$, where $k\geq 1$. Clearly, there is nothing to prove if $\mathfrak{r}_1$ and $\mathfrak{r}_2$ have different lengths. Statement (a) follows from \Cref{lemma:symmetricpoly}, because 
\[
\prod_{j=1}^k \mathcal{D}^\theta_{i_j^1} = \prod_{j=1}^k \mathcal{D}^\theta_{i_j^2}
\]
by the commutativity of the Dunkl operators. Statement (b) then follows from the fact that $\mathcal{D}_{\mathfrak{r}_2}^\theta(F(x_1,\ldots,x_N))$ is $\mathcal{D}_{\mathfrak{r}_1}^\theta(F(x_1,\ldots,x_N))$ with its variables permuted, so their constant terms must be equal when $N\geq \max(\mathfrak{r}_1,\mathfrak{r}_2)$.
\end{proof}

\begin{lemma}
\label{coefficients}
Suppose $\mathfrak{r}=\{i_j\}_{1\leq j\leq k}$ is a sequence of positive integers. For a term
\[
    p=\prod_{i=1}^m c_F^{\nu_i},
\]
there exists a polynomial $f(x, y)$ with rational coefficients such that for all complex numbers $\theta$ and $N\geq\max(\mathfrak{r})$, the coefficient of $p$ in $[1]\mathcal{D}^{\theta}_\mathfrak{r}(F(x_1,\ldots,x_N))$ is $f(\theta, N)$.
\end{lemma}

\begin{proof}
Suppose $N\geq\max(\mathfrak{r})$. Let $U_N$ be the set of sequences $s$ of derivatives, term multiplications, and switches with variables $x_1,\ldots,x_N$ and indices $\mathfrak{r}$. For $s\in U_N$, we consider $s$ as a function of $\theta$, $s(\theta)$. For $s\in U_N$, let $l(s)$ be the number of switches in $s$, with $0\leq l(s)\leq k$. Then, if $s\in U_N$, we have that for some polynomial $P(x_1,\ldots,x_N)$ and partitions $\nu_j'$, $1\leq j\leq r$,
\[
r(s(\theta))_k=P(x_1,\ldots,x_N)\prod_{j=1}^r c_F^{\nu_j'}\theta^{l(s)}
\]
for all $\theta$. This is because in $s(\theta)$, $\theta$ is multiplied once for each switch and the other operators do not depend on $\theta$.

Let $T_N$ be the set of $s\in U_N$ such that $P(x_1,\ldots,x_N)$ is nonzero and has degree $0$ in the $x_i$ and $\prod_{j=1}^r c_F^{\nu_j'}=p$. If $s\in T_N$ and $r(s(\theta))_k=d\theta^{l(s)}p$ for $d\in\mathbb{R}\backslash\{0\}$, let $C'(s)=d$; note that $C'(s)$ does not depend on $\theta$. The coefficient of $p$ in $[1]\mathcal{D}^{\theta}_\mathfrak{r}(F(x_1,\ldots,x_N))$ is $\sum_{s\in T_N} C'(s)\theta^{l(s)}$ for all $\theta$.

If the sum of the $|\nu_i|$ is not $k$, from \Cref{polynomial} the coefficient of $p$ is $0$. Suppose that the sum of the $|\nu_i|$ is $k$. Let $S$ be the set of elements in $\mathfrak{r}$. Also, for $s\in T_N$, let $X(s)$ be the set $x_i$ such that $i\notin\mathfrak{r}$ and $x_i$ appears in switches or term multiplications of $s$. Since each $x_i\in X(s)$ from term multiplications must be eliminated by at least one switch by \Cref{claim:switches}, and each other $x_i\in X(s)$ appears in at least one switch, $|X(s)|\leq k$ because there are at most $k$ switches. 

For each $A\subset\{x_1,\ldots,x_N\}\backslash S$, where the number of $s\in T_N$ such that $X(s)=A$ is finite, let
\[
    g(A)=\sum_{s\in T_N, X(s)=A} C'(s)\theta^{l(s)}=\sum_{i=0}^k\theta^i\sum_{\substack{s\in T_N, X(s)=A, \\ l(s)=i}} C'(s).
\]
For $A$ with $|A|=Q$ to exist, we must have $N\geq |S|+Q$. Suppose $Q$ is an integer, $0\leq Q\leq k$. By symmetry, $g(A)$ is the same for all $A\subset\{x_1, \ldots, x_N\}\backslash S$, such that $|A|=Q$. With this, let
\[
r^N_{Q,i}=\sum_{\substack{s\in T_N, X(s)=A, \\ l(s)=i}} C'(s)
\]
for all $A$ with $|A|=Q$.

\begin{claim}
For $0\leq i\leq k$, $r^N_{Q,i}$ is the same for all $N\geq \max(\max(\mathfrak{r}), |S|+Q)$. 
\end{claim}

\begin{proof}
Let $A$ be the set of $x_i$ for the $Q$ smallest positive integers $i$ such that $x_i\notin S$. We see that $A$ is the same for all $N$. Suppose $T_N^i$ is the set of $s\in T_N$ such that $X(s)=A$ and $l(s)=i$ for $0\leq i\leq k$. 

The possible derivatives of $s\in T_N^i$ are derivatives with respect to $x_i$ for $i\in\mathfrak{r}$, and the possible switches of $s\in T_N^i$ are switches from elements of $\mathfrak{r}$ to elements of $\mathfrak{r}\cup A$. Also, the possible term multiplications of $s\in T_N^i$ are derivatives with respect to $x_i$ for $i\in\mathfrak{r}$ of terms with all $x_j$ in $\mathfrak{r}\cup A$. Therefore, the possible derivatives, switches, and term multiplications of $s\in T_N^i$ do not depend on $N$. The number of switches is $i$, which does not depend on $N$ as well. From this, $T_N^i$ is the same for all $N\geq \max(\max(\mathfrak{r}), |S|+Q)$. However, the sum of $C'(s)$ for $s\in T_N^i$ is $r^N_{Q,i}$, and because $T_N^i$ does not depend on $N$, $r^N_{Q,i}$ does not depend on $N$ as well.
\end{proof}

Suppose $r_{Q,i}$ is the value of $r^N_{Q,i}$ for $N\geq \max(\max(\mathfrak{r}), |S|+Q)$. Then, let $p_Q(x)=\sum_{i=0}^k r_{Q,i}x^i$ for $0\leq Q\leq k$. We have that $g(A)=p_Q(\theta)$ for $A$ with $|A|=Q$. Also, note that $\binom{N-|S|}{Q}$ is the number of $A$ such that $|A|=Q$. With this, if
\[
    f(x, y)=\sum_{0\leq Q\leq k} p_Q(x)\binom{y-|S|}{Q},
\]
the coefficient of $p$ is $f(\theta, N)$ for all complex numbers $\theta$ and $N\geq \max(\mathfrak{r})$. Observe that in this formula, we use the identity $\binom{N-|S|}{Q}=0$ if $N<Q+|S|$. We are done.
\end{proof}

Suppose $k$ is a positive integer. Let $S_k$ be the set of all multisets $q$ of partitions with size at least $1$ such that the sum of $|\nu|$ for $\nu\in q$ is $k$. Additionally, let
\[
    \mathfrak{P}_k= \Bigg\{\prod_{\nu\in q} c_F^\nu \Bigg| q\in S_k\Bigg\},
\]
and if $p\in\mathfrak{P}_k$ equals $p=\prod_{\nu\in q} c_F^\nu$ for $q\in S_k$, let $\ell(p)=|q|$.

Using \Cref{coefficients}, for indices $\mathfrak{r}$ of length $k$ and $N\geq\max(\mathfrak{r})$, let the coefficient of $p$ in $[1]\mathcal{D}^\theta_\mathfrak{r}(F(x_1,\ldots,x_N))$ be $f_{p, \mathfrak{r}}(\theta, N)$ for $p$ in $\mathfrak{P}_k$. Then, for $N\geq\max(\mathfrak{r})$,
\begin{equation}
\label{eq:polyexpansion}
    [1]\mathcal{D}^\theta_\mathfrak{r}(F(x_1,\ldots,x_N)) = \sum_{p\in\mathfrak{P}_k} f_{p,\mathfrak{r}}(\theta, N)p.
\end{equation}
Also, from \Cref{order} with $c=0$ and $\theta_N=\theta$ for $N\geq 1$ for any $\theta\in\mathbb{C}$, $f_{p,\mathfrak{r}}(x, y)$ has degree at most $k-\ell(p)$ in $y$. Then, let the coefficient of $y^{k-\ell(p)}$ in $f_{p,\mathfrak{r}}(x, y)$ be $s_{p,\mathfrak{r}}(x)$, with $s_{p,\mathfrak{r}}(x)=0$ being possible.

\begin{lemma}
\label{cor:leadingorderterm} Suppose $k\geq 1$, $\mathfrak{r}=\{i_j\}_{1\leq j\leq k}$ is a sequence of positive integers, $\lambda=\sigma(\mathfrak{r})$, and $m=\ell(\lambda)$. Then,
\[
    \sum_{p\in \mathfrak{P}_k} s_{p,\mathfrak{r}}(x)p = \prod_{i=1}^{\ell(\lambda)}\left(\sum_{\pi\in NC(\lambda_i)} \prod_{B\in\pi}x^{|B|-1}\left(\sum_{\nu\in P, |\nu|= |B|}(-1)^{\ell(\nu)-1}\frac{|\nu|P(\nu)}{\ell(\nu)}c_F^\nu\right)\right).
\]
\end{lemma}

\begin{proof}[Proof of \Cref{cor:leadingorderterm}]
Consider the indices $\mathfrak{r}'$ of length $k$, where the first $\lambda_1$ indices are $1$, the next $\lambda_2$ are $2$, and so forth, until the last $\lambda_m$ indices are $m$. From \Cref{cor:commutative}, $[1]\mathcal{D}^\theta_{\mathfrak{r}}(F(x_1,\ldots,x_N))=[1]\mathcal{D}^\theta_{\mathfrak{r}'}(F(x_1,\ldots,x_N))$ for all $N\geq\max(\mathfrak{r})$. Therefore, $f_{p,\mathfrak{r}}(x, y)=f_{p,\mathfrak{r}'}(x, y)$ and $s_{p,\mathfrak{r}}(x)=s_{p,\mathfrak{r}'}(x)$ for all $p\in\mathfrak{P}_k$. 

Let $c=0$ and $\theta$ be a complex number. Suppose $c_F^\nu$ is a complex number for all $\nu\in P^+$ and $F_N(x_1,\ldots,x_N)$ has $c_{F_N}^\nu = Nc_F^\nu$ for $\nu\in P$ with $1\leq \ell(\nu)\leq N$ and $c_{F_N}^\nu=0$ for $\nu\in P$ with $\ell(\nu)\geq N+1$. Then, from \Cref{prop:limitzero}, \Cref{remainder1}, and \Cref{lln:generalization} with $c=0$,
\begin{align*}
    & \lim_{N\rightarrow\infty}\frac{1}{N^k} [1]\mathcal{D}^\theta_{\mathfrak{r}'}(F_N(x_1,\ldots,x_N)) = \lim_{N\rightarrow\infty}\left(\frac{1}{N^{k+m}}\sum_{l\in I_N(\lambda)}[1]\mathcal{D}^\theta_{l}(F_N(x_1,\ldots,x_N))\right) \\
    & = \prod_{i=1}^m\left(\sum_{\pi\in NC(\lambda_i)} \prod_{B\in\pi}\theta^{|B|-1}\left(\sum_{\nu\in P, |\nu|= |B|}(-1)^{\ell(\nu)-1}\frac{|\nu|P(\nu)}{\ell(\nu)}c_F^\nu\right)\right).
\end{align*}
However, substituting $Nc_F^\nu$ for $c_{F_N}^\nu$ in \pref{eq:polyexpansion} gives
\[
    \frac{1}{N^k}[1]\mathcal{D}^\theta_{\mathfrak{r}'}(F_N(x_1,\ldots,x_N)) = \sum_{p\in\mathfrak{P}_k} \frac{f_{p,\mathfrak{r}'}(\theta, N)p}{N^{k-\ell(p)}},
\]
where the $c_F^\nu$ and thus $p\in\mathfrak{P}_k$ are constants. Since $f_{p,\mathfrak{r}'}(x, y)$ has degree at most $k-\ell(p)$ in $y$ from \Cref{order}, $\lim_{N\rightarrow\infty}\frac{f_{p,\mathfrak{r}'}(\theta, N)p}{N^{k-\ell(p)}}=s_{p,\mathfrak{r}'}(\theta)p=s_{p,\mathfrak{r}}(\theta)p$. Then,
\[
    \lim_{N\rightarrow\infty}\frac{1}{N^k}[1] \mathcal{D}^\theta_{\mathfrak{r}'}(F_N(x_1,\ldots,x_N))=\sum_{p\in\mathfrak{P}_k}s_{p,\mathfrak{r}}(\theta)p.
\]
Since $\theta$ and the $c_F^\nu$ can be any complex numbers, as a polynomial in $\theta$ and the $c_F^\nu$,
\[
\sum_{p\in\mathfrak{P}_k}s_{p,\mathfrak{r}}(\theta)p = \prod_{i=1}^m\left(\sum_{\pi\in NC(\lambda_i)} \prod_{B\in\pi}\theta^{|B|-1}\left(\sum_{\nu\in P, |\nu|= |B|}(-1)^{\ell(\nu)-1}\frac{|\nu|P(\nu)}{\ell(\nu)}c_F^\nu\right)\right),
\]
which gives the result.
\end{proof}

\begin{proof}[Proof of \Cref{thm:equivalence}]
First, we prove the analogue of \Cref{lln:generalization} in this setting. To show that this implies the result, we can follow the argument in \Cref{subsec:mainproof}. Suppose $\lambda\in P^+$ and replace $\ln(G_{\theta_N}(x_1,\ldots,x_N;\mu_N))$ with $\tilde{F}_N(x_1,\ldots,x_N)$ as we do previously, where $\tilde{F}_N$ also depends on $|\lambda|$. Then, for $\nu\in P^+$ with $|\nu|\leq |\lambda|$,
\[
c_\nu^N \triangleq \frac{P(\nu)}{|\nu|!N^{1-c}}\frac{\partial}{\partial x_{i_1}}\cdots\frac{\partial}{\partial x_{i_r}} \tilde{F}_N(x_1,\ldots,x_N))\bigg|_{x_i=0,1\leq i\leq N},
\]
for any positive integers $i_1, \ldots, i_r \leq N$ such that $\sigma((i_1,\ldots,i_r))=\nu$, $|c_\nu^N| = N^{o_N(1)}$ for all $\nu\in P^+$, and the goal is to show that
\[
    \lim_{N\rightarrow\infty} \left(\frac{1}{N^{\ell(\lambda)+|\lambda|(1-c)}}\sum_{l\in I_N(\lambda)}[1]\mathcal{D}^{\theta_N}_l(\tilde{F}_N(x_1,\ldots,x_N))\right) =  \prod_{i=1}^{\ell(\lambda)}\left(\sum_{\pi\in NC(\lambda_i)} \prod_{B\in\pi} c_{|B|}\right)
\]
for all $\lambda\in P^+$ if and only if \pref{eq:cumulantlimit} is true. Furthermore, observe that $c_{\tilde{F}_N}^\nu=N^{1-c}c_\nu^N$.

Suppose $k=|\lambda|$ and $\mathfrak{r}\in I_N(\lambda)$; note that we do not necessarily have that $\sigma(\mathfrak{r})=\lambda$. Because we are considering $\tilde{F}_N$ rather than $F$, let $\mathfrak{P}_k^N = \{\prod_{\nu\in q} c_{\tilde{F}_N}^\nu \mid q\in S_k\}$. Suppose $p\in\mathfrak{P}_k^N$ and $m=\ell(p)$.

Recall that the coefficient of $p$ in $[1]\mathcal{D}_\mathfrak{r}^{\theta_N}(\tilde{F}_N(x_1,\ldots,x_N))$ is $f_{p,\mathfrak{r}}(\theta_N,N)$. Consider the term $f_{p,\mathfrak{r}}(\theta_N,N)p$. We have that $|p|=O(N^{(1-c)m+o_N(1)})$ by the $|c_\nu^N|=N^{o_N(1)}$ condition. Let $g_{p,\mathfrak{r}}(\theta,N)$ denote the sum of the summands of $f_{p,\mathfrak{r}}(\theta, N)$ with degree $k-m$ in $\theta$; note that this is the maximal degree of $\theta$ because any sequence that contributes to the coefficient of $p$ must have exactly $m$ term multiplications. The values $T_\ell$ for $0\leq \ell\leq k-m-1$ are defined in the proof of \Cref{order}; the proof shows that $|T_\ell|=O(N^\ell)$, so
\[
|f_{p,\mathfrak{r}} - g_{p,\mathfrak{r}}|(\theta_N,N) \leq \sum_{\ell=0}^{k-m-1} |\theta_N|^{\ell} O(N^{\ell}) = \sum_{\ell=0}^{k-m-1} O(N^{\ell(1-c)}) = O(N^{(k-m-1)(1-c)}).
\]
It is then clear that 
\[
|f_{p,\mathfrak{r}} - g_{p,\mathfrak{r}}|(\theta_N,N) |p| = O(N^{(k-1)(1-c)+o_N(1)}).
\]
It follows that
\[
\left|\sum_{p\in\mathfrak{P}_k^N} (f_{p,\mathfrak{r}} - g_{p,\mathfrak{r}})(\theta_N,N)\right| = O(N^{(k-1)(1-c)+o_N(1)}).
\]
Therefore,
\[
[1]\mathcal{D}_{\mathfrak{r}}^{\theta_N}(\tilde{F}_N(x_1,\ldots,x_N)) = \sum_{p\in\mathfrak{P}_k^N} g_{p,\mathfrak{r}}(\theta_N, N) p + O(N^{(k-1)(1-c)+o_N(1)}).
\]

Let $h_{p,\mathfrak{r}}(\theta,N)$ be the term of $g_{p,\mathfrak{r}}(\theta,N)$ with degree $k-m$ in $N$; we have previously noted that $k-m$ is the maximal degree of $N$. Then, $h_{p,\mathfrak{r}}(\theta,N)$ is $(\theta N)^{k-m}$ multiplied by a rational number, which is possibly zero. Furthermore, 
\[
[1]\mathcal{D}_{\mathfrak{r}}^{\theta_N}(\tilde{F}_N(x_1,\ldots,x_N)) = \sum_{p\in\mathfrak{P}_k^N} h_{p,\mathfrak{r}}(\theta_N, N) p + O(N^{k(1-c)-1+o_N(1)}) + O(N^{(k-1)(1-c)+o_N(1)}).
\]

From \Cref{cor:leadingorderterm}, $s_{p,\mathfrak{r}}(x)$ is a multiple of $x^{k-m}$. That is, the coefficient of $N^{k-m}$ in $f_{p,\mathfrak{r}}(\theta,N)$ is a multiple of $\theta^{k-m}$. This makes sense because the only sequences that can contribute to the coefficient $N^{k-m}$ must consist of $k-m$ switches, each of which contributes a factor of $\theta$. This implies that $h_{p,\mathfrak{r}}(\theta,N) = s_{p,\mathfrak{r}}(\theta)N^{k-m}$. Hence,
\[
[1]\mathcal{D}_{\mathfrak{r}}^{\theta_N}(\tilde{F}_N(x_1,\ldots,x_N)) = \sum_{p\in\mathfrak{P}_k^N} s_{p,\mathfrak{r}}(\theta_N)N^{k-m} p + O(N^{k(1-c)-1+o_N(1)}) + O(N^{(k-1)(1-c)+o_N(1)}).
\]
We then have that
\begin{equation}
\label{eq:sumpoly}
\begin{split}
&\sum_{\mathfrak{r}\in I_N(\lambda)} [1]\mathcal{D}_{\mathfrak{r}}^{\theta_N}(\tilde{F}_N(x_1,\ldots,x_N)) \\
& = \sum_{\mathfrak{r}\in I_N(\lambda)} \sum_{p\in\mathfrak{P}_k^N} s_{p,\mathfrak{r}}(\theta_N)N^{k-\ell(p)} p + O(N^{k(1-c)-1}) + O(N^{(k-1)(1-c)}) \\
& = \sum_{p\in\mathfrak{P}_k^N}s_{p,\lambda}(\theta_N)N^{k-\ell(p) + \ell(\lambda)}p +  O(N^{\ell(\lambda)+k(1-c)-1+o_N(1)}) + O(N^{\ell(\lambda)+(k-1)(1-c)+o_N(1)}).
\end{split}
\end{equation}
We explain how we arrive at this expression. Similarly to the proof of \Cref{remainder1}, for $\mathfrak{r}\in I_N(\lambda)$ such that $\sigma(\mathfrak{r})\not=\lambda$, for all $p\in\mathfrak{P}_k^N$ we have that
\[
s_{p,\mathfrak{r}}(\theta_N)N^{k-\ell(p)}p = O(|\theta_N|^{k-\ell(p)} N^{k-\ell(p)})p = O(N^{k(1-c)+o_N(1)}).
\]
Since the number of such $\mathfrak{r}$ is $O(N^{\ell(\lambda)-1})$, this contributes to the $O(N^{\ell(\lambda)+k(1-c)-1+o_N(1)})$ remainder term.

Next, observe that $s_{p,\lambda}(\theta_N) = (1+o_N(1))s_{p,\lambda}(\theta)N^{-c(k-\ell(p))}$. Therefore,
\begin{align*}
\sum_{p\in\mathfrak{P}_k^N}s_{p,\lambda}(\theta_N)N^{k-\ell(p)+\ell(\lambda)} p&= (1+o_N(1))\sum_{p\in\mathfrak{P}_k^N} s_{p,\lambda}(\theta)N^{-c(k-\ell(p))}N^{k-\ell(p)+\ell(\lambda)}p \\
& = (1+o_N(1))N^{\ell(\lambda)}\sum_{p\in\mathfrak{P}_k^N} s_{p,\lambda}(\theta)N^{(1-c)(k-\ell(p))}p \\
& = (1+o_N(1))N^{\ell(\lambda)+k(1-c)} \sum_{p\in\mathfrak{P}_k^N} s_{p,\lambda}(\theta)\frac{p}{N^{\ell(p)(1-c)}}.
\end{align*}
Hence,
\begin{align*}
&\sum_{\mathfrak{r}\in I_N(\lambda)} [1]\mathcal{D}_{\mathfrak{r}}^{\theta_N}(\tilde{F}_N(x_1,\ldots,x_N)) =\\
& (1+o_N(1))N^{\ell(\lambda)+k(1-c)} \sum_{p\in\mathfrak{P}_k^N} s_{p,\lambda}(\theta)\frac{p}{N^{\ell(p)(1-c)}} +  O(N^{-1+o_N(1)}) + O(N^{-(1-c)+o_N(1)}).
\end{align*}
Afterwards, \Cref{cor:leadingorderterm} with $\lambda$ as $\mathfrak{r}$ gives that 
\begin{equation}
\label{eq:simplify}
\begin{split}
&\frac{1}{N^{\ell(\lambda)+k(1-c)}}\sum_{\mathfrak{r}\in I_N(\lambda)} [1]\mathcal{D}_{\mathfrak{r}}^{\theta_N}(\tilde{F}_N(x_1,\ldots,x_N)) \\
& = (1+o_N(1)) \sum_{p\in\mathfrak{P}_k^N} s_{p,\lambda}(\theta)\frac{p}{N^{\ell(p)(1-c)}} +  O(N^{-1+o_N(1)}) + O(N^{-(1-c)+o_N(1)}) \\
& = (1+o_N(1)) \prod_{i=1}^{\ell(\lambda)}\left(\sum_{\pi\in NC(\lambda_i)} \prod_{B\in\pi}\theta^{|B|-1}\left(\sum_{\nu\in P, |\nu|= |B|}(-1)^{\ell(\nu)-1}\frac{|\nu|P(\nu)}{\ell(\nu)}c_N^\nu\right)\right) + o_N(1),
\end{split}
\end{equation}
where we have used $c_{\tilde{F}_N}^\nu=N^{1-c} c_N^\nu$. 

We prove the forward direction. Assume that $\{\mu_N\}_{N\geq 1}$ satisfy a $c$-LLN with free cumulants $\{c_k\}_{k\geq 1}$. Using $\lambda=(k)$ in \pref{eq:simplify} gives that 
\[
 \lim_{N\rightarrow\infty}\sum_{\pi\in NC(k)} \prod_{B\in\pi}\theta^{|B|-1}\left(\sum_{\nu\in P, |\nu|= |B|}(-1)^{\ell(\nu)-1}\frac{|\nu|P(\nu)}{\ell(\nu)}c_N^\nu\right)= \sum_{\pi\in NC(k)}\prod_{B\in\pi} c_{|B|}.
\]
Afterwards, we can show \pref{eq:cumulantlimit} using induction on $k$ to prove the forward direction.

For the reverse direction, using \pref{eq:cumulantlimit} in \pref{eq:simplify} gives that
\[
\lim_{N\rightarrow\infty} \frac{1}{N^{\ell(\lambda)+k(1-c)}}\sum_{\mathfrak{r}\in I_N(\lambda)} [1]\mathcal{D}_{\mathfrak{r}}^{\theta_N}(\tilde{F}_N(x_1,\ldots,x_N)) = \prod_{i=1}^{\ell(\lambda)}\left(\sum_{\pi\in NC(\lambda_i)} \prod_{B\in\pi}c_{|B|}\right),
\]
which shows that $\{\mu_N\}_{N\geq 1}$ satisfy a $c$-LLN.
\end{proof}

Observe that in the previous proof, we have derived the following result.

\begin{lemma}
\label{lemma:remainder}
Suppose $k\geq 1$ and $\mathfrak{r}=\{i_j\}_{1\leq j\leq k}$. Then, for $N\geq \max(\mathfrak{r})$,
\[
[1]\mathcal{D}_{\mathfrak{r}}^{\theta}(F(x_1,\ldots,x_N)) = \sum_{p\in\mathfrak{P}_k} (s_{p,\mathfrak{r}}(\theta)N^{k-\ell(p)} + R_{1,\mathfrak{r},p}(\theta,N)+ R_{2,\mathfrak{r},p}(\theta,N))p,
\]
where for all $p\in\mathfrak{P}_k^N$,
\begin{itemize}
\item $s_{p,\mathfrak{r}}(\theta)=c_1\theta^{k-\ell(p)}$ for some rational number $c_1$.
\item The maximum degree of $\theta$ in the polynomial $R_{1,\mathfrak{r},p}(\theta,N)$ is at most $k-\ell(p)-1$ and the maximum degree of $N$ in the coefficient of $\theta^\ell$ is at most $\ell$ for $0\leq \ell\leq k-\ell(p)-1$.
\item The polynomial $R_{2,\mathfrak{r},p}(\theta,N)=c_2(N)\theta^{k-\ell(p)}$ for some polynomial $c_2$ and the maximum degree of $N$ in $c_2$ is at most $k-\ell(p)-1$.
\end{itemize}
\end{lemma}

\begin{remark}
By considering the regime $\theta N\rightarrow\gamma$, we have that for all $k\geq 1$ and $p\in\mathfrak{P}_k$,
\[
\gamma^{k-\ell(p)} \frac{s_{p,\mathfrak{r}}(\theta)}{\theta^{k-\ell(p)}} + \sum_{i=0}^{k-\ell(p)-1} \gamma^i \alpha_i
\]
is the limit of the coefficient of $p$ in \cite{matrix}*{(5.2)} as $N\rightarrow\infty$, where $\alpha_i$ is the coefficient of $\theta^iN^i$ in $R_{1,\mathfrak{r},p}(\theta,N)$ for $0\leq i\leq k-\ell(p)-1$.
\end{remark}

Furthermore, we have the following corollary, which is analogous to \Cref{cor:scale}.

\begin{corollary}
\label{cor:equivalencescale}
Suppose $\theta\in\mathbb{C}$ has nonnegative real part and $c$ is a real number such that $c<1$. Let $\{\theta_N\}_{N\geq 1}$ be a sequence of complex numbers with nonnegative real part such that $\lim_{N\rightarrow\infty}N^c\theta_N=\theta$. Let $\{\mu_N\}_{N\geq 1}$ be a sequence of probability measures such that for all $N\geq 1$, $\mu_N$ is in $\mathcal{M}_N$ and is exponentially decaying. For $N\geq 1$ and $\nu\in P^+$, suppose $\alpha_\nu$ is a real number and define
\[
c_\nu^N\triangleq\frac{P(\nu)}{|\nu|!N^{1-\alpha_\nu}}\cdot\frac{\partial}{\partial x_{i_1}}\cdots\frac{\partial}{\partial x_{i_r}} \ln(G_{\theta_N}(x_1,\ldots,x_N;\mu_N))\bigg|_{x_i=0,1\leq i\leq N}
\]
for any positive integers $i_1, \ldots, i_r \leq N$ such that $\sigma((i_1,\ldots,i_r))=\nu$. By symmetry, any choice of $i_1,\ldots,i_r$ results in the same derivative. Let $\Delta=\inf_{\nu\in P^+} \left(\frac{\alpha_\nu - c}{|\nu|}\right)$.

Assume that for all $\nu\in P^+$, $|c_\nu^N|= N^{o_N(1)}$. Then, $\{\mu_N\}_{N\geq 1}$ satisfies a $(c+\Delta)$-LLN with free cumulants $\{c_k\}_{k\geq 1}$ if and only if 
\begin{equation}
\label{eq:cumulantlimit2}
\lim_{N\rightarrow\infty} \theta^{k-1}\sum_{\nu\in P, |\nu|=k}(-1)^{\ell(\nu)-1} \frac{|\nu|P(\nu)}{\ell(\nu)} c_\nu^N \mathbf{1}\left\{\Delta = \frac{\alpha_\nu-c}{k}\right\} = c_k
\end{equation}
for all $k\geq 1$. 
\end{corollary}

\begin{proof}
We follow the method of the proof of \Cref{cor:scale}. For $D=N^\Delta$, let $\mu_N^D$ be the pushforward of $\mu_N$ with respect to the function $f:A\rightarrow A, x\mapsto Dx$. By \Cref{prop:scaledbgf}, for all $\nu\in P^+$,
\[
\frac{P(\nu)}{|\nu|!N^{1-c}}\cdot\frac{\partial}{\partial x_{i_1}}\cdots\frac{\partial}{\partial x_{i_r}} \ln(G_{\theta_N}(x_1,\ldots,x_N;\mu_N^D))\bigg|_{x_i=0,1\leq i\leq N}  = \frac{N^{\Delta|\nu|-\alpha_\nu}}{N^{-c}}c_\nu^N.
\]
for all positive integers $i_1, \ldots, i_r$ such that $\sigma((i_1,\ldots,i_r))=\nu$. 

Since $\Delta|\nu| -\alpha_\nu \leq -c$, it is clear that $\left|\frac{N^{\Delta|\nu|-\alpha_\nu}}{N^{-c}}c_\nu^N\right|=N^{o_N(1)}$ for $\nu\in P^+$. Then, by \Cref{thm:equivalence}, $\{\mu_N^D\}_{N\geq 1}$ satisfy a $c$-LLN with free cumulants $\{c_k\}_{k\geq 1}$ if and only if
\[
\lim_{N\rightarrow\infty} \theta^{k-1}\sum_{\nu\in P, |\nu|=k}(-1)^{\ell(\nu)-1} \frac{|\nu|P(\nu)}{\ell(\nu)} \frac{N^{\Delta|\nu|-\alpha_\nu}}{N^{-c}}c_\nu^N  = c_k
\]
for all $k\geq 1$, which is equivalent to \pref{eq:cumulantlimit2} being satisfied for all $k\geq 1$. Since $\{\mu_N^D\}_{N\geq 1}$ satisfy a $c$-LLN with free cumulants $\{c_k\}_{k\geq 1}$ if and only if $\{\mu_N\}_{N\geq 1}$ satisfy a $(c+\Delta)$-LLN with free cumulants $\{c_k\}_{k\geq 1}$, we are done.
\end{proof}

\subsection{Application to the \texorpdfstring{$\beta$-Hermite ensemble}{}}

We discuss an example of computing the coefficient of a term, where the coefficient satisfies the conditions of \Cref{lemma:remainder}.

\begin{lemma}
\label{lemma:coeffex}
Suppose $\mathfrak{r}$ consists of $2k$ equal positive integers and $\theta=1$. For $N\geq \max(\mathfrak{r})$, the coefficient of $(c_F^{(2)})^k$ in $[1]\mathcal{D}_\mathfrak{r}^\theta(F(x_1,\ldots,x_N))$ is
\[
\frac{1}{N} \frac{(2k)!}{k!}\sum_{j=0}^{k}2^j\binom{k}{j}\binom{N}{j+1}.
\]
\end{lemma}

\begin{proof}
Assume that $\mu_N$ has density $d_{N,2}$. Then, using \pref{eq:bgf} gives that $c_F^{(2)}=\frac{1}{2}$. Hence, the coefficient of $(c_F^{(2)})^k$ is 
\[
\frac{2^k}{N}[1]\mathcal{P}_{2k}^\theta G_\theta(x_1,\ldots,x_N;\mu_N) = \frac{2^k}{N} \mathbb{E}_{\mu_N}\left[\sum_{j=1}^N a_j^{2k}\right],
\]
see \Cref{prop:expectedvalue}. Then, using the calculation of $\mathbb{E}_{\mu_N}\left[\sum_{j=1}^N a_j^{2k}\right]$ in \cite{eulerchar} finishes the proof.
\end{proof}

\begin{remark}
Similarly, using \cite{random_matrix}*{(6.5.30)} gives an expression for the coefficient of $(c_F^{(2)})^k$ in $[1]\mathcal{D}_\mathfrak{r}^\theta(F(x_1,\ldots,x_N))$ for general $\mathfrak{r}$ when $\theta=1$.
\end{remark}

Observe that we have only computed the coefficient of $(c_F^{(2)})^k$ when $\mathfrak{r}$ consists of $2k$ equal positive integers for $\theta=1$. We do not exactly compute the general formula as a polynomial of $\theta$ and $N$. We show an estimate of this formula with error $O(N^{-2})$.

\begin{lemma}
\label{lemma:coeff}
Suppose $\mathfrak{r}$ consists of $2k$ equal positive integers. For $N\geq \text{max}(\mathfrak{r})$, the coefficient of $(c_F^{(2)})^k$ in $[1]\mathcal{D}_{\mathfrak{r}}^\theta(F(x_1,\ldots,x_N))$ is
\[
\frac{2^k\theta^k\binom{2k}{k}N^k}{k+1} +2^{k-1}\left(k+\sum_{l=1}^k \frac{1}{l+1}\binom{2l}{l}4^{k-l}\right)(\theta^{k-1}-\theta^k)N^{k-1} + O_\theta(N^{k-2}).
\]
\end{lemma}

\begin{proof}
First, observe that the coefficient of $N^k$ is a multiple of $\theta^k$ because the only method to contribute a factor of $N^k$ is to have $k$ switches and $k$ term multiplications, see \Cref{lemma:remainder}. By setting $\theta=1$ and using \Cref{lemma:coeffex}, we obtain the coefficient of $N^k$.

Similarly, the coefficient of $N^{k-1}$ is a linear combination of $\theta^{k-1}$ and $\theta^k$. Since we know the coefficient of $N^{k-1}$ when $\theta=1$, it suffices to find the coefficient of $\theta^{k-1}N^{k-1}$. 

Without loss of generality, assume that $\mathfrak{r}$ consists of $2k$ $1$s. The coefficient of $\theta^{k-1}N^{k-1}$ is contributed to by sequences of $k$ term multiplications, $k-1$ switches, and $1$ derivative.

Since we are computing the coefficient of $(c_F^{(2)})^k$, for each term multiplication we multiply by $2c_F^{(2)} x_1$. For simplicity, assume that $c_F^{(2)}=1$, so that we multiply by $2x_1$ for each term multiplication.

Furthermore, in order to achieve the factor $N^{k-1}$, each of the $k-1$ switches are a switch from $1$ to a distinct index $i\in [N]\backslash\{1\}$. If some of the indices are the same, then we would contribute a factor of at most $N^{k-2}$. The number of ways to choose the indices of the switches in this case is $(k-1)!\binom{N-1}{k-1}$; since we are only considering the coefficient of $\theta^{k-1}N^{k-1}$, the number of ways to choose the indices is equivalent to $N^{k-1}$. By symmetry, we can assume that the indices from the first switch to the last are $2, 3, \ldots, k$.

Additionally, if after the switch from $1$ to $i\geq 2$ there is a summand that contains $x_i$, note that it is impossible to remove $x_i$ without deleting each summand containing $x_i$, because each derivative is $\partial_1$ and the switches are to distinct indices. In this setting, the only summand resulting from applying a switch to a term which does not contain $x_i$ is the term with the exponent of $x_1$ decreased by one. This because the term multiplications each multiply by $2x_1$ and the switches have distinct indices, so there is no $x_i$ prior to applying the switch from $1$ to $i$. Hence, we can replace the switch from $1$ to $i$ with $d_1$. 

\begin{remark}Note that this argument is from \cite{matrix}, which uses it to justify replacing the switch from $i_j$ to $i$ with $d_{i_j}$ in step 3 of the proof of Proposition 5.5. We similarly replace the switch from $1$ to $i$ with $d_1$, which contrasts with the proof of \Cref{claim:bijval} where we would replace the switch with the change from $1$ to $i$, see \Cref{remark:replace}.
\end{remark}

Let $\zeta$ by some ordering of the $k$ term multiplications by $2x_1$, $k-1$ switches, and $1$ derivative $\partial_1$. The value of $r(s)_{2k}$ for each of the $N^{k-1} +O(N^{k-2})$ sequences that are equivalent to $\zeta$ is the same; let this common value be $v(\zeta)\theta^{k-1}$, where the factor of $\theta^{k-1}$ is from the $k-1$ switches. Then, the coefficient of $N^{k-1}\theta^{k-1}$ is $\sum_{\zeta} v(\zeta)$.

Since each of the switches is equivalent to $d_1$, we refer to them as $d_1$ for simplicity. First, observe that at each location of $\zeta$, the number of previous term multiplications must be at least the number of previous $d_1$s and derivatives, otherwise the output will be zero. Hence, if we do not distinguish between the $d_1$s and derivative, each $\zeta$ corresponds to a Dyck path of length $2k$.

Recall that a Dyck path of length $2k$ is a path from $(0,0)$ to $(k,k)$ which does not cross above the line $y=x$. Each Dyck path corresponds to $k$ orderings $\zeta$, since we can view each $(1,0)$ step as multiplying by $2x_1$ and each $(0,1)$ step as $d_1$ or $\partial_1$, and there are $k$ ways to choose which $(0,1)$ step is $\partial_1$. When we apply $d_1$ or $\partial_1$, we decrease the degree of $x_1$ by $1$. The difference is that when applying $\partial_1$, we also multiply by the degree of $x_1$. 

Suppose $p$ is a Dyck path from $(0,0)$ to $(k,k)$. Observe that the sum of $v(\zeta)$ for $\zeta$ corresponding to $p$ is $k2^{k-1}$ plus $2^k$ times the area between $p$ and $y=x$, since we multiply by $2x_1$ $k$ times and for each of the $k$ choices of the derivative, we multiply by the $x$-distance between the current location on $p$ and $y=x$. 

Let $\Delta(p)$ denote the area between $p$ and $y=x$. Then, $\sum_\zeta v(\zeta) = k2^{k-1} + 2^k\sum_p \Delta(p)$. However, from \cite{catalan},
\[
\sum_p \Delta(p) = \frac{1}{2}\sum_{l=1}^k \frac{1}{l+1}\binom{2l}{l}4^{k-l}.
\]
Hence,
\[
\sum_\zeta v(\zeta) = k2^{k-1} + 2^{k-1}\sum_{l=1}^k \frac{1}{l+1}\binom{2l}{l}4^{k-l},
\]
which is also the coefficient of $N^{k-1}\theta^{k-1}$. 

Note that the coefficient of $N^{k-1}$ of $(c_F^{(2)})^k$ is $0$ when $\theta=1$, see \Cref{lemma:coeffex}. This implies that the coefficient of $N^{k-1}\theta^k$ is the coefficient of $N^{k-1}\theta^{k-1}$ times negative one. 
\end{proof}

As mentioned earlier, the leading order term of the $2k$th moment of the $\beta$-Hermite ensemble $d_{N,\beta}$ is known to be $\frac{N^k}{k+1}\binom{2k}{k}$, see \Cref{scaledlln} or \cite{eigenvalstat}*{Theorem 6.2.3}. As an example of an application of this paper's results, we approximate the lower order terms. Particularly, using \Cref{lemma:coeff}, we can estimate the moments of the $\beta$-Hermite ensemble for all $\beta>0$ with error $O(N^{-2})$, where $\beta=2\theta$. Note that for the cases $\beta\in\{1,2,4\}$, we can arrive at these estimates using recursion, see \cite{eulerchar} for the GUE case and \cite{momentsrecursion} for the GOE and GSE cases. The following result is true for general $\beta$.

\begin{theorem}
\label{thm:momentestimate}
Suppose $\theta>0$ and $(a_1, \ldots,a_N)\sim d_{N,2\theta}$. Then,
\[
\mathbb{E}\left[\frac{1}{N}\sum_{i=1}^N (a_i)^{2k}\right] = \frac{\binom{2k}{k}N^k}{k+1} +\frac{1}{2}\left(k+\sum_{l=1}^k \frac{1}{l+1}\binom{2l}{l}4^{k-l}\right)(\theta^{-1}-1)N^{k-1} + O_\theta(N^{k-2}).
\]
\end{theorem}

\begin{proof}
From \pref{eq:bgf}, $c_F^{(2)}=\frac{1}{2\theta}$. Then, from \Cref{prop:expectedvalue} and \Cref{lemma:coeff}, the $2k$th moment is 
\[
\frac{1}{(2\theta)^k}\left(\frac{2^k\theta^k\binom{2k}{k}N^k}{k+1} +2^{k-1}\left(k+\sum_{l=1}^k \frac{1}{l+1}\binom{2l}{l}4^{k-l}\right)(\theta^{k-1}-\theta^k)N^{k-1} + O_\theta(N^{k-2})\right),
\]
which finishes the proof.
\end{proof}

\subsection{Generalization to the \texorpdfstring{$|\theta_N N|\rightarrow\infty$}{} regime}

We note that if $\lim_{N\rightarrow\infty}\theta_NN^c=\theta$ for $\theta\in\mathbb{C}$ nonzero and $c<1$, then $|\theta_NN|\rightarrow\infty$. Based on this observation, we consider the regime $|\theta_N N|\rightarrow\infty$. This regime is considered for the $\beta$-Hermite ensemble in \cite{betagaussian} and for the $\beta$-Laguerre ensemble in \cite{betalaguerre}. Since the Bessel generating function for the $\beta$-Hermite ensemble is known, we can apply the results of this section in this case, similarly to as we do in \Cref{sec:eigenvalue}. The setting that we consider is analogous to the setting that we describe in \Cref{subsec:setup}.

For a sequence $\{\mu_N\}_{N\geq 1}$ of probability measures such that $\mu_N\in\mathcal{M}_N$ for $N\geq 1$, we let the random variable $p_k^{N,\theta_N}$ be
\[
p_k^{N,\theta_N} \triangleq \frac{1}{N}\sum_{i=1}^N \left(\frac{a_i}{N\theta_N}\right)^k,
\]
where $(a_1,\ldots,a_N)\sim\mu_N$ for $N,k\geq 1$. Similarly, the moments $\{m_k\}_{k\geq 1}$ of $\{\mu_N\}_{N\geq 1}$ are given by $m_k=\lim_{N\rightarrow\infty}\mathbb{E}[p_k^{N,\theta_N}]$ for $k\geq 1$.

\begin{definition}
\label{def:llngeneral}
Suppose $\{\theta_N\}_{N\geq 1}$ is a sequence of nonzero complex numbers. A sequence $\{\mu_N\}_{N\geq 1}$ of probability measures such that $\mu_N\in\mathcal{M}_N$ for $N\geq 1$ \textit{satisfies a Law of Large Numbers} with scaling sequence $\{\theta_N\}_{N\geq 1}$ and moments $\{m_k\}_{k\geq 1}$ if
\[
\lim_{N\rightarrow\infty} \mathbb{E}_{(a_1,\ldots,a_N)\sim\mu_N}\left(\prod_{i=1}^s p_{k_i}^{N,\theta_N}\right)  = \prod_{i=1}^s m_{k_i}
\]
for all positive integers $s$ and $k_i$, $1\leq i\leq s$.
\end{definition}

The following result is an analogue of \Cref{thm:equivalence}, and implies \Cref{thm:equivalence} when $\theta\not=0$.

\begin{theorem}
\label{thm:equivalence2}
Let $\{\theta_N\}_{N\geq 1}$ be a sequence of nonzero complex numbers with nonnegative real part such that $\lim_{N\rightarrow\infty}|N\theta_N|=\infty$. Let $\{\mu_N\}_{N\geq 1}$ be a sequence of probability measures such that for all $N\geq 1$, $\mu_N$ is in $\mathcal{M}_N$ and is exponentially decaying. For $N\geq 1$ and $\nu\in P^+$, define
\[
c_\nu^N\triangleq\frac{P(\nu)}{|\nu|!\theta_N N}\cdot\frac{\partial}{\partial x_{i_1}}\cdots\frac{\partial}{\partial x_{i_r}} \ln(G_{\theta_N}(x_1,\ldots,x_N;\mu_N))\bigg|_{x_i=0,1\leq i\leq N}
\]
for any positive integers $i_1, \ldots, i_r \leq N$ such that $\sigma((i_1,\ldots,i_r))=\nu$. By symmetry, any choice of $i_1,\ldots,i_r$ results in the same derivative.

Assume that for all $\nu\in P^+$, $|c_\nu^N|= N^{o_N(1)}$ and $|c_\nu^N|=|\theta_N N|^{o_N(1)}$. Then, $\{\mu_N\}_{N\geq 1}$ satisfies a LLN with scaling sequence $\{\theta_N\}_{N\geq 1}$ and free cumulants $\{c_k\}_{k\geq 1}$ if and only if 
\[
\lim_{N\rightarrow\infty}\sum_{\nu\in P, |\nu|=k}(-1)^{\ell(\nu)-1} \frac{|\nu|P(\nu)}{\ell(\nu)} c_\nu^N = c_k
\]
for all $k\geq 1$. Recall that if the free cumulants are $\{c_k\}_{k\geq 1}$, then the moments are $m_k=\sum_{\pi\in NC(k)} \prod_{B\in \pi}c_{|B|}$ for $k\geq 1$.
\end{theorem}

\begin{proof}
We follow the proof of \Cref{thm:equivalence} in \Cref{subsec:proofequiv}; in particular, we use the notation from the proof. Instead of \pref{eq:sumpoly}, we get that
\begin{align*}
&\sum_{\mathfrak{r}\in I_N(\lambda)} [1]\mathcal{D}_{\mathfrak{r}}^{\theta_N}(\tilde{F}_N(x_1,\ldots,x_N)) \\
& = \sum_{p\in\mathfrak{P}_k^N}s_{p,\lambda}(\theta_N)N^{k-\ell(p) + \ell(\lambda)}p +  O(|\theta_NN|^k N^{\ell(\lambda)-1+o_N(1)}) + O(|\theta_NN|^{k-1+o_N(1)}N^{\ell(\lambda)}).
\end{align*}
Define the complex number $c_{p,\lambda}\triangleq\frac{s_{p,\lambda}(\theta)}{\theta^{k-\ell(p)}}$ for $p\in\mathfrak{P}_k^N$ or equivalently, $c_{p,\lambda}\triangleq s_{p,\lambda}(1)$; recall that $s_{p,\lambda}(\theta)$ is a complex number multiplied by $\theta^{k-\ell(p)}$, see \Cref{lemma:remainder}. Then,
\begin{align*}
&\sum_{\mathfrak{r}\in I_N(\lambda)} [1]\mathcal{D}_{\mathfrak{r}}^{\theta_N}(\tilde{F}_N(x_1,\ldots,x_N)) \\
& = \sum_{p\in\mathfrak{P}_k^N}c_{p,\lambda}(\theta_NN)^{k-\ell(p)}N^{\ell(\lambda)}p +  O(|\theta_NN|^k N^{\ell(\lambda)-1+o_N(1)}) + O(|\theta_NN|^{k-1+o_N(1)}N^{\ell(\lambda)}).
\end{align*}
Hence,
\begin{align*}
&\frac{1}{N^{\ell(\lambda)}(\theta_N N)^k}\sum_{\mathfrak{r}\in I_N(\lambda)} [1]\mathcal{D}_{\mathfrak{r}}^{\theta_N}(\tilde{F}_N(x_1,\ldots,x_N)) \\ &  = \sum_{p\in\mathfrak{P}_k^N} c_{p,\lambda}\frac{p}{(\theta_N N)^{\ell(p)}} +  O(N^{-1+o_N(1)}) + O(|\theta_NN|^{-1+o_N(1)}).
\end{align*}

After observing that $c_{\tilde{F}_N}^\nu = \theta_N N c_N^\nu$ and using \Cref{cor:leadingorderterm} with $x=1$, we get that
\begin{align*}
&\frac{1}{N^{\ell(\lambda)}(\theta_N N)^k}\sum_{\mathfrak{r}\in I_N(\lambda)} [1]\mathcal{D}_{\mathfrak{r}}^{\theta_N}(\tilde{F}_N(x_1,\ldots,x_N)) \\ &  = \prod_{i=1}^{\ell(\lambda)}\left(\sum_{\pi\in NC(\lambda_i)} \prod_{B\in\pi}\left(\sum_{\nu\in P, |\nu|= |B|}(-1)^{\ell(\nu)-1}\frac{|\nu|P(\nu)}{\ell(\nu)}c_N^\nu\right)\right) + o_N(1).
\end{align*}
Afterwards, we can repeat the proof of \Cref{thm:equivalence}.
\end{proof}

\bibliography{mybib.bib}

\end{document}